\documentclass{article}
\usepackage{amssymb}
\usepackage{amsfonts}
\usepackage{amsmath}
\usepackage{graphicx}
\usepackage{amsgen,color}

\providecommand{\U}[1]{\protect\rule{.1in}{.1in}}
\numberwithin{equation}{section}
\numberwithin{equation}{section}
\numberwithin{equation}{section}
\def\e{\varepsilon}

\newtheorem{theorem}{Theorem}[section]

\newtheorem{definition}[theorem]{Definition}

\newtheorem{lemma}[theorem]{Lemma}

\newtheorem{proposition}[theorem]{Proposition}

\newtheorem{remark}[theorem]{Remark}

\newenvironment{proof}[1][Proof]{\noindent\textbf{#1.} }{\ \rule{0.5em}{0.5em}}

\begin{document}

\title{Integral representation results in $BV\times L^p$.} 
\author{\textsc{Gra\c{c}a Carita}\thanks{%
CIMA-UE, Departamento de Matem\'{a}tica, Universidade de \'{E}vora, Rua Rom%
\~{a}o Ramalho, 59 7000 671 \'{E}vora, Portugal e-mail: gcarita@uevora.pt}, 
\textsc{\ Elvira Zappale}\thanks{%
D.I.In., Universita' degli Studi di Salerno, Via Giovanni Paolo II 132,
84084 Fisciano (SA) Italy e-mail:ezappale@unisa.it}}
\maketitle

\begin{abstract}
Integral representation results are obtained for the relaxation of some classes of energy functionals depending on two vector fields with different
behaviors, which may appear in the context of image decomposition and thermochemical equilibrium problems. 

Keywords: relaxation, convexity-quasiconvexity, functions of bounded variations.

MSC2000 classification: 49J45, 74Q05
\end{abstract}

\section{Introduction}

Minimization of energies depending on two independent vector fields have been introduced to model several phenomena. Namely, when $u$ is a Sobolev function  in $W^{1,q}, q >1,$ and $v$ is in $L^p$, the study of these energies (see \eqref{functional}) was motivated by the analysis of coherent thermochemical equilibria in a multiphase multicomponent system, with $\nabla u$ representing the elastic strain and $v$ the chemical composition of the material. In the theory of linear magnetostriction, the stored energy depends on the linearized strain and the direction of magnetization, we refer to \cite{FKP2, FKP1} and the references therein for more details. Moreover, when $p=q$ this type of energies is used to model Cosserat theory and bending phenomena in nonlinear  elasticity and also for the description of thin structures, see \cite{LDR} and \cite{BFM}. Here $v$ takes into account either Cosserat vectors or bending moments and $\nabla u$ is the elastic strain. When $u$ is a function of bounded variation, functionals similar to \eqref{functional} enter into image decomposition models, i.e., in order to denoise and restore a given image $f$, it is required to minimize a functional which is the sum of a 'total variation' term (i.e. a `norm' of $Du$) and a penalization term, i.e. a norm in a suitable functional space of $f-u-v$. Essentially $f$ can be decomposed into the sum of two components $u$ and $v$.  The first component (cartoon), $u$, is well structured and it describes the homogeneous objects which are present in the image, while the second component, $v$, contains the oscillating pattern (both textures and noise). We refer to \cite{ROF, Meyer, VO,VO2, AABFC} among the extensive literature in this field.

In order to cover a wide class of applications we start from the functional setting $W^{1,1}\times L^p$, with anisotropic energies with linear growth in the gradient variable $\nabla u$. Indeed, let $1<p\leq\infty$, for every $(u,v) \in W^{1,1}( \Omega;%
\mathbb{R}^d) \times L^{p}(\Omega;\mathbb{R}^{m}) $ 
define the functional 
\begin{equation}  \label{functional}
J(u,v) :=\int_{\Omega}f(x,u,v,\nabla u) dx
\end{equation}
where $f:\Omega\times\mathbb R^d\times\mathbb{R}^{m}\times\mathbb{R}%
^{d\times N}\rightarrow[ 0,+\infty) $ is a continuous function satisfying standard coercivity and growth conditions that will be precised below. We discuss separately the cases $1<p<\infty$ and $p=\infty.$ Thus we introduce for $1<p<\infty$ the functional%
\begin{equation}
\overline{J}_{p}( u,v) :=\inf\{ \underset{n\rightarrow +\infty
}{\lim\inf}J(u_{n},v_{n}) :u_{n}\in W^{1,1}(\Omega;\mathbb{R}^d),~v_{n}\in L^{p}(\Omega;\mathbb R^m)
,~u_{n}\rightarrow u\text{ in }L^{1},~v_{n}\rightharpoonup v\text{ in }%
L^{p}\} ,   \label{relaxedp}
\end{equation}
for any pair $(u,v) \in BV( \Omega;\mathbb R^d)
\times L^{p}( \Omega;\mathbb{R}^{m}) $ and, for $p=\infty$ the
functional%
\begin{equation}
\overline{J}_{\infty}( u,v) :=\inf\{ \underset{n\rightarrow
\infty}{\lim\inf}J(u_{n},v_{n}) :u_{n}\in W^{1,1}( \Omega;%
\mathbb R^d) ,v_{n}\in L^{\infty}( \Omega ;\mathbb{R}%
^{m}) ,u_{n}\rightarrow u\text{ in }L^{1},v_{n}\overset{\ast}{%
\rightharpoonup}v\text{ in }L^{\infty}\} ,   \label{relaxedinfty}
\end{equation}
for any pair $(u,v) \in BV(\Omega;\mathbb R^d)
\times L^{\infty}(\Omega;\mathbb{R}^{m}).$

Since bounded sequences $\{u_n\}$ in $W^{1,1}(\Omega;\mathbb R^d)$ converge
in $L^1$ to a $BV$ function $u$ and bounded sequences $\{v_n\}$ in $%
L^p(\Omega;\mathbb{R}^m)$ if $1<p\leq \infty$, 
weakly converge to a function $v \in L^p(\Omega;\mathbb{%
R}^m)$, (weakly $\ast$ in $L^\infty$), the relaxed functionals $\overline{J}%
_{p}$ and $\overline{J}_{\infty}$ will be composed by a Lebesgue part, a
jump part concentrated on the jump set of $u\in BV(\Omega;\mathbb{R}^d)$ and
a Cantor part, absolutely continuous with respect to the Cantor part of the
distributional gradient $Du$. On the other hand, as already emphasized in 
\cite{FKP2}, it is crucial to observe that $v$ is not pointwise defined on the jump
and the `Cantor' parts sets of $u$, thus specific features of the density $f$
will come into play to ensure a proper integral representation. The one of \eqref{relaxedp} is obtained in Theorem \ref{MainResultp} below, via the blow-up method introduced in \cite{FM2},
under the following hypotheses:
\begin{itemize}
	\item[$(H_0)$] $f(x,u,\cdot, \cdot)$ is convex-quasiconvex for every $%
	(x,u)\in \Omega \times \mathbb R^d$;
	\item[$(H_1)_{p}$] There exists a positive constant $C$ such that%
	\begin{equation*}
	\frac{1}{C}( \vert b\vert ^{p}+\vert \xi\vert
	) -C\leq f( x,u,b,\xi) \leq C( 1+\vert
	b\vert ^{p}+\vert \xi\vert) \hbox{ for every }(x,u,b,\xi) \in\Omega\times\mathbb R ^d\times \mathbb{R}%
	^{m}\times\mathbb{R}^{d\times N};
	\end{equation*}
	\item[$(H_{2}) _{p}$] For every compact set $K\subset\Omega
	\times\mathbb{R}^d$ there exists a continuous function $\omega _{K}:%
	\mathbb{R\rightarrow}[ 0,+\infty) $ with $\omega_{K}(
	0) =0$ such that
	\begin{itemize}
		\item[$(1)$] $\vert f( x,u,b,\xi) -f(
		x^{\prime},u^{\prime},b,\xi) \vert \leq\omega_{K}(\vert x-x^{\prime }\vert +\vert u-u^{\prime}\vert) ( 1+\vert b\vert ^{p}+\vert \xi\vert) $
		for every $(x,u,b,\xi) $ and $( x^{\prime},u^{\prime
		},b,\xi) \in K\times\mathbb{R}^{m}\times\mathbb{R}^{d\times N};$
		\item[$(2)$] Moreover, given $x_{0}\in\Omega$ and $\varepsilon>0$ there exists $\delta>0$
		such that if $\vert x-x_{0}\vert <\delta$ then 
		\begin{equation*}
		\displaystyle{f(
			x,u,b,\xi) -f( x_{0},u,b,\xi) \geq-\varepsilon(
			1+\vert b\vert ^{p}+\vert \xi\vert),} \hbox{ for every }(u,b,\xi) \in\mathbb{R}^d\times\mathbb{R}^{m}\times%
		\mathbb{R}^{d\times N};
		\end{equation*} 
		
	\end{itemize}
	\item[$( H_{3})_{p}$] There exist $c^{\prime}>0,~L>0,~0<\tau%
	\leq1$, such that%
	\begin{equation*}
	t>0,~\xi\in\mathbb{R}^{d\times N},\text{ with }t\vert \xi\vert
	+ t|b|^p>L\Longrightarrow\left\vert \frac{f( x,u,t^{\frac{1}{p}}b,t\xi) }{t}-f_p^{\infty
	}( x,u,b,\xi) \right\vert \leq c'\frac{|b|^{(1-\tau)p}+ |\xi|^{1-\tau}}{t^\tau} 
	\end{equation*}
	for every $(x,u) \in\Omega\times\mathbb R^d,$ where $f_p^{\infty }$ is the $(p,1)-$ recession function of $f$ defined for every $%
	(x,u,b,\xi)\in \Omega \times \mathbb{R}^d\times \mathbb{R}^{m}\times \mathbb R^{d\times N}$ as 
	\begin{equation}
	\displaystyle{f_p^{\infty }(x,u,b,\xi ):=\limsup_{t\rightarrow +\infty}\frac{%
			f(x,u,t^{\frac{1}{p}}b,t\xi )}{t}}.  \label{recessionp}
	\end{equation}
\end{itemize}

\begin{theorem}
\label{MainResultp} Let $J$ and ${\overline J}_p$ be given by \eqref{functional} and  \eqref{relaxedp}, respectively, with $f$ satisfying $(H_0)$, $(H_1)_p-(H_3)_p$ then%
\begin{equation}
\label{representationLp}
\overline{J}_p(u,v)= \int_\Omega f(x,u,v, \nabla u) dx +
\int_{J_u} K_p(x,0, u^+, u^-,\nu_u)d\mathcal{H}^{N-1} +
\int_\Omega f_p^\infty(x,u,0,\tfrac{dDu}{d|D^c u|}) d|D^c u|,
\end{equation}
for every $(u,v)\in BV(\Omega;\mathbb{R}^d)\times L^p(\Omega;\mathbb{R}^m)$,
where $K_p:\Omega \times\mathbb{R}%
^m\times\mathbb{R}^d\times\mathbb{R}^{d}\times S^{N-1}\rightarrow[
0,+\infty) $ is defined as
\begin{equation}
	K_p( x,b,c,d,\nu) :=\inf\left\{ \int_{Q_{\nu}}f_p^{\infty}(x,w( y),\eta( y) ,\nabla w(y))
	dy:w\in\mathcal{A}( c,d,\nu), \eta\in L^{\infty}( Q_{\nu};
	\mathbb{R}^{m}), \int_{Q_{\nu}}\eta\,dy=b\right\}, 
	\label{Kp}
\end{equation}
with 
\begin{equation}\label{calA}
\begin{array}{ll}
{\mathcal A}(c,d,\nu):=\{w \in W^{1,1}(Q_\nu;\mathbb R^d): &w(y)= c \hbox{ if } y\cdot \nu = \frac{1}{2}, w(y)= d  \hbox{ if } y \cdot \nu=-\frac{1}{2},\\
\\
&w \hbox{ is  }1-\text{periodic in }  \nu_1,\dots, \nu_{N-1}\,  \text{directions}\}.
\end{array}
\end{equation}

\end{theorem}

In order to provide an integral description of the functional $\overline{J}_{\infty }$, introduced in $(\ref{relaxedinfty})$, we prove Theorem \ref{MainResultinfty} replacing assumptions $(H_{1}) _{p}-( H_{3}) _{p}$ by the\ following ones:

\begin{itemize}
\item[$( H_{1}) _{\infty}$] Given $M>0,$ there exists $C_{M}>0$ 
such that, if $\vert b\vert \leq M$ then%
\begin{equation*}
\tfrac{1}{C_{M}}\vert \xi\vert -C_{M}\leq f(x,u,b,\xi)
\leq C_{M}( 1+\vert \xi\vert) , \hbox{ for every }(x,u,\xi) \in\Omega\times\mathbb{R}^d\times\mathbb{R%
}^{d\times N};
\end{equation*}

\item[$(H_2)_{\infty}$] For every $M>0,$ and for every compact set $%
K\subset\Omega\times\mathbb{R}^d$ there exists a continuous function 
$\omega_{M,K}(0) =0$ such that if $\vert b\vert \leq M$
then
\begin{equation*}
\vert f(x,u,b,\xi) -f(
x^{\prime},u^{\prime},b,\xi) \vert \leq\omega_{M,K}(
\vert x-x^{\prime }\vert +\vert u-u^{\prime}\vert
) ( 1+\vert \xi\vert) 
\end{equation*}
for every $(x,u,\xi) , (x^{\prime},u^{\prime},\xi)
\in K\times\mathbb{R}^{d\times N}.$  Moreover, given $M>0,$ $x_{0}\in\Omega,$ and $\varepsilon>0$ there exists $%
\delta>0$ such that if $\vert b\vert \leq M$ and $\vert
x-x_{0}\vert \leq\delta$ then%
\begin{equation*}
f( x,u,b,\xi) -f( x_{0},u,b,\xi) \geq-\varepsilon
( 1+\vert \xi\vert)  \hbox{ for every }( u,\xi) \in\mathbb{R}^d\times\mathbb{R}^{d\times N};
\end{equation*}

\item[$(H_{3})_{\infty }$] Given $M>0$, there exist $c_{M}^{\prime
}>0,~L>0,~0<\tau \leq 1$ such that 
\begin{equation*}
\vert b\vert \leq M,~t>0,~\xi \in \mathbb{R}^{d\times N},~\text{%
with }t\vert \xi \vert >L\Longrightarrow \left\vert \tfrac{f(
x,u,b,t\xi) }{t}-f^{\infty }( x,u,b,\xi) \right\vert \leq
c_M'\tfrac{\vert \xi \vert ^{1-\tau }}{t^{\tau }}
\end{equation*}%
for every $( x,u) \in \Omega \times \mathbb R^d,$
where $f^\infty(b,\xi)$  is the $(\infty-1)-$recession function, i.e. the `standard' recession function in the last variable, defined for every $%
(x,u,b,\xi)\in \Omega \times \mathbb R^d\times \mathbb{R}^{m}\times \mathbb R^{d\times N}$ as 
\begin{equation}
\label{recession}
\displaystyle{f^\infty(x,u,b,\xi):=\limsup_{t \to +\infty}\frac{f(x,u,b,t\xi)}{t}.}
\end{equation}
\end{itemize}

\begin{theorem}
\label{MainResultinfty} Let $J$ and ${\overline J}_\infty$ be given by \eqref{functional} and \eqref{relaxedinfty}, respectively, with $f$
satisfying $(H_0)$, $(H_1)_{\infty}-(H_3)_{\infty}$ then
$$
\overline{J}_\infty(u,v)= \int_\Omega f(x,u,v,\nabla u) dx +
\int_{J_u} K_\infty(x,0, u^+, u^-,\nu_u)d\mathcal{H}^{N-1} +
\int_\Omega f^\infty(x,u,0,\tfrac{dDu}{d|D^c u|}) d|D^c u|,
$$
for every $(u,v)\in BV(\Omega;\mathbb{R}^d)\times L^p(\Omega;\mathbb{R}^m)$,
where $K_\infty:\Omega \times\mathbb{R}%
^m\times\mathbb{R}^d\times\mathbb{R}^d\times S^{N-1}\rightarrow[
0,+\infty) $  is defined by 
\begin{equation}
K_\infty( x,b,c,d,\nu) :=\inf\left\{ \int_{Q_{\nu}}f^{\infty}(x,w( y),\eta( y) ,\nabla w(y))
dy:w\in\mathcal{A}( c,d,\nu), \eta\in L^{\infty}( Q_{\nu};
\mathbb{R}^{m}), \int_{Q_{\nu}}\eta\,dy=b\right\}, 
\label{Kinfty}
\end{equation}
where ${\cal A}(c,d,\nu)$ is as in \eqref{calA}.
\end{theorem}

It is worth to observe that assumption $(H_0)$ can be removed in Theorems \ref{MainResultp} and \ref{MainResultinfty}, thus replacing $f$ by its  convex-quasiconvex envelope in the above integral representations, and in \eqref{recessionp}, \eqref{Kp}, \eqref{recession} and \eqref{Kinfty}.  

We stress the fact that Theorems \ref{MainResultp} and \ref{MainResultinfty} generalize the result contained in \cite[Theorem 1.1]{RZCh} where an energy density $f(x,u,b,\xi):= W(x,u,\xi)+\varphi(x,u,v)$ has been considered. On the other hand we observe that the density $K_p$ (respectively $K_\infty$) in the latter case reduces to the density $K$ introduced in \cite{FR}, appearing in \cite[theorem 2.16]{FM2}, relative to $W$, and $f^\infty_p$ (respectively $f^\infty$) coincides with $W^\infty$, the latter being defined as $W^\infty(x,u,\xi):=\limsup_{t \to +\infty}\tfrac{W(x,u,t\xi)}{t}$.

It is worth to observe that if $f$ does not depend on $u$, the energy densities involved in the representations of $\overline{J}_p$ and $\overline{J}_\infty$ coincide, see Remark \ref{Remark2.17FMr}, \eqref{recessionp} and \eqref{recession}.

The paper is organized as follows. Section \ref{notpre} is devoted to notations, preliminaries and auxiliary results. Section \ref{propenerdens} contains the properties of the energy densities. The proofs of main theorems are in sections \ref{MainLp} and \ref{MainLinfty}.
The Appendix is devoted to remove $(H_0)$ in the proof of Theorems \ref{MainResultp} and \ref{MainResultinfty}.

\section{Notations and auxiliary results}\label{notpre} 

In this section we establish notations and present some
preliminary results on measure theory and functions of bounded variation
that will be useful through the paper. An auxiliary lemma, crucial to obtain the lower bound inequality is also proven.

All over the paper $\Omega$ will represent a bounded open set of $\mathbb R^N$ and ${\cal A}(\Omega)$ will be the family of all open subsets of $\Omega$. We denote by $Q:=(-1/2,1/2)^{N}$ the unit cube in $\mathbb{R}^N$
and if $\nu \in S^{N-1}$ and $(\nu _{1},\dots ,\nu _{N-1},\nu )$ is
an orthonormal basis of $\mathbb{R}^{N}$, $Q_{\nu }$ denotes the unit cube
centered at the origin with its faces parallel to $\nu _{1},\dots
,\nu _{N-1},\nu $. If $x\in \mathbb{R}^{N}$ and $\varepsilon >0$, we set $Q(x,\varepsilon
):=x+\varepsilon \,Q$ and $Q_{\nu }(x,\varepsilon ):=x+\varepsilon \,Q_{\nu }$, and $B(x_0,\varepsilon)\subset \mathbb R^N$ is the ball centered at $x_0$ with radius $\varepsilon.$ By $\mathcal{M}(\Omega)$ we represent the space of all signed Radon measures in $\Omega$ with
bounded total variation. By the Riesz Representation Theorem, $\mathcal{M}%
(\Omega)$ can be identified to the dual of the separable space $C_0(\Omega)$
of continuous functions on $\Omega$ vanishing on the boundary $\partial
\Omega$. The $N$-dimensional Lebesgue measure in $\mathbb{R}^N$ is
designated as $\mathcal{L}^N$, while $\mathcal{H}^{N-1}$ denotes the $(N-1)$-dimensional
Hausdorff measure. 
If $\mu \in \mathcal{M}(\Omega)$ and $\lambda \in \mathcal{M}(\Omega)$ is a
nonnegative Radon measure, we denote by $\frac{d\mu}{d\lambda}$ the
Radon-Nikod\'ym derivative of $\mu$ with respect to $\lambda$. By a
generalization of the Besicovitch Differentiation Theorem (see \cite[%
Proposition 2.2]{ADM}), it can be proved that there exists a Borel set $E
\subset \Omega$ such that $\lambda(E)=0$ and 
\begin{equation}
\label{ADMProposition2.2}
\frac{d\mu}{d\lambda}(x)=\lim_{\varepsilon \to 0^+} \frac{\mu(x+\varepsilon \, C)}{%
\lambda(x+\varepsilon \, C)}
\end{equation}
for all $x \in \mathrm{Supp }\, \mu \setminus E$ and any open convex set $C$
containing the origin. We recall that the exceptional set $E$ does not depend on $C$. The theorem below will be exploited in the sequel, besides not explicitly mentioned.

\begin{theorem}
\label{thm2.8FM2} If $\mu$ is a nonnegative Radon measure and if $f \in L^1_{%
\mathrm{loc}}(\mathbb{R}^N,\mu)$ then 
\begin{equation*}
\lim_{\varepsilon \to 0^+} \frac{1}{\mu(x+ \varepsilon C)}\int_{x+
\varepsilon C} | f(y) - f ( x ) | d\mu(y) =0 
\end{equation*}
for $\mu-$ a.e. $x\in \mathbb{R}^N$ and for every bounded, convex, open set 
$C$ containing the origin.
\end{theorem}

\begin{definition}
A function $u\in L^1(\Omega;\mathbb R^d)$ is said to be of \emph{bounded variation}, and we write $u\in BV(\Omega;\mathbb R^d)$, if all
its first distributional derivatives $D_{j}u_{i}$ belong to $\mathcal{M}%
(\Omega)$ for $1\leq i\leq d$ and $1\leq j\leq N$.
\end{definition}

The matrix-valued measure whose entries are $D_{j}u_{i}$ is denoted by $Du$
and $|Du|$ stands for its total variation. We observe that if $u\in
BV(\Omega;\mathbb R^d)$ then $u\mapsto|Du|(\Omega)$ is lower
semicontinuous in $BV(\Omega;\mathbb R^d )$ with respect to the $L_{%
\mathrm{loc}}^{1}(\Omega;\mathbb{R}^d)$ topology. A set $E \subset \Omega$ has finite perimeter in $\Omega$ if $\text{Per}(E;\Omega):=|D\chi_E|(\Omega)<+\infty$, where $\chi_E$ denotes the characteristic function of $E$. 

We briefly recall some facts about functions of bounded variation and we refer the reader to \cite{AFP} for details. 

\begin{definition}
Given $u\in BV( \Omega;\mathbb{R}^{d}) $ the \emph{approximate
upper}\textit{\ }\emph{limit }and the \emph{approximate lower limit} of each
component $u^{i}$, $i=1,\dots,d$, are defined by 
\begin{equation*}
( u^{i}) ^{+}( x) :=\inf\left\{ t\in\mathbb{R}%
:\,\lim_{\varepsilon\rightarrow0^{+}}\frac{\mathcal{L}^{N}( \{
y\in\Omega\cap Q( x,\varepsilon) :\,u^{i}(y)
>t\}) }{\varepsilon^{N}}=0\right\} 
\end{equation*}
and 
\begin{equation*}
( u^{i}) ^{-}(x) :=\sup\left\{ t\in\mathbb{R}%
:\,\lim_{\varepsilon\rightarrow0^{+}}\frac{\mathcal{L}^{N}( \{
y\in\Omega\cap Q( x,\varepsilon) :\,u^{i}(y)
<t\} ) }{\varepsilon^{N}}=0\right\} , 
\end{equation*}
respectively. The \emph{jump set}\textit{\ }of $u$ is given by 
\begin{equation*}
J_{u}:=\bigcup_{i=1}^{d}\{ x\in\Omega:\,( u^{i}) ^{-}(
x) <( u^{i}) ^{+}(x)\} . 
\end{equation*}
\end{definition}

\begin{proposition}
\label{thm2.3BBBF} If $u\in BV( \Omega;\mathbb R ^d) $ then

\begin{enumerate}
\item[i)] for $\mathcal{L}^{N}-$\hbox{a.e.} $x_{0}\in\Omega,  $  $\displaystyle{\lim_{\varepsilon\to 0^+}\frac{1}{\varepsilon}\left\{\frac {1}{{\varepsilon}^{N}}\int_{Q( x_{0},\varepsilon) }\vert
u(x)-u(x_{0})-\nabla u( x_{0}) \cdot(x-x_{0})\vert ^{\frac{N%
}{N-1}}dx\right\} ^{\frac{N-1}{N}}=0};$

\item[ii)] for $\mathcal{H}^{N-1}-$a.e. $x_{0}\in J_{u}$ there exist $%
u^{+}(x_0) ,$ $u^{-}(x_0) \in\mathbb{R}^d$ and $%
\nu \in S^{N-1}$ normal to $J_{u}$ at $x_{0},$ such that 
\begin{equation}
\lim_{\varepsilon\rightarrow0^{+}}\frac{1}{\varepsilon^{N}}%
\int_{Q_{\nu}^{+}( x_{0},\varepsilon) }\vert u(
x) -u^{+}(x_{0})\vert dx=0,\qquad\lim_{\varepsilon
\rightarrow0^{+}}\frac{1}{\varepsilon^{N}}\int_{Q_{\nu}^{-}(
x_{0},\varepsilon) }\vert u(x) -u^{-}(
x_{0})\vert dx=0,   \label{jump}
\end{equation}
\noindent where $Q_{\nu}^{+}( x_{0},\varepsilon) :=\{ y\in Q_{\nu}( x_{0},\varepsilon) :\langle x-x_{0},\nu\rangle
>0\} $ and $Q_{\nu}^{-}( x_{0},\varepsilon) :=\{ x\in
Q_{\nu}( x_{0},\varepsilon) :\left\langle
x-x_{0},\nu\right\rangle <0\} $;

\item[iii)] for $\mathcal{H}^{N-1}-$a.e. $x_{0}\in\Omega\backslash J_{u}$ \;\;
$
\lim_{\varepsilon\rightarrow0^{+}}\frac{1}{\mathcal{\varepsilon}^{N}}%
\int_{Q(x_{0},\varepsilon)}\vert u(x)-u(x_{0})\vert dx=0. 
$
\end{enumerate}
\end{proposition}

The next result, which will be exploited in the proof of the upper bound, can be found in \cite[Theorem 1, Chapter 4]{S}.

\begin{theorem}[Whitney's covering theorem] \label{Whitneythm}
Let $F \subset \mathbb R^N$ be a closed set. Then there exists a countable family of closed  cubes of the form $Q_i:=a_i+\delta Q_\nu$, such that the following hold:
\begin{itemize}
\item[i)] $\mathbb R^N \setminus F= \cup_{i=1}^\infty \overline{Q_i}$;
\item[ii)] the cubes $Q_i$ have mutually disjoint interiors;
\item[iii)] ${\rm diam}\, Q_i\leq {\rm dist}(Q_i,F)\leq 4\, {\rm diam}\,Q_i$.
\end{itemize}
\end{theorem}

The proof of the result below can be found in \cite[Lemma 3.1]{B}. With the aim of the applications below, we state it  as in \cite[Theorem 2.7]{BF}
\begin{proposition}
\label{Baldo}
Let $E$ be a subset of $\Omega$ such that ${\rm Per}(E;\Omega)<+\infty$. There exists a sequence of polyhedral sets $\{E_k\}$ (i.e. $E_k$ are bounded, strongly Lipschitz domains) with $\partial E_k= H_1\cup H_2\cup\dots \cup H_p$ where each $H_i$ is a closed subset of a hyperplane $\{x\in \mathbb R^N: x \cdot \nu_i=\alpha_i\}$ satisfying the following properties:
\begin{itemize}
\item[i)] ${\cal L}^N((E_k\cap \Omega)\setminus E) \cup(E\setminus (E_k\cap \Omega)))
\to 0 $ as $k \to +\infty$;
\item[ii)] ${\rm Per}(E_k;\Omega)\to {\rm Per}(E;\Omega)$ as $k \to +\infty$;
\item[iii)] ${\cal H}^{N-1}(\partial E_k\cap \partial \Omega)=0;$
\item[iv)]${\cal L}^N(E_k)={\cal L}^N(E)$.
\end{itemize}
\end{proposition}

The following Lemma that will be exploited for the lower bound inequality in Theorem \ref{MainResultp} is very similar to \cite[Lemma 3.1]{FM2}. 
\begin{lemma}\label{3.1FMr}
Le $f: \mathbb{R}^d\times \mathbb{R}^m\times \mathbb{R}^{d\times N} \to [0,+\infty)$ be a continuous function such that
\begin{equation}
\label{growthp}
\displaystyle{0\leq f(u,b,\xi)\leq C(1+ |b|^p+ |\xi|),}
\end{equation}
for some $C>0$ and for $(x,u,b,\xi)\in \Omega \times \mathbb R^d \times \mathbb R^m \times \mathbb R^{d \times N}$.
Let 
$
u_0(x):=\left\{
\begin{array}{lc}
c \hbox{ if } x_N >0\\
d \hbox{ if }  x_N \leq 0
\end{array}
\right.$, $\{w_n\} \subset W^{1,1}(Q;\mathbb R^d)$ and $\{v_n\} \subset L^p(Q;\mathbb R^m)$ be such that $w_n \to u_0$ in $L^1(Q;\mathbb R^d)$ and $v_n \rightharpoonup v$ in $L^p(Q;\mathbb R^m)$, with $v \in L^\infty(Q;\mathbb R^m)$ and $\int_Q  v dx = b$. If $\varrho$ is a mollifier, $\varrho_n(x):=(\tfrac{1}{\varepsilon_n})^N\varrho(\tfrac{x}{\varepsilon_n})$, and $\{\varepsilon_n\}$ is a sequence of real numbers such that $\varepsilon_n \to 0^+$, then there exist two sequences of functions $\{\xi_n\} \subset {\cal A}(c,d, \nu)$  and $\{\bar{v}_n\}\subset L^p(Q;\mathbb R^m)$  such that
\begin{align}
\xi_n= \varrho_{i(n)}\ast u_0 \hbox{ on } \partial Q, \;\;\;
\xi_n\to u_0 \hbox{ in } L^1(Q;\mathbb R^d), \;\;\;
\overline{v}_n\rightharpoonup v \hbox{ in }L^p(Q;\mathbb R^m), \notag \\
\int_Q \overline{v}_n dx = b, \;
\liminf_{n \to +\infty} \int_Q f ( w_n, v_n,\nabla w_n)dx\geq \limsup_{n\to +\infty}\int_Q f( \xi_n, {\bar v}_n, \nabla \xi_n)dx. \label{lemma2.8results}
\end{align}
\end{lemma}
\begin{proof}
Without loss of generality, assume that
\begin{equation*}
\liminf_{n \to +\infty} \int_Q f ( w_n, v_n,\nabla w_n)dx=\lim_{n\to+\infty} \int_Q f ( w_n, v_n,\nabla w_n)dx<+\infty.
\end{equation*}
Define $z_n(x):=(\varrho_n \ast u_0)(x)=\int_{B\left(x,\varepsilon_n\right)}\varrho_n(x-y)u_0(y)dy.$
Since $\varrho$ is a mollifier, we have $z_n(x+ e_i)=z_n(x)$ for every $i=1,\dots, N-1$,
$$
z_n=\left\{\begin{array}{ll} c &\hbox{ if }x_N >\varepsilon_n,\\
d &\hbox{ if }x_N \leq -\varepsilon_n, 
\end{array}
\right. \;\;\; \|\nabla z_n\|_{L^\infty(Q)}= O\left(\varepsilon_n\right), \;\;\; z_n \in {\cal A}(c,d, e_N).$$

For $j \in \mathbb N$, define $L_j:=\left\{x \in Q: {\rm dist}(x;\partial Q)<\tfrac{1}{j}\right\}.$
Take $j=2$, and divide $L_2$ into two layers $S^1_{2}, S^2_{2}$. It is clear that for every $n \in \mathbb N,$ there exists $S \in \{S^1_{2}, S^2_{2}\}$ such that
$\int_S (|\nabla w_n|+ |v_n|^p)dx \leq \frac{C}{2}$, where $C$ is the constant which uniformly bounds $\int_Q|\nabla w_n|dx $ and $\int_Q|v_n|^p dx$ in $Q$, since $w_n \to u_0$ in $L^1(Q;\mathbb R^d)$ and $v_n \rightharpoonup v$ in $L^p(Q;\mathbb R^m)$.
Since there are only two layers and infinitely many indices, we can conclude that one of the two layers, defined as $S_2:= \{x \in Q: \alpha_2 < {\rm dist}(x, \partial Q)< \beta_2\}$, for $0<\alpha_2 <\beta_2 <1$ satisfies
$$
\int_{S_2} |\nabla w_{n_2}|+ |v_{n_2}|^p dx \leq \frac{C}{2},
$$
for a subsequence $\{n_2\}$ of $\{n\}$.
Let $\eta_2$ be a smooth cut-off function $0\leq \eta_2\leq 1$, such that 
$\eta_2= 1$ in the complement of $\{x \in Q: {\rm dist}(x, \partial Q)< \beta_2\} $ and $\eta_2=0$ in $\{x \in Q: {\rm dist}(x, \partial Q)< \alpha_2\}$, and $\|\nabla(\eta_2)\|_{L^\infty}= O(\frac{1}{|S_2|})$. 
Clearly,
$$
\lim_{n_2 \to +\infty}\int_Q \eta_2 v_{n_2}dx= \int_Q \eta_2 v\,dx,
$$
since $\eta_2 \in L^\infty(Q;\mathbb R^m)$ and $v_{n_2} \rightharpoonup v$ in $L^p(Q;\mathbb{R}^m)$.
Also, for the same sequence $\{n_2\}$, we have
$$
\displaystyle{\lim_{n_2 \to +\infty} \left|\int_Q (v-\eta_2 v_{n_2})dx\right|=\left|\int_Q (1-\eta_2)v\,dx\right|\leq \|v\|_{L^\infty}\left|\int_Q (1-\eta_2) \,dx \right|.}
$$
Moreover, we can find a number $n(2)\in \{n_2\}$ large enough so that
$$
\frac{1}{|S_2|}\int_{S_2} |w_{n(2)}-z_{n(2)}|dx <\frac{1}{2},\,\,\,\,
\displaystyle{\frac{|\int_Q (v-\eta_2 v_{n(2)})dx|}{\left|1- \int_Q \eta_2 dx \right|} < \|v\|_{L^\infty} +1.}
$$
Next we divide $L_3$ into three layers $S^1_{3}, S^2_{3}, S^3_{3}$.
For each $n_2$ there exists $S\in \{S^1_{3}, S^2_{3}, S^3_{3}\}$, 
such that $\int_S |\nabla w_{n_2}|+ |v_{n_2}|^p dx \leq \frac{C}{3}.$  

Since there are only three layers with infinitely many indices, we conclude that one of the layers $S_3 \in \{S^1_{3}, S^2_{3}, S^3_{3}\}$
satisfies 
$$
\displaystyle{\int_{S_3}|\nabla w_{n_3}|+ |v_{n_3}|^p dx \leq \frac{C}{3},}
$$
for a subsequence $\{n_3\}$ of $\{n_2\}$.
Let $\eta_3$ be a smooth cut off function, $0 \leq \eta_3 <1$ , $\eta_3=1$ in the complement of $\{x \in Q: {\rm dist}(x, \partial Q)< \beta_3\}$ and $\eta_3=0$ in $\{x \in Q: {\rm dist}(x, \partial Q)< \alpha_3\}$, and $\|\nabla \eta_3\|_{L^\infty}= O\left(\frac{1}{|S_3|}\right)$, and
$$
\displaystyle{\lim_{n_3 \to +\infty} \left|\int_Q (v-\eta_3 v_{n_3})dx\right|=\left|\int_Q (1-\eta_3)v\,dx\right|\leq \|v\|_{L^\infty}\left|\int_Q (1-\eta_3) \,dx \right|.}
$$
The convergence of $w_{n_3} \to u_0$ in $L^1$, allows us to choose
$n(3)\in \{n_3\}$, $n(3)>n(2)$ large enough, such that
$$
\frac{1}{|S_3|}\int_{S_3}|w_{n(3)}- z_{n(3)}|dx \leq \frac{1}{3}
\hbox{ and } 
\displaystyle{\frac{|\int_Q (v-\eta_3 v_{n(3)})dx|}{\left|1- \int_Q \eta_3 dx \right|} < \|v\|_{L^\infty} +1.}
$$
Precisely, in this way, we construct the sequence $n(j)$ such that 
\begin{equation}\label{boundvj}
\displaystyle{\int_{S_j}|\nabla w_{n(j)}|+ |v_{n(j)}|^p dx \leq\tfrac{C}{j},}\,
\frac{1}{|S_j|}\int_{S_j}|w_{n(j)}- z_{n(j)}|dx \leq \tfrac{1}{j},
\hbox{ and }
\displaystyle{\frac{|\int_Q (v-\eta_j v_{n(j)})dx|}{\left|1- \int_Q \eta_j dx \right|} < \|v\|_{L^\infty} +1.}
\end{equation}
Let us define $\overline{w}_j(x):=(1-\eta_j(x))z_{n(j)}(x)+ \eta_j(x)w_{n(j)}(x),$ and 
$$
\overline{v}_j(x):=(1-\eta_j(x)) \frac{\int_{Q} (v-\eta_j v_{n(j)})dx}{1-\int_{Q} \eta_j dx}+ \eta_j(x)v_{n(j)}(x).
$$
Then
$$
\frac{1}{|Q|}\int_Q \overline{v}_jdx =b,\;\;
\|\overline{v}_j\|_{L^p(Q)}\leq C,
\hbox{ and }
\overline{w}_j\lfloor_{\partial Q}=\overline{w}_j\lfloor_{\partial Q}= u_0.
$$
In particular, $\overline{w}_j \to u_0 $ in $L^1(Q;\mathbb R^d)$ and 
$ \overline{v}_j \rightharpoonup b$ in $L^p(Q;\mathbb R^m)$.
The first convergence is trivial, the second one can be proven first observing that
it is enough to consider test functions $\varphi \in C_0(Q)$. Then the bounds in $\eqref{boundvj}$ entail that 
$$
\displaystyle{\lim_{j \to +\infty}\int_Q ((1-\eta_j) \tfrac{\int_{Q} (v-\eta_j v_{n(j)})dx}{1-\int_{Q} \eta_j dx}+ \eta_jv_{n(j)}- b) \varphi dx	=\lim_{j \to +\infty} \int_Q (\eta_j v_{n(j)}- v)\varphi dx.}
$$
Then
$$
\displaystyle{\lim_{j \to +\infty}\int_Q (\eta_j v_{n(j)}- v)\varphi dx= \lim_{j \to +\infty}\int_Q (v_{n(j)}- v)\varphi dx + \lim_{j \to +\infty} \int_Q (-1+\eta_j)v_{n(j)}\varphi dx=0,}
$$
The first limit  in the right hand side is $0$ since $v_{n(j)}\rightharpoonup v$ in $L^p(Q;\mathbb R^m)$ and the second is $0$ since $v_{n(j)}$ is $s$-equi-integrable for every $1\leq s<p$ and $\eta_j \to 1$. Hence we have
\begin{align*}
\lim_{n \to +\infty}\int_Q f(w_n, v_n, \nabla w_n)dx&=\lim_{j \to +\infty}\int_Q f(w_j, v_j, \nabla w_j)dx\\
&\geq\lim_{j \to +\infty}\int_Q f(\overline{w}_j, \overline{v}_j, \nabla \overline{w}_j)dx-\limsup_{j \to +\infty}\int_{\{x\in Q: {\rm dist}(x, \partial Q)<\alpha_j\}} f(\overline{w}_j, \overline{v}_j, \nabla \overline{w}_j)dx\\
&-\limsup_{j \to +\infty}\int_{S_j} f(\overline{w}_j, \overline{v}_j, \nabla \overline{w}_j)dx.
\end{align*}
Thus it results
\begin{align*}
&\lim_{n \to +\infty}\int_Q f(w_n, v_n, \nabla w_n)dx\geq\lim_{j \to +\infty}\int_Q f(\overline{w}_j, \overline{v}_j, \nabla \overline{w}_j)dx-\limsup_{j\to +\infty} \int_{L_j}C(1+ \|v\|_{L^\infty})^pdx\\
&-\limsup_{j \to +\infty}\int_{S_j} (1+|\nabla z_{n(j)}| + |\nabla  w_{n(j)}| + |v_{n(j)}|^p + \|v\|_{L^\infty}^p +\frac{1}{|S_j|} |w_{n(j)}- z_{n(j)}|)dx,
\end{align*}
where we have used the fact that $\nabla z_{n(j)}=0$ in $L_j$. 
Observing that, using co-area formula, $\int_{S_j}|\nabla z_{n(j)}|dx \to 0$ as $j \to +\infty$, we obtain the desired result.\end{proof}

\begin{remark}
\label{3.17FMrrem}
\begin{itemize}
\item[i)] For every $v \in L^p(\Omega;\mathbb R^m)$, under the same assumptions of Lemma \ref{3.1FMr} we can prove \eqref{lemma2.8results} without keeping the average.
\item[ii)]We observe that the same type of arguments can be exploited to prove a similar result for the $BV\times L^\infty$ case. Namely, if $f: \mathbb{R}^d\times \mathbb{R}^m\times \mathbb{R}^{d\times N} \to [0,+\infty)$ is a continuous function such that for every $b \in \mathbb R^m,$ with $|b|\leq M$ there exists a constant $C_M$ for which $\displaystyle{0\leq f(u,b,\xi)\leq C_M(1+ |\xi|),}$
for $(u,b,\xi)\in  \mathbb R^d \times\mathbb R^m \times \mathbb R^{d \times N},$
then \eqref{lemma2.8results} holds considering the sequence $\{v_n\}\subset L^\infty(Q;\mathbb R^m)$ and finding a correspondent sequence $\{\overline{v}_n\}\subset L^\infty(Q;\mathbb R^m)$ such that
$\overline{v}_n\overset{\ast}{\rightharpoonup} v \hbox{ in }L^\infty(Q;\mathbb R^m),$ and $\int_Q \overline{v}_n dx= b$ The main differences in the proof are the use of the above growth condition in place of \eqref{growthp}, and the fact  that $\{\overline{v}_j\}$ and $\{v_{n(j)}\}$ are uniformly bounded in $L^\infty$. 
\end{itemize}
\end{remark}
Next we recall the definition of Yosida transform that it will be useful in the proof of the upper bound.

\begin{definition}\label{defYosidap} For any function $f:\Omega\times\mathbb{R}^d\times\mathbb{R}^m\times\mathbb{R}^{d\times N}\longrightarrow\mathbb{R}$ 
let, for any $\lambda>0$, the \emph{Yosida transform} of $f$ be defined as 
$$f_\lambda(x,u,b,\xi):=\sup_{(x',u')\in\Omega\times\mathbb R ^d}\left\{f(x',u',b,\xi)-\lambda C(|x-x'|+|u-u'|)(1+|b|+|\xi|)\right\},$$ for any $(x,u,b,\xi)\in\Omega\times\mathbb{R}^d\times \mathbb R^m\times\mathbb{R}^{d\times N}$.
\end{definition}
The proof of next proposition follows along the lines \cite[Proposition 4.6]{AMT}

\begin{proposition}\label{propYosida}.
Let $f:\Omega\times\mathbb R ^d\times\mathbb{R}^m\times\mathbb{R}^{d\times N}\longrightarrow\mathbb{R}$ be 
such that $f(\cdot,\cdot,b,\xi)$ is continuous for any $(b,\xi)\in\mathbb{R}^m \times \mathbb{R}^{d\times N}$. Then the \emph{Yosida transform} of $f$ satisfies the following properties:

\begin{itemize}
\item[i)] $f_\lambda(x,u,b,\xi)\geq f(x,u,b,\xi)$ and $f_\lambda (x,u,b,\xi)$ decreases to $f(x,u,b,\xi)$ as $\lambda \to + \infty$.
\item[ii)] $f_\lambda(x,u,b,\xi) \geq f_\eta(x,u,b,\xi)$ if $\lambda \leq \eta$ for every $(x,u,b,\xi)\in \Omega \times \mathbb R^d \times \mathbb R^m \times \mathbb R^{d \times N}$.
\item[iii)] $|f_\lambda (x,u,b,\xi)-f_\lambda(x',u',b,\xi)| \leq \lambda (|x-x'|+|u-u'|)(1+ |\xi| + |b|)$ for every $(x,u,b,\xi), $ $(x',u',b,\xi)\in \Omega \times \mathbb R^d \times \mathbb R^m \times \mathbb R^{d \times N}$.
\item[iv)] The approximation is uniform on compact sets. Precisely, let $K$ be a compact subset of $\Omega \times \mathbb R^d$ and let $\delta>0$. There exists $\lambda>0 $ such that

$\displaystyle{f(x,u,b,\xi)\leq f_\lambda(x,u,b,\xi)\leq f(x,u,b,\xi) + \delta (1+ |b|+ |\xi|)},$
 for every $(x, u, b, \xi) \in K \times \mathbb R^m \times \mathbb R^{d \times N}$.
\end{itemize}
\end{proposition}

\section{Properties of the energy densities}
\label{propenerdens}

\subsection{Convex-quasiconvex functions}\label{removeCqf}

We start by recalling the notion of convex-quasiconvex function, presented
in \cite{FKP1} (see also \cite{LDR}, \cite{FKP2} and \cite{FFL}).

\begin{definition}
A Borel measurable function $h:\mathbb{R}^{m}\times\mathbb{R}^{d\times
N}\rightarrow\mathbb{R}$ is said to be convex-quasiconvex if, for every $%
(b,\xi)\in\mathbb{R}^{m}\times\mathbb{R}^{d\times N}$, there exists a
bounded open set $D$ of $\mathbb{R}^{N}$ such that 
\begin{equation*}
h(b,\xi)\leq\frac{1}{|D|}\int_{D}h(b+\eta(x),\xi+\nabla\varphi(x))\,dx, 
\end{equation*}
for every $\eta\in L^{\infty}(D;\mathbb{R}^{m})$, with ${\int_{D}\eta(x)%
\,dx=0}$, and for every $\varphi\in W_{0}^{1,\infty}(D;\mathbb R ^d )$.
\end{definition}

\begin{remark}\label{remplip}

\begin{enumerate}
\item[i)] If $h$ is convex-quasiconvex then
the inequality above is true for any bounded open set $%
D\subset\mathbb{R}^{N}.$
\item[ii)] A convex-quasiconvex function is separately convex.
\item[iii)] Throughout this paper we will work with functions $f$ defined in $%
\Omega\times\mathbb{R}^d\times\mathbb{R}^{m}\times\mathbb{R}^{d\times N}$
and when saying that $f$ is convex-quasiconvex, we consider the previous
definition with respect to the last two variables of $f.$
\item[iv)] If $f$ satisfies $(H_{1})_{p}$, Proposition 2.11 ii) in \cite{CRZ1} entails that $f$ is $(p,1)-$ \emph{Lipschitz continuous}, namely
there exists $\gamma>0$ such that 
\begin{equation}
\vert f(x,u,b,\xi) -f( x,u,b^{\prime},\xi^{\prime
}) \vert \leq\gamma(\vert \xi-\xi^{\prime}\vert
+( 1+\vert b\vert ^{p-1}+\vert b^{\prime}\vert
^{p-1}+\vert \xi\vert ^{\frac{1}{p^{\prime}}}+\vert
\xi^{\prime}\vert ^{\frac{1}{p^{\prime}}}) \vert
b-b^{\prime}\vert)   \label{p-lipschitzcontinuity}
\end{equation}
for every $b,~b^{\prime}\in\mathbb{R}^{m},$ $\xi,~\xi^{\prime}\in \mathbb{R}^{d\times N}$ and $(x,u) \in\Omega\times\mathbb R ^d $, where $p'$ is the conjugate exponent of $p$.
\item[v)] If $f$ satisfies
 $(H_1)_{\infty}$, \cite[Proposition 4]{RZ} guarantees that $f$ is $(\infty,1)-$\emph{Lipschitz continuous}, i.e. given $M>0$ there exists a constant $\beta(M)>0$ such that
\begin{equation}
\vert f( x,u,b,\xi) -f( x,u,b^{\prime},\xi^{\prime
}) \vert \leq\beta( 1+\vert \xi\vert +\vert
\xi^{\prime}\vert) \vert b-b^{\prime}\vert
+\beta\vert \xi-\xi^{\prime}\vert 
\label{infty-lipschitzcontinuity}
\end{equation}
for every $b,~b^{\prime}\in\mathbb{R}^{m},$ such that $\vert
b\vert \leq M$ and $\vert b^{\prime}\vert \leq M,$ for
every $\xi ,~\xi^{\prime}\in\mathbb{R}^{d\times N}$ and for every $(
x,u) \in\Omega\times\mathbb R ^d.$
\end{enumerate}
\end{remark}

\subsection{The recession functions}\label{recessionfunctions}

Let $f:\Omega \times \mathbb R ^d \times \mathbb{R}^m \times \mathbb{R}^{d
\times N}\to [0, +\infty[$, and let $f_p^\infty: \Omega \times \mathbb R ^d
\times \mathbb{R}^m \times \mathbb R ^{d \times N}\to [0, +\infty[$, be its $(p,1)-$ recession function, defined in \eqref{recessionp}. We observe that $f^\infty_p$ satisfies the following homogeneity property,
\begin{equation}
\label{p-1homogeneity}
\displaystyle{f^\infty_p(x,u,t^\frac{1}{p}b, t\xi)= t f^\infty_p(x,u, b,\xi) \; \hbox{ for every }t \in \mathbb R^+, x \in \Omega, u \in \mathbb R^d, b \in \mathbb R^m, \xi \in \mathbb R^{d \times N}.}
\end{equation}

Notice that, under growth condition $(H_1)_p$ on $f$, we could consider both $f^\infty_p$ and $f^\infty$ (i.e. $(\infty,1)-$ recession function of $f$ as in \eqref{recession}), and the latter one turns out to be independent on $b$, i.e. $f^\infty(x,u,b,\xi)=f^\infty(x,u,0,\xi)$ for every $(x,u,b,\xi)\in \Omega\times \mathbb R^d\times \mathbb R^m \times \mathbb R^{d\times N}$, provided $f$ is separately convex. 
Moreover, it results that in general
$f^\infty_p(x,u,b,\xi)\not= f^\infty(x,u,b,\xi)$ but the equality holds if $b=0$.

The following properties are an easy consequence of the definition of
$(p,1)-$ recession function and of properties $(H_{0})$, $(H_{1})_{p},~(H_{2})_{p}$,
when $1<p<\infty. $ 

\begin{proposition}
\label{proprecession} Let $f:\Omega \times \mathbb{R}%
^d\times \mathbb{R}^{m}\times \mathbb{R}^{d\times N}\rightarrow \lbrack
0,+\infty \lbrack $, and let $f^\infty_p$ defined by \eqref{recessionp}, provided $f$ satisfies $(H_{0})$, $%
(H_{1})_{p},~(H_{2})_{p}$, then

\begin{enumerate}
\item[i)] $f_p^{\infty }$ is convex-quasiconvex;

\item[ii)] there exists $C>0$ such that%
\begin{equation}
\tfrac{1}{C}( \vert b\vert ^{p}+\vert \xi \vert
) \leq f_p^{\infty }( x,u,b,\xi) \leq C(\vert
b\vert ^{p}+\vert \xi \vert);  \label{finftypgrowth}
\end{equation}

\item[iii)] for every $K\subset \subset \Omega \times \mathbb{R}^d$ there
exists a continuous function $\omega _{K}$ with $\omega _{K}( 0)
=0$ such that%
\begin{equation}
\vert f_p^{\infty }( x,u,b,\xi) -f_p^{\infty }( x^{\prime
},u^{\prime },b,\xi)\vert \leq \omega _{K}(\vert
x-x^{\prime }\vert +\vert u-u^{\prime }\vert ) (\vert b\vert ^{p}+\vert \xi\vert) 
\label{modulusofcontinuityp}
\end{equation}%
for every $( x,u,b,\xi) $ and $( x^{\prime },u^{\prime
},b,\xi) $ in $K\times \mathbb{R}^{m}\times \mathbb{R}^{d\times N}.$

Moreover, given $x_{0}\in \Omega ,$ and $\varepsilon >0$ there exists $%
\delta >0$ such that if $\vert x-x_{0}\vert <\delta $ then%
\begin{equation*}
f_p^{\infty }( x,u,b,\xi) -f_p^{\infty }( x_{0},u,b,\xi)
\geq -\varepsilon(|b|^p+|\xi|) 
\end{equation*}%
for every $(u,b,\xi) \in \mathbb{R}^d\times \mathbb{R}^{m}\times \mathbb{R}^{d\times N}.$

\item[iv)] In particular, $f_p^{\infty }$ is continuous.
\end{enumerate}
\end{proposition}

\begin{proof}

\noindent $i)$ The convexity-quasiconvexity of $f_p^{\infty }$ can be proven exactly as in \cite[Lemma 2.1]{FKP2}.

\noindent $ii)$ By definition $( \ref{recessionp}) $ we may find a
subsequence $\{ t_{k}\} $ such that%
\begin{equation*}
f_p^{\infty }( x,u,b,\xi) =\lim_{k\rightarrow +\infty }\frac{%
f( x,u,t_k^{\frac{1}{p}}b,t_{k}\xi) }{t_{k}}.
\end{equation*}%
By $( H_{1}) _{p}$ one has%
\begin{equation*}
f_p^{\infty }( x,u,b,\xi) \leq \lim_{k\rightarrow +\infty }\frac{%
C( 1+t_k|b|^{p}+t_k|\xi|) }{t_{k}}=C(|b|^p+|\xi|) 
\end{equation*}%
and 
\begin{equation*}
f_p^{\infty }(x,u,b,\xi) \geq \lim_{k\rightarrow +\infty }\frac{%
\frac{1}{C}( t_k|b|^{p}+t_k|\xi|) -C}{t_{k}}\geq \frac{1}{C}(|b|^p+|\xi|) .
\end{equation*}%
Hence $( H_{1}) _{p}$ holds for $f^{\infty }.$

\noindent $iii)$ Again \eqref{recessionp} entails that for every $(x,u), (x^{\prime
},u^{\prime })\in \Omega \times \mathbb{R}^d$ and $(b,\xi )\in 
\mathbb{R}^{m}\times \mathbb{R}^{d\times N}$ that, up to a
subsequence not relabeled, 
\begin{align}\nonumber
& f_p^{\infty }( x,u,b,\xi) -f_p^{\infty }( x^{\prime
},u^{\prime },b,\xi)\leq \lim_{k\rightarrow +\infty }\frac{f(x,u,t_k^{\frac{1}{p}}b,t_{k}\xi )-f(x',u',t_k^\frac{1}{p}b,t_{k}\xi )}{t_{k}}.
\end{align}%
By $(H_{2})_{p}$, for every $K\subset \Omega \times \mathbb{R}^d$ there
exists $\omega _{K}:\mathbb{R\rightarrow }[ 0,+\infty) $
continuous with $\omega _{K}( 0) =0$ such that if $%
(x,u),(x^{\prime },u^{\prime })\in K$, for every $(b,\xi )\in \mathbb{R}%
^{m}\times \mathbb{R}^{d\times N}$ it results 
\begin{align}
\lim\limits_{k\rightarrow +\infty }\frac{f(
x,u,t_k^{\frac{1}{p}}b,t_{k}\xi) -f( x^{\prime },u^{\prime },t_k^{\frac{1}{p}}b,t_{k}\xi) }{%
t_{k}}&\leq\lim\limits_{k\rightarrow +\infty }\frac{\omega_{K}(
\vert x-x^{\prime }\vert +\vert u-u^{\prime }\vert
) ( 1+t_k\vert b\vert ^{p}+t_k\vert \xi
\vert) }{t_{k}} \notag \\
& \displaystyle{=\omega _{K}(\vert x-x^{\prime }\vert
+\vert u-u^{\prime }\vert) (|b|^p+|\xi|).} 
\notag
\end{align}%
Changing the role of $f_p^{\infty }(x,u,b,\xi )$ with $f_p^{\infty }(x^{\prime
},u^{\prime },b,\xi ),$ \eqref{modulusofcontinuityp} follows.

For what concerns the second inequality in $iii)$, by \eqref{recessionp} and $(2)$ of $(H_{2})_{p}$ and, up
to a subsequence not relabeled, we have for every $x,x_{0}\in \Omega $
such that $|x-x_{0}|<\delta $, and every $(u,b,\xi )\in \mathbb R^d \times 
\mathbb{R}^{m}\times \mathbb{R}^{d\times N}$ 
\begin{align*}
f_p^{\infty }(x,u,b,\xi )-f_p^{\infty }(x_{0},u,b,\xi )&\geq
\lim_{k\rightarrow +\infty }\frac{f(x,u,t_k^{\frac{1}{p}}b,t_{k}\xi
)-f(x_{0},u,t_k^{\frac{1}{p}}b,t_{k}\xi )}{t_{k}} \\
&\geq -\varepsilon \lim_{k\rightarrow +\infty }\frac{%
1+t_k|b|^{p}+|t_{k}\xi |}{t_{k}}=-\varepsilon (|b|^p+|\xi |).
\end{align*}

\noindent $iv)$ The convexity-quasiconvexity and
\eqref{finftypgrowth} guarantee that $f_p^{\infty }$ is continuous with respect to $(b,\xi)$, in particular it is $(p,1)-$ Lipschitz continuous in $b$ and $\xi$ uniformly with
respect to $(x,u)$. Thus \eqref{modulusofcontinuityp}, \eqref{p-lipschitzcontinuity} and the triangular
inequality entail that 
\begin{align}
& \displaystyle{|f_p^{\infty }(x,u,b,\xi )-f_p^{\infty }(x^{\prime },u^{\prime
},b^{\prime },\xi ^{\prime })|} \notag \\ 
& \displaystyle{\leq \omega _{k}(|x-x^{\prime }|+|u-u^{\prime }|)(|\xi | +|b|^p) +\gamma|\xi
-\xi ^{\prime }|+ \gamma (1+ |b|^{p-1}+|b'|^{p-1}+ |\xi|^{\frac{1}{p'}}+ |\xi'|^{\frac{1}{p'}})|b-b'|\leq \varepsilon }  \notag
\end{align}%
provided that $|x-x^{\prime }|,|u-u^{\prime }|$, $|b-b'|$ and $|\xi -\xi ^{\prime }|$
are small.\end{proof}

\begin{remark}\label{hypmin} We emphasize that not all the assumptions on $f$ in Proposition \ref{proprecession} are necessary to prove the items above. In particular, one has that the proof of $ii)$ uses only the fact that $f$ verifies $(H_1)_p$. Moreover,  $iii)$ follows from \eqref{recessionp} and $(H_{2})_{p}$ $i)$ and $ii)$.
\end{remark}
Regarding the recession function for $p=\infty$ in \eqref{recession}, a result analogous to Proposition \ref{proprecession} holds, but the proof is omitted for the sake of brevity.

\begin{proposition}
\label{proprecessioninfty} Let $f:\Omega \times \mathbb R^d\times \mathbb{R}^{m}\times \mathbb{R}^{d\times N}\rightarrow \lbrack
0,+\infty \lbrack $,  and let $f^\infty$ be defined by \eqref{recession}, provided $f$ satisfies $(H_{0})$, $%
(H_{1})_{\infty},~(H_{2})_{\infty}$, then

\begin{enumerate}
\item[i)] $f^{\infty }$ is convex-quasiconvex.

\item[ii)] For every $M>0$, there exists $C_M>0$ such that $\frac{1}{C_M}\vert \xi \vert
\leq f^{\infty }(x,u,b,\xi) \leq C_M\vert \xi \vert$, for every $b\in \mathbb R^m$ such that $|b|\leq M$.

\item[iii)] For every $M>0,$ and for every compact set $%
K\subset\Omega\times\mathbb{R}^d$ there exists a continuous function $
\omega_{M,K}:\mathbb{R\rightarrow}[0,+\infty) $ with $
\omega_{M,K}( 0) =0$ such that if $\vert b\vert \leq M$
then
\begin{equation*}
\vert f^\infty(x,u,b,\xi) -f^\infty(
x^{\prime},u^{\prime},b,\xi) \vert \leq\omega_{M,K}(
\vert x-x^{\prime }\vert +\vert u-u^{\prime}\vert) \vert \xi\vert 
\end{equation*}
for every $(x,u,\xi) , (x^{\prime},u^{\prime},\xi)
\in K\times\mathbb{R}^{d\times N}.$

Moreover, given $x_{0}\in \Omega ,$ and $\varepsilon >0$ there exists $\delta >0$ such that if $\vert x-x_{0}\vert <\delta $ then%
\begin{equation*}
f^\infty ( x,u,b,\xi) -f^\infty( x_{0},u,b,\xi)
\geq -\varepsilon|\xi| 
\end{equation*}%
for every $( u,b,\xi) \in \mathbb{R}^d\times \mathbb{R}%
^{m}\times \mathbb{R}^{d\times N}.$

\item[iv)] In particular, $f^{\infty }$ is continuous.
\end{enumerate}
\end{proposition}

\subsection{The surface energy densities}

For any convex-quasiconvex function $f:\Omega\times\mathbb{R}^d\times%
\mathbb{R}^{m}\times\mathbb{R}^{d\times N}\rightarrow[ 0,+\infty)
,$ and $1<p\leq \infty$, we define the following surface energy densities $K_p:\Omega \times\mathbb R ^d\times\mathbb{R}^d\times\mathbb{R}^{m}\times S^{N-1}\rightarrow[
0,+\infty) $ by \eqref{Kp} if $1<p<\infty$ and  by \eqref{Kinfty} if $p=\infty$.   

A density argument guarantees that the family ${\cal A}$ in formulas \eqref{calA} can be constituted by functions in $W^{1,\infty}$, as quoted in \cite{AFP}. Analogously, in \eqref{Kp} the set $L^\infty$ can be replaced by $L^p$. 

The following result provides some properties of the density $K_p$ and develops along the lines of Lemma 2.15 in \cite{FM2}.

\begin{proposition}\label{Lemma2.15FM}
Assume $f:\Omega\times \mathbb{R}^d\times\mathbb{R}^m\times\mathbb{R}^{d\times N}\longrightarrow [0,+\infty)$ is a convex-quasiconvex function satisfying ${(H_1)_p}$, ${(H_2)_p}$ and ${(H_3)_p}$. Then
\begin{itemize}
\item[a)] there exists a constant $C$ such that
$$
|K_p(x,0,c,d,\nu)- K_p(x,0,c',d',\nu)| \leq C(|c-c'|+ |d-d'|)$$  for every $(x,c,d,\nu)$ and $(x,c',d',\nu)$ in $\Omega \times \mathbb R^d \times \mathbb R^d \times S^{N-1}$;

\item[b)] $(x,b,\nu) \mapsto K_p(x,b,c,d,\nu)$ is upper semicontinuous for every $c,d\in \mathbb R^d$;
\item[c)] $K_p(\cdot,\cdot, \cdot, 0,\cdot)$ is upper semicontinuous in $\Omega \times \mathbb R^d \times \mathbb R^d\times S^{N-1}$;
\item[d)] there exists a constant $C>0$ such that $$0\leq K_p(x,b,c,d,\nu) \leq C(|c-d|+ |b|^p),\ \forall\ (x,b,c,d,\nu)\in \Omega\times \mathbb R^m\times  \mathbb R^d\times \mathbb{R}^d \times S^{N-1}.$$
\item[e)] For  all $x_0 \in \Omega$  and for all $\varepsilon >0$ there exists $\delta>0$ such that $|x-x_0|<\delta$ implies
$$
\displaystyle{|K_p(x,b,c,d,\nu)-K_p(x_0,b,c,d,\nu)|\leq \varepsilon C(1+|b|^p+ |d-c|)}.
$$
\end{itemize}
\end{proposition}
\begin{proof}[Proof]
Condition c) is a consequence of a) and b). To prove a) we construct an admissible field $w^\ast \in {\cal A}(c',d',\nu)$ as in Lemma 2.15 in \cite{FM2} and we define $\eta^\ast \in L^\infty(Q;\mathbb R^m)$ with 0 average in $Q_\nu$ as follows 

$$\eta^\ast(y):=\left\{
\begin{array}{ll}
2^{\frac{1}{p}}\eta(2y) &\hbox{ if } |y\cdot \nu|\leq \frac{1}{4},\vspace{0.2cm}\\
0 &\hbox{ if }\frac{1}{4}\leq |y \cdot \nu| \leq \frac{1}{2}.
\end{array}
\right.
$$
where $\eta$ has been extended by periodicity to all $\mathbb R^N$ and still denoted by $\eta$.
Using conditions \eqref{p-1homogeneity}, \eqref{finftypgrowth} and the periodicity of $w$ and $\eta$ one obtains
$$
K_p(x,0,c',d',\nu)\leq \displaystyle{\int_{Q_\nu}f^\infty_p(x,w(z),\eta(z),\nabla w(z))\,dz + C(|c-c'|+|d-d'|).}
$$
Taking the infimum over all $w \in \mathcal{A}(c,d,\nu)$ and $\eta \in L^\infty(Q;\mathbb R^m)$ we conclude that
$$
K_p(x,0,c',d',\nu)\leq K_p(x,0,c,d,\nu)+C(|c-c'|+|d-d'|).
$$
The reverse inequality is obtained by letting $w \in \mathcal{A}(c',d',\nu)$ and building $w^{\ast} \in \mathcal{A}(c,d,\nu).$

To prove b), we start noticing that

$$
\displaystyle{K_p(x,b,c,d,\nu):=\inf\left\{\int_{Q}f_p^\infty(x,w(y),\eta(y),\nabla w(y)R^T)\,dy:\ w\in\mathcal{A}(c,d,e_N), \eta\in L^\infty(Q;\mathbb{R}^m),\ \int_{Q} \eta\,dy=b\right\}},
$$
where $R\in SO(N)$ is such that $Re_N=\nu$ and $RQ=Q_\nu$. Also, due to the growth conditions, by density arguments, it suffices to choose smooth functions $w$.

Let $(x_n,b_n,\nu_n) \to (x,b,\nu)$, given $\e>0$ let $w \in {\cal A}(c,d,e_N)$ be a smooth function and $\eta \in L^\infty(Q;\mathbb{R}^m)$ with $\int_Q \eta\,dy =b$ such that
$$
\left| K_p(x,b,c,d,\nu)- \int_Q f^\infty_p(x, w(y), \eta(y), \nabla w(y)R^T)\, dy \right| < \e.
$$
Consider $\eta_n \in L^\infty(Q;\mathbb R^m)$ such that $\int_{Q}\eta_n\,dy = b_n$ and $\eta_n \to \eta$ in $L^p(Q;\mathbb R^m)$. For example $\eta_n:=\eta+ b_n -b$. Let $X$ be a compact subset of $\Omega \times \mathbb R^d$ containing a neighborhood of $\{(x,w(y)):\ y \in Q\}$. By condition \eqref{modulusofcontinuityp},  there exists a continuous function $\omega_{X}$, with $\omega_{X}(0)=0$ such that
\begin{equation}\label{eq2.8FMr}
\left|f^\infty_p(y,u,b, \xi)- f^\infty_p(y',u', b, \xi)\right| \leq \omega_{X}(|y-y'|+ |u-u'|) (|b|^p+|\xi|)
\end{equation}
for every $(y,u,\xi), (y',u',\xi) \in X \times \mathbb R^{d\times N}$ and $b\in \mathbb{R}^m$. As already noticed in Proposition  \ref{proprecession}, the recession function $f^\infty_p$ is convex-quasiconvex and we have the following $(p,1)-$ Lipschitz condition for $f^\infty_p$:
\begin{equation}\label{plipschitz}
\left|f^\infty_p(x,u,b,\xi)- f^\infty_p(x,u,b',\xi')\right|\leq \gamma( |\xi-\xi'|+ (1+|b|^{p-1}+|b'|^{p-1}+|\xi|^{\frac{1}{p}}+|\xi'|^{\frac{1}{p}})|b-b'|),
\end{equation}
for every $(x,u,b,\xi)$ and $(x,u,b',\xi')$ in $\Omega\times\mathbb{R}^d\times \mathbb R^m\times\mathbb{R}^{d\times N}$.

As in \cite[Lemma 2.15]{FM2}, consider orthogonal transformations $R_n$ such that $R_n e_N= \nu_n$ and $R_n\to R$. By virtue of the preceding estimates, and standard arguments, for $n$ large enough we have
$$
 K_p(x,b_n,c,d,\nu_n)\leq 
\displaystyle{\e + \int_Q f_p^\infty(x,w(y),\eta(y), \nabla w(y)R^T)\,dy}
\leq 2 \e + K_p(x,b,c,d,\nu).
$$
Letting $\e \to 0^+$ we conclude that \;
$
\displaystyle{\limsup_{n \to + \infty} K_p(x,b_n,c,d, \nu_n) \leq K_p(x,b,c,d,\nu).}
$

The proof of d) is identical to the proof of \cite[Lemma 2.15 d)]{FM2}.

The proof of $e)$ develops along the lines of \cite[Proposition 2.9 (ii)]{BF}.
\end{proof}

\begin{proposition}\label{Kindependentofv}
Let $f:\Omega\times \mathbb{R}^d \times \mathbb{R}^m\times \mathbb{R}^{d\times N}\longrightarrow [0,+\infty)$ be a continuous function, $f^\infty$ be its recession function given by $(\ref{recession})$ and let $K_\infty$ be defined as in \eqref{Kinfty}. Then $K_\infty(x,\cdot,c,d,\nu)$ is a constant function for any fixed $(x,c,d,\nu)\in\Omega\times\mathbb{R}^d\times\mathbb{R}^d\times S^{N-1}$.
\end{proposition}

\begin{proof}
Let $(x,c,d,\nu)$ be arbitrary in $\Omega\times\mathbb{R}^d\times\mathbb{R}^d\times S^{N-1}$. Let $b,\overline{b}\in\mathbb{R}^m$ such that $b\neq\overline{b}$. We claim that $K_\infty(x,\overline{b},c,d,\nu)\le K_\infty(x,b,c,d,\nu)$. 
Let $w\in\mathcal{A}(c,d,\nu)$ and $\eta\in L^\infty(Q_\nu;\mathbb{R}^m)$ such that $\int_{Q_\nu}\eta\,dy=b$ be arbitrary and extend them by $Q_\nu$-periodicity to all $\mathbb{R}^N$. Then define in $Q_\nu$
$$
\overline{w}(y):=\left\{\begin{array}{ll} 
c &\hbox{if }-\frac{1}{2}\le y\cdot\nu<-\frac{1}{4}, \vspace{0.1cm} \\
w(2y)& \hbox{if }|y\cdot\nu|\le\frac{1}{4}, \vspace{0.1cm}\\
d &\text{if }\frac{1}{4}< y\cdot\nu\le\frac{1}{2}, \end{array}\right.  \overline{\eta}(y):=\left\{\begin{array}{cl}\eta(2y)& \text{if }|y\cdot\nu|\le\frac{1}{4},\vspace{0.1cm}\\k &\text{if }\frac{1}{4}< |y\cdot\nu|\le\frac{1}{2}, \end{array}\right.
$$ where $k$ is the constant such that $\int_{Q_\nu}\overline{\eta}\,dy=\overline{b}$, $k=2\overline{b}- b$. Notice that $\overline{w}\in\mathcal{A}(c,d,\nu)$, thus
$$\begin{array}{rcl}K_\infty(x,\overline{b},c,d,\nu)&\le& \displaystyle{\int_{Q_\nu}f^\infty(x,\overline{w}(y),\overline{\eta}(y),\nabla\overline{w}(y))\,dy=\int_{\{y\in Q_\nu:\ |y\cdot\nu|\le 1/4\}}f^\infty(x,w(2y),\eta(2y),2\nabla w(2y))\,dy}\vspace{0.1cm}\\
& = & \displaystyle{\frac{2}{2^N}\int_{\{z \in \mathbb R^N: |z\cdot\nu_i|\le 1, i=1,\dots,N-1, |z\cdot\nu|\le 1/2\}}f^\infty(x,w(z),\eta(z),\nabla w(z))\,dz}\vspace{0.1cm}\\
& = & \displaystyle{\int_{Q_\nu}f^\infty(x,w(z),\eta(z),\nabla w(z))\,dz,}$$
\end{array}$$
where we have used in the second identity the fact that $f^\infty(x,u,b,\cdot)$ is a positively 1-homogeneous function so, in particular, $f^\infty(x,u,b,0)=0$. The last identity follows from the periodicity of $w$ and $\eta$. The claim is achieved by taking the infimum on $w$ and $\eta$ on the right hand side. 

The reverse inequality follows by interchanging the roles of $b$ and $\overline{b}.$
\end{proof}
\medskip

\begin{proposition}\label{Lemma2.15FMinfty}
Assume that $f:\Omega\times \mathbb{R}^d\times\mathbb{R}^m\times\mathbb{R}^{d\times N}\longrightarrow [0,+\infty)$ is a convex-quasiconvex function satisfying $(H_0)$, $(H_1)_\infty$, $(H_2)_\infty$ and $(H_3)_\infty$. Then
\begin{itemize}
\item[a)] there exists a constant $C>0$ such that
$$
|K_\infty(x,b,c,d,\nu)- K_\infty(x,b',c',d',\nu)| \leq C(|c-c'|+ |d-d'|)$$  for every $(x,b,c,d,\nu)$ and $(x,b',c',d',\nu)$ in $\Omega \times \mathbb R^m \times \mathbb R^d \times \mathbb{R}^d\times S^{N-1}$;
\item[b)] $(x,b,\nu) \mapsto K_\infty(x,b,c,d,\nu)$ is upper semicontinuous for every $c,d\in \mathbb R^d$;
\item[c)] $K_\infty$ is upper semicontinuous in $\Omega \times \mathbb R^m \times \mathbb R^d \times\mathbb{R}^d\times S^{N-1}$;
\item[d)] there exists a constant $C>0$ such that $K_\infty(x,b,c,d,\nu) \leq C|c-d|$, for every $(x,b,c,d,\nu)\in \Omega\times \mathbb R^m\times  \mathbb R^d\times \mathbb{R}^d \times S^{N-1}.$
\end{itemize}
\end{proposition}
\begin{proof}[Proof]
The proof is very similar to Proposition \ref{Lemma2.15FM}. We just emphasize the main differences. 
To prove a) we start by noticing that by Proposition \ref{Kindependentofv}, $K_\infty(x,b,c,d,\nu)=K_\infty(x,0,c,d,\nu)$ and $K_\infty(x,b',c,d',\nu)=K_\infty(x,0,c',d',\nu)$. So we fix $w\in\mathcal{A}(c,d,\nu), \, \eta\in L^\infty(Q_\nu;\mathbb{R}^m)$ with $\int_{Q_\nu}\eta\, dy=0$ and construct $w^\ast \in {\cal A}(c',d',\nu)$ similarly as in Lemma 2.15 in \cite{FM2} and let $\eta^\ast \in L^\infty(Q_\nu;\mathbb R^m)$ with average 0 in $Q_{\nu}$ be given by

$$
\eta^\ast(y):=\left\{
\begin{array}{ll}
	\eta(2y) &\hbox{ if } |y\cdot \nu|\leq \frac{1}{4},\vspace{0.2cm}\\
	0 &\hbox{ if }\frac{1}{4}\leq |y \cdot \nu| \leq \frac{1}{2}.
\end{array}
\right.
$$

The proof of b) follows directly from Proposition \ref{Lemma2.15FM} (b) using again Proposition \ref{Kindependentofv}, replacing \eqref{eq2.8FMr} by 
\begin{equation}\nonumber
\left|f^\infty(y,u,b,\xi)- f^\infty(y',u', b, \xi)\right| \leq \omega_{X, M}(|y-y'|+ |u-u'|) |\xi|
\end{equation}
for every $(y,u,\xi), (y',u',\xi) \in X \times \mathbb R^{d \times N}$ and $b\in \mathbb{R}^m$ with $|b|\le M,$ where $M:=\Vert \eta \Vert_{L^\infty}$. And the $p-$Lipschitz continuity \eqref{plipschitz} should be replaced by
the condition
$$
\left|f^\infty(x,u,b,\xi)- f^\infty(x,u,b',\xi')\right|\leq \beta(M,n,m,N)\left((1+ |\xi|+ |\xi'|)\, |b-b'|+ |\xi-\xi'|\right),
$$
for every $(x,u,\xi)$ and $(x,u,\xi')$ in $\Omega\times\mathbb{R}^d\times\mathbb{R}^{d\times N}$ and $b,b'\in\mathbb{R}^m$ with $|b|,|b'|\le M$.\end{proof}

\begin{remark}\label{Remark2.17FMr}
If $f$ does not depend on $u$, i.e. $f \equiv f(x,v,\nabla u)$, then $K_p$ and  $K_\infty$ coincide with the  recession functions $f^\infty_p$ and $f^\infty$, respectively.
Namely, for every $(x,b,c,d,\nu)\in \Omega \times \mathbb R^d\times\mathbb R^d \times \mathbb R^m\times S^{N-1}$,
\begin{equation*}
K_p(x,b,c,d,\nu)= f^\infty_p(x, b,(c-d)\otimes \nu),
\end{equation*}
and
\begin{equation}
\label{5.84AFPinfty}
K_\infty(x,b,c,d,\nu)= f^\infty(x, b,(c-d)\otimes \nu).
\end{equation}
To obtain the above formulas, we refer to the arguments used to prove \cite[formula (5.83)]{AFP}.
From Proposition \ref{Kindependentofv}, \eqref{5.84AFPinfty} becomes
$$
\displaystyle{K_\infty(x,b,c,d,\nu)= f^\infty(x, b,(c-d)\otimes \nu)=f^\infty(x,0,(c-d)\otimes \nu).}
$$
We observe that the latter equality, in the above formula, was already proven in \cite{FKP1}.

We underline that, also when $f$ exhibits explicit dependence on $u$, there is coincidence between $K_p$ and $K_\infty$, for example consider the cases $f(x,u, b,\xi):= g(x, u)\sqrt{|b|^{2 p} + |\xi|^2}$, with $g$ suitably chosen in order to satisfy assumptions $(H_1)_p -(H_3)_p$, or $f(x,u, b,\xi):=\sqrt{|b|^{2p}+|(u, \xi)|^2}$.
\end{remark}

The following approximation result will be used to prove the upper bound inequality in Theorem \ref{MainResultinfty}.
\begin{proposition}\label{PropKr} Let $f:\Omega\times \mathbb{R}^d\times\mathbb{R}^m\times\mathbb{R}^{d\times N}\longrightarrow [0,+\infty)$ be a continuous function, and let $f^\infty$ be as in \eqref{recession}. Fix $r\ge 0$ and let $K_r:\Omega\times\mathbb R^m \times \mathbb{R}^d\times\mathbb{R}^d\times S^{N-1}\longrightarrow [0,+\infty)$ be such that
\begin{equation*}
\begin{array}{l}\displaystyle{K_r(x,b,c,d,\nu):=\inf\left\{\int_{Q_\nu}f^\infty(x,w(y),\eta(y),\nabla w(y))\,dy:\ w\in\mathcal{A}(c,d,\nu),\ \eta\in L^\infty(Q_\nu;\mathbb{R}^m),\right.}\vspace{0.1cm}\\ \hspace{9cm}\displaystyle{\left.\Vert\eta\Vert_{L^\infty(Q_\nu)}\le |b|+r,\ \int_{Q_\nu} \eta\,dy=b\right\}}.\end{array}
\end{equation*}
Then, for each $(x,b,c,d,\nu)$,
\begin{equation}\label{KKr}
K_\infty(x,b,c,d,\nu)=\lim_{r\rightarrow +\infty}K_r(x,b,c,d,\nu)=\inf_{r\ge 0}K_r(x,b,c,d,\nu).
\end{equation}
\end{proposition}

\begin{proof}[Proof]
The fact that, $K_r(x,b,c,d,\nu)$ is decreasing in $r$ gives the last identity. Moreover, $K_\infty(x,b,c,d,\nu)\le K_r(x,b,c,d,\nu)$ for any $r$, therefore it is enough to find $r_n$ such that $\lim_{n\rightarrow +\infty} K_{r_n}(x,b,c,d,\nu)=K_\infty(x,b,c,d,\nu)$. By definition of $K_\infty$, given $n\in\mathbb{N}$ we can get $\omega_n\in\mathcal{A}(c,d,\nu)$, and $\eta_n\in L^\infty(Q_\nu;\mathbb{R}^m)$ with $\int_{Q_\nu}\eta_n(y)\,dy=b$ and such that
$$K_\infty(x,b,c,d,\nu)+\frac{1}{n}>\int_{Q_\nu}f^\infty(x,\omega_n(y),\eta_n(y),\nabla\omega_n(y))\,dy.$$ Setting $r_n:=\left\Vert \eta_n \right\Vert_{L^\infty}-|v|$ we get

$$
K_\infty(x,b,c,d,\nu)+\frac{1}{n}\ge K_{r_n}(x,b,c,d,\nu)\ge K_\infty(x,b,c,d,\nu)
$$ which yields the desired condition by letting $n\rightarrow +\infty$.\end{proof}

\begin{remark}
\label{Kre}
 
Notice that, for  all $x_0 \in \Omega$  and all $\varepsilon >0$ there exists $\delta>0$ such that $|x-x_0|<\delta$ implies the existence of a suitable constant $C_{|b|+r}$ for which
\begin{equation}\label{ucKr}
\displaystyle{|K_r(x,b,c,d,\nu)-K_r(x_0,b,c,d,\nu)|\leq \varepsilon C_{|b|+r}(1+ |d-c|)},
\end{equation}
for every $b \in \mathbb R^m$, $c,d,\in\mathbb R^d, \nu \in S^{N-1}$.

We also observe that arguments entirely similar to those in Proposition \ref{Lemma2.15FMinfty} guarantee that
\begin{equation*}
K_r(x,b,c,d,\nu)\leq C_{|b|+r}|c-d|,
\end{equation*}
for every $(x,b,c,d,\nu)\in \Omega\times \mathbb R^m\times \mathbb R^d \times \mathbb R^d \times S^{N-1}$.
\end{remark}

\section{Main Results: $BV\times L^p, 1<p<\infty$}\label{MainLp}
In this section we prove Theorem \ref{MainResultp}. 

\subsection{Lower semicontinuity in $BV\times L^{p}$}

\begin{theorem}\label{lowerboundL_pthm}
Let $f:\Omega
\times\mathbb{R}^{d}\times\mathbb{R}^{m}\times\mathbb{R}^{d\times
N}\rightarrow[ 0,+\infty) $ be a continuous function satisfying $(H_0)$, $%
( H_{1}) _{p}-( H_{3}) _{p}$. Then 
\begin{align}
\underset{n\rightarrow+\infty}{\lim\inf}\int_{\Omega}f(x,u_{n},v_{n},\nabla u_{n}) dx &
\geq\int_{\Omega}f(x,u,v,\nabla
u) dx  +\int_{J_{u}}K_p( x,0,u^{+},u^{-}
,\nu_{u}) d\mathcal{H}^{N-1} \label{lowerboundLp}\\
& +\int_{\Omega}f_p^{\infty}(x,u,0,\tfrac{dD^{c}u}{%
d\vert D^{c}u\vert }) d\vert
D^{c}u\vert  \notag
\end{align}
in $BV(\Omega;\mathbb{R}^{d}) \times L^{p}(\Omega ;\mathbb{%
R}^{m}) $ with respect to the ($L^{1}-$strong$~\times~L^{p}-$%
weak)$-$convergence,
where $K_p$ is given by \eqref{Kp} and $f^\infty_p$ is the $(p,1)-$ \emph{recession function} given by \eqref{recessionp}.
\end{theorem}

\begin{proof}
Using the same arguments as in \cite[Theorem II.4]{AF} and \cite[Proposition 2.4]{FM1} we may reduce to $u_n\in C_{0}^{\infty}( \mathbb{R}%
^{N};\mathbb R ^d) $ and $v_{n}\in C_{0}^{\infty}( \mathbb{R}%
^{N};\mathbb R ^ m).$ Due to $( H_{1}) _{p}$ we may assume, without loss of generality,
that

\begin{equation*}
\underset{n\rightarrow+\infty}{\lim\inf}\int_{\Omega}f( x,u_{n},v_{n},\nabla u_{n})
dx=\lim_{n\rightarrow+\infty}\int_{\Omega}f( x,u_{n},v_{n},\nabla u_{n}) dx<+\infty. 
\end{equation*}

Hence, up to a subsequence, $\mu_{n}:=f( x,u_{n}
,v_{n},\nabla u_{n}) \mathcal{L}^{N}%
\overset{\ast}{\rightharpoonup}\mu$ in the sense of measures for some
positive Radon measure $\mu.$ By the Radon-Nikod\'{y}m theorem we can decompose $%
\mu$ as a sum of four mutually nonnegative measures, namely, $\mu=\mu_{a}\mathcal{L}^{N}+\mu_{j}\mathcal{H}^{N-1}\lfloor J_{u}+\mu_{c}\vert D^{c}u\vert +\mu_{s}.$

By Besicovitch derivation theorem 
\begin{align*}
\mu_{a}( x_{0}) & =\lim_{\varepsilon\rightarrow0^{+}}\frac {%
\mu( B(x_{0},\varepsilon)) }{\mathcal{L}^{N}(
B( x_{0},\varepsilon)) }<+\infty,~\text{for }\mathcal{L}%
^{N}-\text{a.e. }x_{0}\in\Omega,  \notag \\
\mu_{j}( x_{0}) & =\lim_{\varepsilon\rightarrow0^{+}}\frac {%
\mu (Q_{\nu}( x_{0},\varepsilon))}{\mathcal{H}%
^{N-1}( Q_{\nu}( x_{0},\varepsilon) \cap J_{u}) }%
<+\infty,~\text{for }\mathcal{H}^{N-1}-\text{a.e. }x_{0}\in J_{u}\cap \Omega,\\
\mu_{c}( x_{0}) & =\lim_{\varepsilon\rightarrow0^{+}}\frac {%
\mu( Q( x_{0},\varepsilon)) }{\vert
Du\vert ( Q( x_{0},\varepsilon)) }<+\infty,~%
\text{for }\vert D^{c}u\vert -\text{a.e. }x_{0}\in \Omega.  \notag
\end{align*}
We claim that 
\begin{align}
\mu_{a}( x_{0}) & \geq f( x_{0},u( x_{0})
,v( x_{0}) ,\nabla u( x_{0})) ,~\ \ \ \ \text{ for }\mathcal{L}^{N}-\text{a.e. }x_{0}\in\Omega , \label{lowerboundbulk}\\
\mu_{j}( x_{0}) & \geq K_p( x_{0},0,u^{+}( x_{0})
,u^{-}( x_{0}) ,\nu_{u}( x_{0})) ,~\ \ \ 
\text{for }\mathcal{H}^{N-1}-\text{a.e. }x_{0}\in J_{u}\cap\Omega,
\label{lowerboundjump}\\
\mu_{c}( x_{0}) & \geq f_p^{\infty}( x_{0},u(x_{0}) ,0,\frac{dD^{c}u}{d\vert D^{c}u\vert }(
x_{0})) \text{, \;for }\vert D^{c}u\vert -\text{a.e.~%
}x_{0}\in\Omega.  \label{lowerboundcantor}
\end{align}
If $(\ref{lowerboundbulk})-(\ref{lowerboundcantor}) 
$ hold then $(\ref{lowerboundLp}) $ follows immediately. Indeed,
since $\mu_{n}\overset{\ast}{\rightharpoonup}\mu$ in the sense of measures
then%
\begin{align*}
\underset{n\rightarrow+\infty}{\lim\inf}\int_{\Omega}f( x,u_{n},v_{n},\nabla u_{n}) dx&\geq
\underset{n\rightarrow+\infty}{\lim\inf}\mu_{n}( \Omega)
\geq\mu( \Omega)\geq\int_{\Omega}\mu_{a}dx+\int_{J_{u}}\mu_{j}d\mathcal{H}%
^{N-1}+\int_{\Omega}\mu_{c}d\vert D^{c}u\vert \\
& \geq\int_{\Omega}f(x,u,v,\nabla
u) dx+\int_{J_{u}}K_p( x,0,u^{+},u^{-},\nu_{u}) d\mathcal{H}%
^{N-1} \\
& +\int_{\Omega}f_p^{\infty}\left( x,u,0,\frac{dD^{c}u}{%
d\vert D^{c}u\vert }\right) d\vert
D^{c}u\vert ,
\end{align*}
where we have used the fact that $\mu_{s}$ is nonnegative.

We prove $(\ref{lowerboundbulk})-(\ref{lowerboundcantor})$ using the blow up method introduced in \cite{FM1}.

\noindent \textbf{Bulk part.} Inequality \eqref{lowerboundbulk} is obtained as in \cite{RZ} Section 3 and \cite{RZerr}.

\smallskip
\noindent\textbf{Jump part.} Consider $x_{0}\in J_{u}$, then there exist $u^-(x_0), u^+(x_0) \in \mathbb R^d$ and $\nu:=\nu_u(x_0)\in S^{N-1}$ such that \eqref%
{jump} holds,
\begin{equation*}
\mu_{j}( x_{0}) =\lim_{\varepsilon\rightarrow0^{+}}\frac {%
\mu( Q_{\nu}(x_{0},\varepsilon)) }{\vert
u^{+}-u^{-}\vert \mathcal{H}^{N-1}\lfloor J_{u}( Q_{\nu}(
x_{0},\varepsilon)) }\in\mathbb{R}
\end{equation*}
and assume $\mu( \partial Q_{\nu}( x_{0},\varepsilon_{k})
) =0$ for $\{ \varepsilon_{k}\} \searrow0^+.$
Moreover, for ${\cal H}^{N-1}\lfloor J_u-$ a.e. $x_0$, we may assume
\begin{align}
\displaystyle{\frac{1}{\vert u^{+}( x_{0}) -u^{-}( x_{0})
\vert }\lim_{k\to +\infty}\frac{1}{\varepsilon_k^{N-1}}\int_{Q_\nu(x_0+ \varepsilon_k)} |v(x)|^pdx=0.}\label{0derivativevp}
\end{align}
Then
\begin{align}
\mu _{j}( x_{0}) & =\displaystyle{\lim_{k\rightarrow +\infty }\frac{\mu(
Q_{\nu }( x_{0},\varepsilon _{k})) }{\vert
u^{+}-u^{-}\vert \mathcal{H}^{N-1}\lfloor J_{u}( Q_{\nu }(
x_{0},\varepsilon _{k})) }  } \nonumber \\
&\displaystyle{\geq \frac{1}{\vert u^{+}( x_{0}) -u^{-}(
x_{0})\vert }\lim_{k\rightarrow +\infty }\lim_{n\rightarrow
+\infty }\frac{1}{\varepsilon _{k}^{N-1}}\int_{Q_\nu(x_0+ \varepsilon_k)}f( x,u_{n}( x) ,v_{n}(x) ,\nabla
u_{n}(x)) dx } \label{mujestimate1}
\\
 &\displaystyle{=\frac{1}{\vert u^{+}(x_{0}) -u^{-}( x_{0})
\vert }\lim_{k\rightarrow +\infty }\lim_{n\rightarrow +\infty
}\int_{Q_\nu}\varepsilon _{k}f(x_{0}+\varepsilon _{k}y,u_{n}(
x_{0}+\varepsilon _{k}y) ,v_{n}(x_{0}+\varepsilon _{k}y)
,\nabla u_{n}( x_{0}+\varepsilon _{k}y)) dy} \nonumber\\
 &\displaystyle{=\frac{1}{\vert u^{+}( x_{0}) -u^{-}( x_{0})
\vert }\lim_{k\rightarrow +\infty }\lim_{n\rightarrow +\infty
}\int_{Q_\nu}\varepsilon _{k}f(x_{0}+\varepsilon _{k}y,u_{n,k}(
y) ,\varepsilon _{k}^{-\frac{1}{p}}v_{n,k}(y) ,\frac{1}{%
\varepsilon _{k}}\nabla u_{n,k}(y)) dy,} \nonumber
\end{align}
where%
\begin{equation*}
u_{n,k}(y) :=u_{n}( x_{0}+\varepsilon _{k}y) ,~\ \ \
\ v_{n,k}(y) :=\varepsilon _{k}^{\frac{1}{p}}v_{n}(
x_{0}+\varepsilon _{k}y) .
\end{equation*}%

We observe that, 
\begin{equation}
\label{ulimit}
\displaystyle{\lim_{k \to +\infty}\lim_{n\to +\infty}\|u_{n,k}(y)- u_0\|_{L^1(Q;\mathbb R^d)}= 0},
\end{equation}
with
\begin{equation}
\label{u0}
u_0(y):=\left\{
\begin{array}{ll}
u^+(x_0) &\hbox{ if }y \cdot \nu > 0,\\
u^-(x_0) &\hbox{ if }y \cdot \nu \leq 0,
\end{array}
\right.
\end{equation}
and for every $\varphi \in L^q(Q;\mathbb R^m)$, 
\begin{equation}\label{doubleweaklimit}
\begin{array}{lll}
\displaystyle{\lim_{k \to+\infty}\lim_{n\to +\infty}\int_Q v_{n,k}(y)\varphi (y)dy=\lim_{k \to +\infty}\varepsilon_k^{\frac{1}{p}} \int_Q v(x_0+ \varepsilon_k y)\varphi(y)dy=0,} 
\end{array}
\end{equation}
where the latter equality is obtained from \eqref{0derivativevp}.

Using the separability of $L^q(Q;\mathbb R^m)$, together with a diagonalization argument, from \eqref{mujestimate}, \eqref{ulimit}  and \eqref{doubleweaklimit}, we obtain the existence of sequences $\bar{u}_k:=u_{n(k),k}$ and $\bar{v}_k:=v_{n(k),k}$ such that $\bar{u}_k \to u_0$ in $L^1(Q;\mathbb R^d)$, $\bar{v}_k \rightharpoonup 0$ in $L^p(Q;\mathbb R^m)$, and we obtain the following estimation for $\mu _{j}$ in terms of $f^\infty_p$
\begin{align}
\mu _{j}( x_{0})&\geq \frac{1}{\vert u^{+}(
x_{0}) -u^{-}(x_{0})\vert }\lim_{k\rightarrow
+\infty }\left \{\int_{Q}f_p^{\infty }(
x_{0},\bar{u}_{k},\bar{v}_{k},\nabla \bar{u}_{k})dy \right. \notag\\
&+\int_{Q} f_p^{\infty }(x_{0}+\varepsilon _{k}y,\bar{u}_{k},\bar{v}_{k},\nabla \bar{u}_{k})
-f_p^{\infty }( x_{0},\bar{u}_k,\bar{v}_{k}
,\nabla \bar{u}_{k}) dy \label{mujestimate}\\
&+\left. \int_{Q} \varepsilon_{k}f (x_{0}+\varepsilon
_{k}y,\bar{u}_{k},\varepsilon_{k}^{-\frac{1}{p}}\bar{v}_{k},\tfrac{1}{\varepsilon _{k}}\nabla \bar{u}_{k})
-f_p^{\infty }(x_{0}+\varepsilon _{k}y,\bar{u}_{k}
,\bar{v}_{k},\nabla \bar{u}_k)
dy\right \}.\notag
\end{align}
From Proposition \ref{proprecession} iii) we get that for any $\varepsilon
>0,$ if $k$ is sufficiently large%
\begin{align*}
&\int_{Q}f_p^{\infty }( x_{0}+\varepsilon_{k}y,\bar{u}_{k},\bar{v}_{k},\nabla \bar{u}_{k})
-f_p^{\infty}(x_{0},\bar{u}_{k},\bar{v}_{k},\nabla \bar{u}_{k}) dy\\
&\geq -\varepsilon \int_{Q} \vert \bar{v}_{k}\vert
^{p}+\vert \nabla \bar{u}_{k}\vert dy
=-\varepsilon \int_{Q}\varepsilon _{k}(\vert v_{k}(
x_{0}+\varepsilon _{k}y)\vert ^{p}+\vert \nabla
u_k( x_{0}+\varepsilon _{k}y)\vert) dy\geq O(\varepsilon). 
\end{align*}
On the other hand, using $( H_{3})_{p}$ and H\"{o}lder inequality we
get%
\begin{align*}
& \int_{Q} \varepsilon _{k}f(x_{0}+\varepsilon
_{k}y,\bar{u}_{k},\varepsilon _{k}^{-\tfrac{1}{p}}\bar{v}_{k},\tfrac{1}{\varepsilon _{k}}\nabla \bar{u}_{k})
-f_p^{\infty}(x_{0}+\varepsilon _{k}y,\bar{u}_{k}
,\bar{v}_{k},\nabla \bar{u}_{k}) dy \\
 &\leq c'\int_{\{ y\in Q: \tfrac{\vert\nabla
\bar{u}_{k}\vert}{\varepsilon _{k}}+\tfrac{|\bar{v}_{k}|^p}{\e_k} \geq L\}}(\varepsilon _{k}^{\tau
}\vert \bar{v}_{k}\vert ^{(1-\tau)p} +\vert\nabla \bar{u}_{k}\vert ^{1-\tau }\varepsilon _{k}^{\tau})dy+C\int_{\{ y\in Q:\tfrac{\vert\nabla
\bar{u}_{k}\vert}{\varepsilon _{k}}+\tfrac{|\bar{v}_{k}|^p}{\e_k} < L \} }(\vert \bar{v}_k\vert ^{p}+\vert \nabla \bar{u}_k\vert) dy \\
 &\leq O(\varepsilon) +c'\varepsilon _{k}^{\tau
}\left(\int_{Q}\vert \nabla \bar{u}_{k}\vert
dy\right) ^{1-\tau}+c'\e_k\int_{Q}|v_k(x_0+\varepsilon_k y)|^{(1-\tau)p}dy +O( \varepsilon) \\
 &\leq O(\varepsilon)+c^{\prime }\varepsilon _{k}^{\tau
 }\left(\int_{Q}\vert \nabla \bar{u}_{k}\vert
 dy\right) ^{1-\tau}+c'\varepsilon_k^{\tau}\left(\varepsilon_k \int_Q |v_k(x_0+\varepsilon_k y)|^p dy\right)^{1-\tau}+O(\varepsilon)=O(\varepsilon),
\end{align*}%
where we have used in last equality \eqref{doubleweaklimit}. Thus we are led to 
\begin{equation}
\label{mujestimate2}
\mu _{j}( x_{0}) \geq \frac{1}{\vert u^{+}(
x_{0}) -u^{-}( x_{0}) \vert }\lim_{k\rightarrow
+\infty }\int_{Q}f_p^{\infty }(
x_{0},\bar{u}_{k},\bar{v}_{k},\nabla \bar{u}_{k})dy +O(\varepsilon ) .
\end{equation}%

\noindent Next we apply Lemma \ref{3.1FMr} to $f^\infty_p(x_0, \cdot, \cdot, \cdot)$, obtaining
\begin{equation}
\label{mujestimate3}
\displaystyle{
\lim_{k\rightarrow
+\infty }\int_{Q}f_p^{\infty }(
x_{0},\bar{u}_{k},\bar{v}_{k},\nabla \bar{u}_{k})dy \geq \limsup_{k \to +\infty}\int_{Q}f_p^{\infty }(
x_{0},\xi_{k},\zeta_{k},\nabla \xi_{k})dy,}
\end{equation}
where $\xi_k \to u_0$ in $L^1(Q;\mathbb R^d)$ and $\xi_k \in {\cal A}(u^+(x_0), u^-(x_0),\nu_u(x_0))$, $\zeta_k \rightharpoonup 0$ in $L^p(Q;\mathbb R^m)$ with
  
\noindent $\int_Q \zeta_k\,dy =0$. In particular, by \eqref{Kp} we have 
\begin{equation*}
\mu _{j}(x_{0})\geq K_p(x,0,u^{+}(x_{0}),u^{-}(x_{0}),\nu _{u}(x_{0}))\hbox{
for }\mathcal{H}^{N-1}-\hbox{a.e.}\, x_{0}\in J_{u}\cap \Omega.
\end{equation*}

\noindent \textbf{Cantor Part.} By definition,

\begin{equation}\label{eq1Cantorlower}
\mu^c(x)=\lim_{\varepsilon\rightarrow 0^+}\frac{\mu(x+\varepsilon Q)}{|Du|(x+\varepsilon Q)},\ |D^c u|-\hbox{a.e.}\ x\in\Omega.
\end{equation}
We start recalling that, by Alberti's rank-one theorem  (see \cite{A}), together with \eqref{ADMProposition2.2},
\begin{equation}\label{eq2Cantorlower}\lim_{\varepsilon \to 0^+}\frac{Du(x+\varepsilon Q)}{|Du|(x+\varepsilon Q)}=\lim_{\varepsilon \to 0^+}\frac{D^cu(x+\varepsilon Q)}{|D^cu|(x+\varepsilon Q)}=A(x),\ |D^c u|-a.e.\ x\in\Omega\end{equation}
for some rank-one matrix $A(x)$ with $|A(x)|=1$.

Since $|D^cu|(J_u)=0$ and still denoting by $u$ the approximate limit of $u$, which is defined in $\Omega\setminus J_u$, we have (cf. Definition 3.63 in \cite{AFP})
\begin{equation}\label{eq3Cantorlower}\lim_{\varepsilon\rightarrow 0^+}\frac{1}{|x+\varepsilon Q|}\int_{x+\varepsilon Q}|u(y)-u(x)| dy=0, |D^c u|-\hbox{a.e.}\ x\in\Omega.\end{equation}
Finally, by Besicovitch Derivation theorem \cite[Theorem 5.52]{AFP},
\begin{equation}\label{eq4Cantorlower}\lim_{\varepsilon\rightarrow 0^+}\frac{|Du|(x+\varepsilon Q)}{\varepsilon^{N-1}}=0,\,\,\,\,\lim_{\varepsilon\rightarrow 0^+}\frac{|Du|(x+\varepsilon Q)}{\varepsilon^{N}}=+\infty, \,\, |D^cu|-\hbox{a.e.}\,\, x \in\Omega.
\end{equation}

\noindent Let $x_0\in\Omega$ be such that (\ref{eq1Cantorlower})$-$(\ref{eq4Cantorlower}) hold. Notice that, as in Lemma 2.13 in \cite{FM2}, we can also assume that  $$\lim_{t\rightarrow 1^-}\liminf_{\varepsilon\rightarrow 0^+}\frac{|Du|((x_0+\varepsilon Q)\setminus (x_0+t\varepsilon Q))}{|Du|(x_0+\varepsilon Q)}=0,$$
so, we can write, for some continuous function $\omega:[0,1]\rightarrow \mathbb{R}$ with $\omega(0)=0$, \begin{equation}\label{eq5Cantorlower}\displaystyle{\liminf_{\varepsilon\rightarrow 0^+}\frac{|Du|((x_0+\varepsilon Q)\setminus  (x_0+t\varepsilon Q))}{|Du|(x_0+\varepsilon Q)}\le \omega(1-t)}.\end{equation}

In the sequel, without loss of generality, we assume $A:=A(x_0)=a\otimes e_N$ with $|a|=1$ and, as in \cite{FM2}, we divide the proof in several steps.
\smallskip

\noindent{\bf Step 1.} For each $0<t<1$ and $\gamma\in (t,1)$ consider $\varepsilon_k\rightarrow 0^+$ such that
\begin{equation}
\label{limitv=0}
\displaystyle{\lim_{k \to +\infty}\frac{\int_{Q_k} |v(x)|dx}{|Du|(Q_k)}= \lim_{k \to +\infty}\frac{\int_{Q_k}|v(x)|^pdx}{|Du|(Q_k)}=0, \; |D^c u|-\hbox{a.e.},}
\end{equation}
where $Q_k:=x_0+\varepsilon_kQ.$ We also observe that, from (\ref{eq2Cantorlower})
\begin{equation}\label{limitA}\displaystyle{\lim_{k\rightarrow +\infty}\frac{Du(Q_k)}{|Du|(Q_k)}=\lim_{k \to +\infty}\frac{D^cu(Q_k)}{|D^cu|(Q_k)}=A.}\end{equation}

Arguing as in \cite[Section 4]{FM2}, conditions (\ref{eq1Cantorlower}) and (\ref{eq3Cantorlower}) imply the existence of subsequences  $\{\bar{u}_k\} \subseteq \{u_n\}$, and $\{\bar{v}_k \}\subseteq \{v_n\}$, defined in $\Omega$, 
such that

\smallskip
\noindent a) $\displaystyle{\mu^c(x_0)\ge \limsup_{k\rightarrow +\infty}\frac{1}{|Du|(Q_k)}\int_{\gamma Q_k}f(x,\bar{u}_k(x),\bar{v}_k(x),\nabla \bar{u}_k(x))\,dx;}$\medskip

\noindent b) $\displaystyle{\lim_{k\rightarrow +\infty}\frac{1}{|Q_k|}\int_{Q_k}|\bar{u}_k(x)-u(x_0)|\,dx=0;}$\medskip

\noindent c) $\displaystyle{\lim_{k\rightarrow +\infty}\frac{1}{\varepsilon_k |Du|(Q_k)}\int_{Q_k}\left| \bar{u}_k(x)-u(x)-\frac{1}{|Q_k|}\int_{Q_k} (\bar{u}_k(z)-u(z))\,dz\right|\,dx=0;}$\medskip

\noindent d) $\frac{(\varepsilon_k^N)^\frac{1}{p}\bar{v}_k(x_0+ \varepsilon_k \cdot)}{|D^c u|^{\frac{1}{p}}(Q_k)}\rightharpoonup 0$ in $L^p(Q;\mathbb R^m)$ as $k\to +\infty,$ which follows from H\"{o}lder inequality and \eqref{limitv=0}.

\smallskip
\noindent{\bf Step 2.} In this step we will obtain an estimate for $\mu^c(x_0)$ similar to condition a), fixing on $f$ the value of $x$ and $u$. Precisely, we prove that there is $n_0\in\mathbb{N}$ such that, for each $n\ge n_0$ there exist $\{\tilde{u}_k\}\subset W^{1,1}(\Omega;\mathbb{R}^d)$, $\{\tilde{v}_k\} \subset L^p(\Omega;\mathbb R^m)$, and $\{a_k\}\subset \mathbb R$ such that $a_k\rightarrow u(x_0)$, ${\tilde v}_k \rightharpoonup v$ in $L^p(\Omega;\mathbb R^m)$ as $k\rightarrow +\infty,$  $\Vert\tilde{u}_k-u(x_0)\Vert_{L^\infty}\le 1/n$ and
\begin{equation}\label{eqwithtildeu}
(1+\omega_{K}(\tfrac{1}{n}))\,\mu^c(x_0)\ge\limsup_{k\rightarrow +\infty}\frac{1}{|Du|(Q_k)}\int_{\gamma Q_k}f(x_0,u(x_0),\tilde{v}_k(y),\nabla\tilde{u}_k(y))\,dy,\end{equation}
where $\omega_K$ is the function in $(H_2)_p$, related to a compact set $K \subset \Omega\times\mathbb{R}^d$ containing $(x_0,u(x_0))$, and the estimate does not depend on $k$. We also prove that 
\begin{equation}\label{secondeqwithak}\lim_{k\rightarrow +\infty}\frac{1}{\varepsilon_k |Du|(Q_k)}\int_{Q_k}\left|\tilde{u}_k(y)-a_k-\left(u(y)-\frac{1}{|Q_k|}\int_{Q_k}u(z)\,dz  \right)\right|\,dy=0.\end{equation}

Observe that by (\ref{eq3Cantorlower}) and condition b) above, we can assume
$$\frac{1}{|Q_k|}\int_{Q_k}|u(y)-u(x_0)|\,dy\le\frac{1}{n^2}\text{\quad and\quad }\frac{1}{|Q_k|}\int_{Q_k}|\bar{u}_k(y)-u(x_0)|\,dy\le\frac{1}{n^2}.$$
Then let $a_k:=\frac{1}{|Q_k|}\int_{Q_k}\bar{u}_k(y)\,dy$. Clearly, from condition b) above, $a_k\rightarrow u(x_0)$. To define $\tilde{u}_k$ we start by considering a family of smooth cut-off functions $\varphi_{r,s}:\mathbb{R}\rightarrow [0,1]$ such that 
\begin{equation}\label{varphirs}
\varphi_{r,s}(t):=\left\{\begin{array}{ccc}1 &  & \text{if }t\leq r, \\ 
0 &  & \text{if }t\geq s,\end{array}\right.
\end{equation} and $\Vert\varphi'_{r,s}\Vert_{L^\infty}\le \frac{c}{s-r}$ for $\frac{2}{n^2}\le r<s \le\frac{1}{2n}$. 
Consider the sequence $\{\tau_{L_k}({\bar v}_k)\}$ of $p$-equi-integrable functions derived from $\{{\bar v}_k\}$, as in \cite[Lemma 8.13]{FL}. Then, for every $\lambda \in (0,+\infty)$ define two families of sequences  $$\tilde{u}_k^{r,s,\lambda}:=a_k+\varphi_{r,s}(|\bar{u}_k-a_k| + \tfrac{|\tau_{L_k}({\bar v}_k)- {\bar v}_k|}{\lambda})(\bar{u}_k-a_k),
$$
$$
\tilde{v}^{r,s,\lambda}_k:= \tau_{L_k}({\bar v}_k)+ \varphi_{r,s}(|\bar{u}_k-a_k| + \tfrac{|\tau_{L_k}(\bar{v}_k)- \bar{v}_k|}{\lambda})(\bar{v}_k -\tau_{L_k}(\bar{v}_k)).
$$ 
Notice that, since $a_k\rightarrow u(x_0)$, for sufficiently large $k$ and independently of $r$, $s$ and $\lambda$, $\Vert\tilde{u}_k^{r,s}-u(x_0)\Vert_{L^\infty}\le 1/n$,  $v_k^{r,s,\lambda}\rightharpoonup v$ in $L^p(\Omega;\mathbb R^m)$ and satisfies $d)$ as $k\to+\infty$. The sequences $\{\tilde{u}_k\}$ and $\{\bar{v}_k\}$ will be chosen among the sequences of the previous family for convenient $r, s$ and $\lambda$. In order to make that choice we start doing some estimates. Using hypothesis $(H_2)_p$, for some compact set  $K$ containing $(x_0,u(x_0))$ and $(y,\tilde{u}_k^{r,s,\lambda})$ for $y\in \gamma Q_k$,
\begin{align}
&\int_{\gamma Q_k}f(x_0,u(x_0),\bar{v}^{r,s,\lambda}_k,\nabla\tilde{u}_k^{r,s,\lambda})\,dy\nonumber\\
&=\int_{\gamma Q_k}f(x_0,u(x_0),\tilde{v}^{r,s,\lambda}_k,\nabla\tilde{u}_k^{r,s,\lambda})-f(y,\tilde{u}_k^{r,s,\lambda},\tilde{v}^{r,s,\lambda}_k,\nabla\tilde{u}_k^{r,s,\lambda})dy+\int_{\gamma Q_k}f(y,\tilde{u}_k^{r,s,\lambda},\tilde{v}^{r,s,\lambda}_k,\nabla\tilde{u}_k^{r,s,\lambda})\,dy\nonumber\\
&\le\int_{\gamma Q_k}\omega(|y-x_0|+|\tilde{u}_k^{r,s,\lambda}-u(x_0)|)(1+|\nabla\tilde{u}_k^{r,s,\lambda}|+ |\tilde{v}^{r,s,\lambda}_k|^p)\,dy+\int_{\gamma Q_k}f(y,\tilde{u}_k^{r,s,\lambda},\tilde{v}^{r,s,\lambda}_k,\nabla\tilde{u}_k^{r,s,\lambda})\,dy\nonumber\\
&\le \int_{\gamma Q_k}\omega (\gamma\varepsilon_k+1/n)(1+|\nabla\tilde{u}_k^{r,s,\lambda}|+ |\tilde{v}_k^{r,s,\lambda}|^p)\,dy +\int_{\gamma Q_k}f(y,\tilde{u}_k^{r,s,\lambda},\tilde{v}^{r,s,\lambda}_k,\nabla\tilde{u}_k^{r,s,\lambda})\,dy.\label{estimatewithfixedxu}
\end{align}
Using hypothesis $(H_1)_p$ and for sufficiently large $n$ we can get the estimate
$$
\begin{array}{ll}\displaystyle{\int_{\gamma Q_k}(|\nabla\tilde{u}_k^{r,s,\lambda}| + |\tilde{v}^{r,s,\lambda}_k|^p)\,dy\le
c\int_{\gamma Q_k}f(y,\tilde{u}_k^{r,s,\lambda},\tilde{v}^{r,s,\lambda}_k,\nabla\tilde{u}_k^{r,s,\lambda})\,dy.}
\end{array}
$$
Recalling that $\frac{\varepsilon_k^N}{|Du|(Q_k)}\rightarrow 0$ (see \eqref{eq4Cantorlower}), to estimate \eqref{estimatewithfixedxu} we are left with
\begin{align*}
&\displaystyle{\frac{1}{|Du|(Q_k)}\int_{\gamma Q_k}f(y,\tilde{u}_k^{r,s,\lambda},\tilde{v}^{r,s,\lambda}_k,\nabla\tilde{u}_k^{r,s,\lambda})\,dy}\\
&\le \displaystyle{\frac{1}{|Du|(Q_k)}\left\{\int_{\gamma Q_k}f(y,\bar{u}_k,\bar{v}_k,\nabla \bar{u}_k)\,dy\left.+\int_{\gamma Q_k\cap\{|\bar{u}_k-a_k|+ \tfrac{|\tau_{L_k}({\bar v}_k)-{\bar v}_k|}{\lambda}\ge s\}}f(y,a_k,\tau_{L_k}(\bar{v}_k),0) \,dy\right.\right.}\\
&+\displaystyle{C\int_{\gamma Q_k\cap\{r<|\bar{u}_k-a_k|+ \tfrac{|\tau_{L_k}({\bar v}_k)- {\bar v}_k|}{\lambda}<s\}} \left(1+\tfrac{1}{s-r}|\bar{u}_k-a_k| \left|\nabla(|\bar{u}_k-a_k|+|\tfrac{\tau_{L_k}({\bar v}_k)-{\bar v}_k}{\lambda}|)\right| \right)dy}\\
&\displaystyle{+\left. C\int_{\gamma Q_k\cap\{r<|\bar{u}_k-a_k|+ \tfrac{|\tau_{L_k}({\bar v}_k)- {\bar v}_k|}{\lambda}<s\}}\left(|\nabla(\bar{u}_k-a_k)||\tau_{L_k}({\bar v}_k)- {\bar v}_k|^p+ |\tau_{L_k}({\bar v}_k)|^p\right)\,dy \right\},}\\
\end{align*}
where we have used  $(H_1)_p,$ co-area formula and exploited the fact that ${\bar v}_k$ is regular.

Thus
\begin{align*}
&\frac{1}{|Du|(Q_k)}\int_{\gamma Q_k}f(y,\tilde{u}_k^{r,s,\lambda},\tilde{v}^{r,s,\lambda}_k,\nabla\tilde{u}_k^{r,s,\lambda})\,dy\\
&\le \displaystyle{\frac{1}{|Du|(Q_k)}\left\{\int_{\gamma Q_k}f(y,\bar{u}_k,\bar{v}_k,\nabla \bar{u}_k)\,dy+C|\gamma Q_k| +C\int_{\gamma Q_k\cap\{|\bar{u}_k-a_k|+ \tfrac{|\tau_{L_k}({\bar v}_k)- {\bar v}_k|}{\lambda}\ge s\}}|\tau_{L_k}({\bar v}_k)|^p dy \right.}\\
&\displaystyle{+C\frac{s}{s-r}\int_{{\gamma Q_k}\cap\{r<|\bar{u}_k-a_k|+\tfrac{|\tau_{L_k}({\bar v}_k)-{\bar v}_k|}{\lambda}<s\}} \left|\nabla\left(|\bar{u}_k-a_k| + \tfrac{|\tau_{L_k}(\bar{v}_k)-{\bar v}_k|}{\lambda}\right)\right|+|\nabla\bar{u}_k|dy}\\
&\displaystyle{+C\left.\int_{\gamma Q_k\cap\{r<|\bar{u}_k-a_k|+ \tfrac{|\tau_{L_k}({\bar v}_k)- {\bar v}_k|}{\lambda}<s\}}[|\tau_{L_k}({\bar v}_k)- {\bar v}_k|^p+ |\tau_{L_k}({\bar v}_k)|^p]\,dy\right \}} \\
&\le\displaystyle{\frac{1}{|Du|(Q_k)}\left\{\int_{\gamma Q_k}f(y,\bar{u}_k,\bar{v}_k,\nabla \bar{u}_k)\,dy+C|\gamma Q_k|+\int_{\gamma Q_k\cap\{|\bar{u}_k-a_k|+ \frac{|\tau_{L_k}({\bar v}_k)- {\bar v}_k|}{\lambda}\ge s\}}|\tau_{L_k}({\bar v}_k)|^p dy \right.}\\
&+\displaystyle{C\frac{s}{s-r}\int_r^s\mathcal{H}^{N-1}(\gamma Q_k\cap\{|\bar{u}_k-a_k|+\tfrac{|\tau_{L_k}({\bar v}_k)- {\bar v}_k|}{\lambda}=t\})\,dt }\displaystyle{+C\int_{\gamma Q_k\cap\{r<|\bar{u}_k-a_k|+ \frac{|\tau_{L_k}({\bar v}_k)- {\bar v}_k|}{\lambda}<s\}}|\nabla\bar{u}_k| \,dy }\\
&\displaystyle{+C\left.\int_{\gamma Q_k\cap\{r<|\bar{u}_k-a_k|+ \frac{|\tau_{L_k}({\bar v}_k)- {\bar v}_k|}{\lambda}<s\}}|\tau_{L_k}({\bar v}_k)- {\bar v}_k|^p+ |\tau_{L_k}({\bar v}_k)|^p,dy \right\}}
\end{align*}
By condition a) above
$$
\displaystyle{\limsup_{k\rightarrow +\infty}\frac{1}{|Du|(Q_k)}\int_{\gamma Q_k}f(y,\bar{u}_k,\bar{v}_k,\nabla \bar{u}_k)\,dy\le \mu^c(x_0)}.$$ 
Moreover, for fixed $k$, for every $\lambda$, and for almost every $s$,
\begin{equation*}
\lim_{r\rightarrow s}\frac{s}{s-r}\int_r^s\mathcal{H}^{N-1}\left(\gamma Q_k)\cap\left\{|\bar{u}_k-a_k|+ \tfrac{|\tau_{L_k}({\bar v}_k)|}{\lambda}=t\right\}\right)\,dt=s\mathcal{H}^{N-1}\left(\gamma Q_k\cap\left\{|\bar{u}_k-a_k|+ \tfrac{|\tau_{L_k}({\bar v}_k)-{\bar v}_k|}{\lambda}
=s\right\}\right)
\end{equation*}
and
\begin{align*}
&\lim_{r \to s}\int_{\gamma Q_k\cap\{r<|\bar{u}_k-a_k|+ \tfrac{|\tau_{L_k}({\bar v}_k)- {\bar v}_k|}{\lambda}
<s\}}|\nabla\bar{u}_k| \,dy\\
&=\lim_{r\to s}\int_{\gamma Q_k\cap\{r<|\bar{u}_k-a_k|+ \tfrac{|\tau_{L_k}({\bar v}_k)- {\bar v}_k|}{\lambda}
<s\}}(|\tau_{L_k}({\bar v}_k)- {\bar v}_k|^p+ |\tau_{L_k}({\bar v}_k)|^p)\,dy =0.
\end{align*}
Then, for each $k$, we can choose $\frac{2}{n^2}\le r_k<s_k \le\frac{1}{2n}$ such that
\begin{equation*}
\frac{1}{|Du|(Q_k)}\int_{\gamma Q_k\cap\{r_k<|\bar{u}_k-a_k|+\tfrac{|\tau_{L_k}({\bar v}_k)-{\bar v}_k|}{\lambda}
<s_k\}}|\nabla\bar{u}_k| \,dy\le \frac{\varepsilon_k^N}{|Du|(Q_k)},
\end{equation*}
and we choose $\lambda_k$ such that 
\begin{equation}
\label{estimateLambda}
\displaystyle{
\int_Q \tfrac{\nabla(|\tau_{L_k}({\bar v}_k)-{\bar v}_k|)}{\lambda_k}dy\leq C}
\end{equation}
for a fixed constant $C$ and, making use of Lemma 2.12 in \cite{FM2}, we observe that
\begin{align*}
&\displaystyle{\frac{1}{|Du|(Q_k)}
	\frac{s}{s-r}\int_r^s\mathcal{H}^{N-1}(\gamma Q_k)\cap\left\{|\bar{u}_k-a_k|+ \tfrac{|\tau_{L_k}({\bar v}_k)- {\bar v}_k|}{\lambda_k}=t\right\})\,dt }\\
&\displaystyle{\le\frac{1}{|Du|(Q_k)}
	\frac{c}{\ln(n)}\int_{\gamma Q_k\cap\{|\bar{u}_k-a_k|+ \tfrac{|\tau_{L_k}({\bar v}_k)- {\bar v}_k|}{\lambda_k}\le \frac{1}{2n}\}}\nabla\left(|\bar{u}_k|+ \tfrac{|\tau_{L_k}({\bar v}_k)- {\bar v}_k|}{\lambda_k}\right) \,dy.}
\end{align*}
We can estimate the last expression by \eqref{estimateLambda} and arguing as in \cite[(4.17)]{FM2}
\begin{align*}
&\displaystyle{\frac{1}{|Du|(Q_k)}\frac{c}{\ln(n)}\int_{\gamma Q_k\cap\{|\bar{u}_k-a_k|+\tfrac{|\tau_{L_k}({\bar v}_k)-{\bar v}_k|}{\lambda_k}\le \tfrac{1}{n}\}}|\nabla\bar{u}_k| \,dy}
\\
&\le \displaystyle{\frac{1}{|Du|(Q_k)}\frac{c}{\ln(n)}\int_{\gamma Q_k\cap\{|\bar{u}_k-a_k|+\tfrac{|\tau_{L_k}({\bar v}_k)-{\bar v}_k|}{\lambda_k}\le \tfrac{1}{n}\}}f(y,\bar{u}_k,\bar{v}_k,\nabla\bar{u}_k)\,dy.}
\end{align*}
Moreover, the $p$-equiintegrability of $\{\tau_{L_k}({\bar v}_k) \}$ guarantees that
$$
\displaystyle{\int_{\gamma Q_k\cap\{|\bar{u}_k-a_k|
+ \tfrac{|\tau_{L_k}({\bar v}_k) - {\bar v}_k|}{\lambda_k}
\geq s_k\}} |\tau_{L_k}({\bar v}_k)|^p\,dy =O\left(\frac{1}{n}\right).}
$$
and from condition a) in Step 1 it follows (\ref{eqwithtildeu}) as we claimed.
\bigskip

To achieve Step 2 it remains to prove (\ref{secondeqwithak}). By a change of variables this is equivalent to prove
$$\lim_{k\rightarrow +\infty}\frac{\varepsilon_k^{N-1}}{|Du|(Q_k)}\int_Q\left|\tilde{u}_k(x_0+\varepsilon_k z)-a_k-\left(u(x_0+\varepsilon_k z)-\frac{1}{|Q_k|}\int_{Q_k}u\,dy  \right)\right|\,dz=0$$ which can be written $\Vert\hat{u}_k-\bar{w}_k\Vert_{L^1(Q)}\rightarrow 0$ if we introduce the functions
$$
\hat{u}_k(z):=\frac{\varepsilon_k^{N-1}}{|Du|(Q_k)}\left[u(x_0+\varepsilon_k z)-\frac{1}{|Q_k|}\int_{Q_k}u(y)\,dy\right],$$
\begin{equation}
\label{wkhat}
\bar{w}_k(z):=\tfrac{\varepsilon_k^{N-1}}{|Du|(Q_k)}\left(\tilde{u}_k(x_0+\varepsilon_k z)-a_k\right), \;
w_k(z):=\tfrac{\varepsilon_k^{N-1}}{|Du|(Q_k)}\left(\bar{u}_k(x_0+\varepsilon_k z)-a_k\right).
\end{equation} 
Thus we have
\begin{align*}
&\Vert\hat{u}_k-\bar{w}_k\Vert_{L^1} \leq \displaystyle{\Vert\hat{u}_k-w_k\Vert_{L^1} +\tfrac{\varepsilon_k^{N-1}}{|Du|(Q_k)}\int_Q|\bar{u}_k(x_0+\varepsilon_k z)-\tilde{u}_k(x_0+\varepsilon_k z)|\,dz}\\
&\displaystyle{=\Vert\hat{u}_k-w_k\Vert_{L^1}+ \tfrac{\varepsilon_k^{N-1}}{|Du|(Q_k)}\times}\\
&\displaystyle{ \times\int_Q\left|\bar{u}_k(x_0+\varepsilon_k z)-a_k)(1-\varphi_k\left(\left|\bar{u}_k(x_0+\varepsilon_k z)-a_k\right|+ \tfrac{|\tau_{L_k} {\bar v}_k(x_0+ \varepsilon_k z)- {\bar v}_k(x_0+ \varepsilon_k z)|}{\lambda_k}\right)\right|\,dz}\\
&\le \displaystyle{\Vert\hat{u}_k-w_k\Vert_{L^1} +\frac{\varepsilon_k^{N-1}}{|Du|(Q_k)}\times}\\
&\displaystyle{\times\int_{\{y\in Q:\ |\bar{u}_k(x_0+\varepsilon_k y)-a_k|+ \tfrac{|\tau_{L_k}\circ {\bar v}_k(x_0+ \varepsilon_k z)- {\bar v}_k(x_0+ \varepsilon_k z)|}{\lambda_k}\ge r_k\}}|\bar{u}_k(x_0+\varepsilon_k z)-a_k|\,dz}\\
&\le \displaystyle{\Vert\hat{u}_k-w_k\Vert_{L^1} +\int_{\{y\in Q:\ |\bar{u}_k(x_0+\varepsilon_k y)-a_k|+\tfrac{|\tau_{L_k}\circ {\bar v}_k(x_0+ \varepsilon_k z)- {\bar v}_k(x_0+ \varepsilon_k z)|}{\lambda_k}\ge r_k\}}|w_k(z)|\,dz.}
\end{align*}

Observe that $\Vert\hat{u}_k-w_k\Vert_{L^1}\rightarrow0$. Indeed it is exactly condition c) in Step 1, if we make the evident change of variables.

For the second term, we start by proving that $\{w_k\}$ is equi-integrable.
Indeed, by the definition of total variation of the BV function $\hat{u}_k$, it is clear that $|D\hat{u}_k|(Q)=1$. Moreover, since $\int_Q\hat{u}_k\,dz=0$, using Poincar\'{e} inequality (cf. \cite[Theorem 3.44]{AFP}) we deduce that $\{\hat{u}_k\}$ is bounded in $L^1$. Therefore the compactness of $BV$ in $L^1$ (cf \cite[Theorem 3.23]{AFP}) implies that $\{\hat{u}_k\}$ is equi-integrable. Then adding the fact that $\Vert\hat{u}_k-w_k\Vert_{L^1}\rightarrow 0$ as $ k\to +\infty$, we get that $\{w_k\}$ is equi-integrable as desired. It remains to prove that
$\left|\left\{y\in Q:\ |\bar{u}_k(x_0+\varepsilon_k y)-a_k|+\tfrac{|\tau_{L_k}\circ {\bar v}_k(x_0+ \varepsilon_k y)- {\bar v}_k(x_0+ \varepsilon_k y)|}{\lambda_k}\ge r_k  \right\}\right|\rightarrow 0$ as $k\rightarrow +\infty$
to obtain $\int_{\{y\in Q:\ |\bar{u}_k(x_0+\varepsilon_k y)-a_k|+\tfrac{|\tau_{L_k}\circ {\bar v}_k(x_0+ \varepsilon_k y)- {\bar v}_k(x_0+ \varepsilon_k y)|}{\lambda_k}\ge r_k\}}|w_k(y)|\,dy\rightarrow 0$ and thus (\ref{secondeqwithak}). Indeed, since $a_k\rightarrow u(x_0)$,  and $\tfrac{|\tau_{L_k}\circ {\bar v}_k(x_0+ \varepsilon_k y)- {\bar v}_k(x_0+ \varepsilon_k y)|}{\lambda_k} \to 0$, for a.e. $y\in Q$ for sufficiently large $k$
\begin{align*}
&\displaystyle{\left|\left\{y\in Q:\ |\bar{u}_k(x_0+\varepsilon_k y)-a_k|+\tfrac{|\tau_{L_k}\circ {\bar v}_k(x_0+ \varepsilon_k y)- {\bar v}_k(x_0+ \varepsilon_k y)|}{\lambda_k}\ge r_k  \right\}\right| }\\
&\displaystyle{\le  \left|\left\{y\in Q:\ |\bar{u}_k(x_0+\varepsilon_k y)-u(x_0)|+\tfrac{|\tau_{L_k}\circ {\bar v}_k(x_0+ \varepsilon_k y)- {\bar v}_k(x_0+ \varepsilon_k y)|}{\lambda_k}\ge \tfrac{1}{n^2}  \right\}\right|}\\
&\le \displaystyle{\int_Q n^2 |\bar{u}_k(x_0+\varepsilon_k y)-u(x_0)|+\tfrac{|\tau_{L_k}\circ {\bar v}_k(x_0+ \varepsilon_k y)- {\bar v}_k(x_0+ \varepsilon_k y)|}{\lambda_k} \,dy }\\
&=\displaystyle{ \frac{n^2}{|Q_k|}\int_{Q_k} \left(|\bar{u}_k(z)-u(x_0)|+\tfrac{|\tau_{L_k}\circ {\bar v}_k(z)- {\bar v}_k(z)|}{\lambda_k}\right)\,dz, }
\end{align*}
the result following from condition b) and the arbitrariness of $\lambda_k$ and \cite[Lemma 8.13]{FL}. 
\smallskip

\noindent{\bf Step 3.} Notice that, defining \begin{equation}
\label{thetak}
\theta_k:=\frac{|Du|(Q_k)}{\varepsilon_k^N},
\end{equation} and recalling the definition of $\bar{w}_k$, (\ref{eqwithtildeu}) can be written as
\begin{equation}\label{eqwithmuk}\left(1+\omega\left(\tfrac{1}{n}\right)\right)\,\mu^c(x_0)\ge\limsup_{k\rightarrow +\infty}\frac{1}{\mu_k}\int_{\gamma Q}f(x_0,u(x_0),\tilde{v}_k(x_0+\varepsilon_k z),\theta_k\nabla\bar{w}_k(z))\,dz.\end{equation}

Let $V_k(z):=\theta_k^{-\frac{1}{p}}\tilde{v}_k(x_0+\varepsilon_k z).$
By $d)$ 
it results that $V_k \rightharpoonup 0 \hbox{ in }L^p(Q;\mathbb R^m),$ and
\begin{equation*}
\left(1+\omega\left(\tfrac{1}{n}\right)\right)\,\mu^c(x_0)\ge\limsup_{k\rightarrow +\infty}\frac{1}{\theta_k}\int_{\gamma Q}f(x_0,u(x_0),\theta_k^{\frac{1}{p}}V_k(z),\theta_k\nabla\bar{w}_k(z))\,dz.
\end{equation*}
Then, modifying $\{V_k\}$ and $\{\bar{w}_k\}$, we get new sequences $\{\tilde{V}_k\}$ and $\{\tilde{w}_k\}$ in order apply the convexity-quasiconvexity of $f$. In fact we will need to work on the boundary of an inner cube $\tau Q$, $\tau\in (t,\gamma)$, and the sequences will be modified in a layer $\tau Q\setminus\tau(1-\delta)Q$. 

To be precise, we claim that it is possible to define  $\tilde{V}_k\rightharpoonup 0$ in $L^p(\tau Q;\mathbb R^m)$, $\int_{\tau Q}\tilde{V}_kdz=0$, and $\tilde{w}_k(x)=A+\varphi_k(x)$ for some $\varphi_k\in W^{1,\infty}_{per}(\tau Q;\mathbb{R}^d)$ and such that
$$\left(1+\omega\left(\tfrac{1}{n}\right)\right)\mu^c(x_0)\ge \lim_{k\rightarrow +\infty}\frac{1}{\theta_k}\int_{\tau Q} f(x_0,u(x_0),\theta_k^{\frac{1}{p}}\tilde{V}_k(z),\theta_k \nabla\tilde{w}_k(z))\,dz +\Lambda(1-t),$$ for some continuous function $\Lambda:[0,1]\rightarrow \mathbb{R}$ with $\Lambda(0)=0$.

We start choosing the function into which $\bar{w}_k$ will be modified. We will show that there is a sequence $\{\xi_k\}$ of smooth functions depending only on $x_N$ such that
\begin{equation}\label{limitxiandhatu}
\Vert\xi_k-\hat{u}_k\Vert_{L^1}\rightarrow 0\quad\text{and}\quad\nabla\xi_k(\tau Q)-D\hat{u}_k(\tau Q)\rightarrow 0\ \text{a.e.}\ \tau\in (0,1) \hbox{ as }k \to +\infty.\end{equation}
The procedure is to average $\hat{u}_k$ in $x_1,...,x_{N-1}$ and regularize the function obtained as follows. Let $\eta_k(x_N):=\int_{Q'}\hat{u}_k(x',x_N)\,dx'$ where $Q':=(-1/2,1/2)^{N-1}$ and $x':=(x_1,...,x_{N-1})$. Define $\zeta_k(x_N):=(\eta_k\ast\rho_k)(x_N)$ for some mollifying function $\rho_k$ such that $\Vert\zeta_k-\eta_k\Vert_{L^1((-\frac{1}{2},\frac{1}{2}))}\le\frac{1}{k}$ and define $\xi_k(x)=\zeta_k(x_N)$. Then
$$
\Vert\xi_k-\hat{u}_k\Vert_{L^1(Q)}  \le  \Vert\zeta_k-\eta_k\Vert_{L^1((-\frac{1}{2},\frac{1}{2}))}+\Vert\eta_k-\hat{u}_k\Vert_{L^1(Q)},
$$
where we have identified $\eta_k$ with its natural extension to Q. By the choice of $\zeta_k, \,\lim_{k\rightarrow +\infty}\Vert\zeta_k-\eta_k\Vert_{L^1}=0$ and for the other term we have, using Poincar\'e inequality,
$$
\displaystyle{\Vert\eta_k-\hat{u}_k\Vert_{L^1((-\frac{1}{2},\frac{1}{2}))} \le \int_Q\left|\int_{Q'}\hat{u}_k(x',z_N)\,dx'-\hat{u}_k(z)\right|\,dz}\leq\int_{-1/2}^{1/2}c\left|D_{x'}\hat{u}_k(\cdot,z_N)\right|(Q')\,dz_N
$$

By definition of $\hat{u}_k$, doing the natural change of variables one get
\begin{equation}D\hat{u}_k(Q)=\frac{Du(Q_k)}{|Du|(Q_k)},\label{characthatu}\end{equation} which, accordingly to (\ref{limitA}) converges to the matrix $A=a\otimes e_N$. Thus we are in conditions to apply Proposition A.1 in \cite{FM2} and obtain $|D\hat{u}_k-(D\hat{u}_k\cdot A)A|\rightarrow 0$. In particular, for $i=1,...,N-1$, $|D\hat{u}_k\,e_i|(Q)=|(D\hat{u}_k-(D\hat{u}_k\cdot A)A)e_i+(D\hat{u}_k\cdot A)A\,e_i|(Q)\rightarrow 0$.

Now we choose the layer where we will change the sequences. Let $\tau\in(t,\gamma)$ be such that $\nabla\xi_k(\tau Q)-D\hat{u}_k(\tau Q)\rightarrow 0$, choose $\delta>0$ such that $(1-\delta)\tau>t$ and \begin{equation}\label{estimategradientxi}|\nabla\xi_k|(\tau Q\setminus \tau(1-\delta)Q)\le |D\hat{u}_k|(Q\setminus tQ)=\frac{|Du|((Q_k)\setminus(tQ_k))}{|Du|(Q_k)}.
\end{equation}
Notice that $\nabla\bar{w}_k(z)=\frac{1}{\theta_k}\nabla \tilde{u}(x_0+\varepsilon_k z).$ Then, (\ref{eqwithtildeu}), $(H_1)_p$, and the second limit in (\ref{eq4Cantorlower})  imply that $\nabla \bar{w}_k$ is bounded in $L^1(\gamma Q)$. In particular we can say
\begin{equation}
\label{wbark}
\int_{\tau Q\setminus \tau(1-\delta)Q}(|V_k|^p+|\nabla \bar{w}_k|)\,dz\le C,\ \forall\ k.
\end{equation}

Then we use the slicing method as in the proof of Lemma \ref{3.1FMr}, replacing the cube $Q$ therein by $\tau Q$. Thus we divide for every $j\in \mathbb N$, $\tau Q\setminus \tau(1-\delta)Q$ into $j$ layers,
getting recursively a sequence $k(j)$, layers $S_j:=\{z\in \tau Q\setminus \tau(1-\delta)Q:\ \alpha_j<{\rm dist}(z,\partial (\tau Q))<\beta_j\}$
and  cut-off functions $\eta_j$ on $\tau Q$ such that
\begin{equation*}
\int_{S_j}(|V_{k(j)}|^p+|\nabla\bar{w}_{k(j)}|)\,dz\le \frac{C}{j},\, \frac{1}{|S_j|}\int_{S_j}|\bar{w}_{k(j)}-\xi_{k(j)}|\,dz\le \frac{1}{j}, \hbox{ and }
\frac{\displaystyle{\left|\frac{1}{|\tau Q|}\int_{\tau Q}-\eta_jV_{k(j)}\,dz\right|}}{\displaystyle{\left|\frac{1}{|\tau Q|}\int_{\tau Q}(1-\eta_j)\,dz\right|}}\le 1.
\end{equation*}
Now, define $$\tilde{V}_j(z):=(1-\eta_j(z))\frac{\displaystyle{\frac{1}{|\tau Q|}\int_{\tau Q}-\eta_jV_{k(j)}\,dz}}{\displaystyle{\frac{1}{|\tau Q|}\int_{\tau Q}(1-\eta_j)\,dx}}+\eta_j(z)V_{k(j)}(z)$$
and 
\begin{equation}
\label{wtildej}
\tilde{w}_j(z):=(1-\eta_j(z))\xi_{k(j)}(z)+\eta_j(z)\bar{w}_{k(j)}(z).
\end{equation}
By (\ref{eqwithmuk}), adding and subtracting $f(x_0,u(x_0),\theta_{k(j)}^{\frac{1}{p}}\tilde{V}_j(z),\theta_{k(j)}\nabla\tilde{w}_j(z))$  inside the integral, having in mind the definition of $\eta_j$ and using $(H_1)_p$ , we get

\begin{align*}
&\displaystyle{\left(1+\omega\left(\tfrac{1}{n}\right)\right)\,\mu^c(x_0)\ge\limsup_{j\rightarrow +\infty}\frac{1}{\theta_{k(j)}}\int_{\tau Q}f(x_0,u(x_0),\theta_{k(j)}^{\frac{1}{p}} V_{k(j)}(z),\theta_{k(j)}\nabla\bar{w}_{k(j)}(z))\,dz}\\ 
&\ge \displaystyle{\limsup_{j\rightarrow +\infty}\frac{1}{\theta_{k(j)}}\left\{\int_{\tau Q}f(x_0,u(x_0),\theta_{k(j)}^{\frac{1}{p}} \tilde{V}_{j},\theta_{k(j)}\nabla\tilde{w}_j)\,dz\right.}\\
&\displaystyle{\left.\,\,\,-\int_{\{x\in \tau Q:{\rm dist}(x,\partial(\tau Q))\le\beta_j\}}f(x_0,u(x_0),\theta_{k(j)}^{\frac{1}{p}}\tilde{V}_j,\theta_{k(j)}\nabla\tilde{w}_j)\,dz\right\}}\\  
&\ge \displaystyle{\limsup_{j\rightarrow +\infty}\frac{1}{\theta_{k(j)}}\int_{\tau Q}f(x_0,u(x_0),\theta_{k(j)}^{\frac{1}{p}}\tilde{V}_j,\theta_{k(j)}\nabla\tilde{w}_j)\,dz-\int_{S_j}C(|\nabla\bar{w}_{k(j)}|+ |\nabla \eta_j|\,|\bar{w}_{k(j)}-\xi_{k(j)}|)dz}\\ 
&\displaystyle{\,\,\,-\int_{S_j} C+ |V_{k(j)}|^p\,dz-\int_{\tau Q\setminus\tau(1-\delta)Q}c(1+|\nabla \xi_{k(j)}|)\,dz}\\ 
&\ge  \displaystyle{\limsup_{j\rightarrow +\infty}\frac{1}{\theta_{k(j)}}\int_{\tau Q}f(x_0,u(x_0),\theta_{k(j)}^{\frac{1}{p}}\tilde{V}_j,\theta_{k(j)}\nabla\tilde{w}_j)\,dz-\frac{c}{j}-\int_{\tau Q\setminus\tau(1-\delta)Q}c(1+|\nabla \xi_{k(j)}|)\,dz.}
\end{align*}
By (\ref{estimategradientxi}) and (\ref{eq5Cantorlower}), $\int_{\tau Q\setminus\tau(1-\delta)Q}c(1+|\nabla \xi_{k(j)}|)\,dz\le \Lambda(1-t)$ for some continuous $\Lambda:[0,1]\rightarrow \mathbb{R}$ with $\Lambda(0)=0$. Therefore we have
$$\left(1+\omega\left(\tfrac{1}{n}\right)\right)\,\mu^c(x_0)\ge \displaystyle{\limsup_{j\rightarrow +\infty}\frac{1}{\theta_{k(j)}}\int_{\tau Q}f(x_0,u(x_0),\theta_{k(j)}^\frac{1}{p}\tilde{V}_j(z),\theta_{k(j)}\nabla\tilde{w}_j(z))\,dz}-\Lambda(1-t),$$
which proves our claim up to a relabeling of the sequence.

\noindent{\bf Step 4.} Using the convexity-quasiconvexity of $f$ we will achieve in this step the desired conclusion. Indeed, as remarked above, the functions $\tilde{V}_k$ have $0$ average in $\tau Q$. On the other hand, we can always construct $\xi_k$ such that $\xi_k(x)-(\frac{\zeta_k(\frac{\tau}{2})-\zeta_k(-\frac{\tau}{2})}{\tau}\otimes e_N)\,x$ is a $\tau Q$-periodic function. This, together with the fact that $\tilde{w}_j=\xi_{k(j)}$ on $\partial (\tau Q)$, yields that $\tilde{w}_j\in (\frac{\zeta_{k(j)}(\frac{\tau}{2})-\zeta_{k(j)}(-\frac{\tau}{2})}{\tau}\otimes e_N)\,x+W^{1,\infty}_{per}(\tau Q;\mathbb R ^d)$. Therefore
$$\left(1+\omega\left(\tfrac{1}{n}\right)\right)\,\mu^c(x_0)\ge O(1-t)+\displaystyle{\limsup_{j\rightarrow +\infty}\frac{|\tau Q|}{\theta_{k(j)}}f\left(x_0,u(x_0),0,\theta_{k(j)}\tfrac{\zeta_k(\frac{\tau}{2})-\zeta_k(-\frac{\tau}{2})}{\tau}\otimes e_N\right)}.$$
If we add and subtract in the previous limit the quantity $\frac{|\tau Q|}{\theta_{k(j)}}f(x_0,u(x_0),0,\frac{\theta_{k(j)}}{|\tau Q|}A)$ we get two terms.

\noindent One gives, by definition, the expected value of the $f^\infty_p$ function, i.e.
$$\displaystyle{\lim_{j\to +\infty}\tfrac{|\tau Q|}{\theta_{k(j)}}f\left(x_0,u(x_0),0,\tfrac{\theta_{k(j)}}{|\tau Q|}A\right)=
f^\infty_p(x_0,u(x_0),0,A).}$$

\noindent The other term can be estimated using the Lipschitz continuity of $f(x_0,u(x_0),0,\cdot)$, i.e. \eqref{p-lipschitzcontinuity},  and (\ref{characthatu}).

\noindent After passing to the limit on $k$, and using \eqref{limitxiandhatu}, \eqref{eq5Cantorlower} and \eqref{limitA}, we get
$$
\left(1+\omega\left(\tfrac{1}{n}\right)\right)\,\mu^c(x_0)\geq O(1-t)+f^\infty_p(x_0,u(x_0),0, A)+\Lambda(1-t)
$$ where $\Lambda$ is a continuous function with $\Lambda(0)=0$.
We finally obtain the desired estimate letting $n\rightarrow +\infty$ and $t\rightarrow 1^{-}$.\end{proof}

\subsection{Upper bound in $BV\times L^p$}\label{ubbvlp}

In order to achieve the representation in Theorem \ref{MainResultp}, we localize our functionals. We define for open sets $A\subset\Omega$ and for any $(u,v)\in BV(\Omega;\mathbb{R}^d)\times L^p(\Omega;\mathbb{R}^m)$,
$${\cal J}_p(u,v;A):=\inf\left\{\liminf_{n\to +\infty} J(u_n,v_n;A):\ u_n\in BV(\Omega;\mathbb{R}^d),\ v_n\in L^p(\Omega;\mathbb{R}^m),\ u_n\to u \hbox{ in } L^1, v_n \rightharpoonup v  \hbox{ in } L^p\right\}$$
where, with an abuse of notation,

\begin{equation}
\label{Jext}
J(u,v;A):=\left\{\begin{array}{l}\displaystyle{\int_A f(x,u,v,\nabla u)\,dx},\ \text{if}\ (u,v)\in W^{1,1}(\Omega;\mathbb{R}^d)\times L^p(\Omega;\mathbb{R}^m),\\
+\infty,\ \text{otherwise.}\end{array}\right.
\end{equation}

We start by observing that $(H_1)_p$ implies that for every $u \in BV(\Omega;\mathbb{R}^d)$ and for every $v \in L^p(\Omega;\mathbb{R}^m)$, it results
\begin{equation*}
{\cal J}_p(u,v;A)\le C\left(|A|+|Du|(A)+\int_A |v|^p dx\right),
\end{equation*}

We observe that, arguing as in \cite[Lemma 3.5]{CRZ1}, ${\cal J}_p$ is a variational functional. This means that the following conditions hold:
\medskip

\noindent 1. ${\cal J}_p$ is local, that is ${\cal J}_p(u,v;A)={\cal J}_p(u',v';A),$ for every $A\in \mathcal{A}( \Omega)$ and every $(u,v),(u',v')\in BV(A;\mathbb{R}^d)\times L^p(A;\mathbb{R}^m)$ such that $u=u'$ and $v=v'$ a.e. in $A$;

\noindent 2. ${\cal J}_p$ is sequentially lower semicontinuous, that is
\begin{equation*}
{\cal J}_p(u,v;A)\le\liminf_{n\to +\infty}{\cal J}_p(u_n,v_n;A),\ \forall\ A\subset\Omega \text{ open, } \ u_n\to u \text{ in } L^1(A;\mathbb{R}^d)\text{ and }v_n\rightharpoonup v\text{ in }L^\infty(A;\mathbb{R}^m);
\end{equation*}
\noindent 3. ${\cal J}_p(u,v;\cdot)$ is the trace of a Radon measure restricted to the family $\mathcal{A}(\Omega).$

The following result is devoted to prove the upper bound in $BV\times L^p$, $1<p<+\infty$.

\begin{theorem}\label{upperboundL_pthm}
Let $f:\Omega
\times\mathbb{R}^d\times\mathbb{R}^{m}\times\mathbb{R}^{d\times
N}\rightarrow[ 0,+\infty)$ be a continuous function satisfying $(H_0)$, $( H_{1}) _{p}-( H_{3}) _{p}$, and ${\overline J}_p$ be defined in \eqref{relaxedp}. Then for every $(u,v)\in BV(\Omega;\mathbb{R}^d) \times L^{p}(\Omega ;\mathbb{%
R}^{m})$
\begin{align}\label{upperboundLp}
{\overline J}_p( u,v;\Omega) 
\leq\int_{\Omega}f(x,u ,v ,\nabla
u) dx +\int_{J_{u}\cap\Omega}K_p( x,0,u^{+} ,u^{-},\nu_{u}) d\mathcal{H}^{N-1}
+\int_{\Omega}f_p^{\infty}(x,u ,0,\tfrac{dD^{c}u}{d\vert D^{c}u\vert }) d\vert D^{c}u\vert .  
\end{align} 
\end{theorem}
\begin{proof}
The representation \eqref{upperboundLp} is achieved first for $(u,v)\in BV(\Omega;\mathbb R^d)\cap L^\infty(\Omega\;\mathbb R^d)\times L^\infty(\Omega;\mathbb R^m)$, then, via an approximation argument as in \cite{AMT}, the result will be obtained in $BV(\Omega;\mathbb R^d)\times L^\infty(\Omega;\mathbb R^m)$. Then a standard truncation argument (see \cite[Theorem 14]{RZ}) leads us to $BV(\Omega;\mathbb R^d)\times L^p(\Omega;\mathbb R^m).$
\smallskip 

\noindent {\bf Part 1}. Let $(u,v)\in BV(\Omega;\mathbb R^d)\cap L^\infty(\Omega\;\mathbb R^d)\times L^\infty(\Omega;\mathbb R^m)$. Since $\overline{J}_p(u,v) = {\cal J}_p(u,v; \Omega)$ and ${\cal J}_p(u,v; \cdot)$ is the trace of a Radon measure on the open subsets of $\Omega$,  absolutely continuous with respect to $|Du| + {\cal L}^N$, it will be enough to prove the following inequalities
\begin{equation}\label{absolutely-continuousUB}
\frac{d {\cal J}_p(u,v;\cdot)}{d {\cal L}^N}(x)\leq f(x,u(x),v(x),\nabla u(x)),\ \mathcal{L}^N-\hbox{a.e.}\ x\in\Omega,
\end{equation}
\begin{equation}\label{Cantor-densityUB}
\frac{d {\cal J}_p(u,v;\cdot)}{d |D^cu|}(x)\leq f^\infty_p\left(x,u(x),0,\tfrac{dD^cu}{d|D^cu|}(x)\right),\ |D^cu|-\hbox{a.e.} \,x\in\Omega.
\end{equation}
\begin{equation}\label{jump-densityUB}
{\cal J}_p(u,v;J_u\cap \Omega)\leq \int_{J_u\cap \Omega}K_p(x,0,u^-(x),u^+(x),\nu_u(x)) d {\cal H}^{N-1}.
\end{equation}
The proof of these inequalities exploits results proven in \cite{AMT}.

\noindent {\bf Bulk part.} The inequality \eqref{absolutely-continuousUB} is an immediate consequence of \cite[Theorems 12 and 14]{RZ}, observing that the same arguments therein can be applied when $u$ is a function of bounded variation.

\noindent{\bf Cantor part.} Let $u \in BV(\Omega;\mathbb R^d) \cap L^\infty(\Omega;\mathbb R^d)$ and $v \in L^\infty(\Omega;\mathbb R^m) $. We follow \cite{FM2} and \cite{FKP1}, identifying $u$ with its approximate limit defined in $\Omega\setminus J_u$. 

Let $u_n:=u\ast \rho_n$, where $\rho_n$ be a sequence of mollifiers, then by \cite[Lemma 2.5]{FM2}, 
\begin{equation}
\label{untou}
u_n(x)\rightarrow u(x),\ |D^cu|-\hbox{a.e.}\ x\in\Omega.\end{equation}
Therefore $u$ is $|D^cu|-$measurable. We write $|Du|=|D^cu|+\eta$, where $\eta$ and $|D^cu|$ are mutually singular Radon measures. Let $x_0\in\Omega$ be such that   
$$\frac{d {\cal J}_p(u,v;\cdot)}{d |D^cu|}(x_0)\text{ exists and is finite,}$$
\begin{equation}\label{FMr5.13}
\lim_{\e \to 0^+}\frac{\eta(B(x_0,\e))}{|D^cu|(B(x_0,\e))}=0,\qquad\lim_{\e \to 0^+}\frac{|Du|(B(x_0,\e))}{|D^cu|(B(x_0,\e))} \hbox{ exists and is finite,}
\end{equation}
\begin{equation}\label{FMr5.14}
\lim_{\e \to 0^+}\frac{\e^N}{|D^cu|(B(x_0,\e))}=0,
\end{equation}

\begin{equation}\label{FMP6.17new}
\lim_{\e \to 0^+}\frac{1}{|D^cu|(B(x_0,\e))}\int_{B(x_0,\e)}|v(x)|\,dx=0, \;\;\; \lim_{\e \to 0^+}\frac{1}{|D^cu|(B(x_0,\e))}\int_{B(x_0,\e)}|v(x)|^p\,dx=0,
\end{equation}

\begin{equation}\label{FMr5.16}A(x_0)=\lim_{\e\to 0}\frac{D^cu(B(x_0,\e))}{|D^cu|(B(x_0,\e))}\text{ exists and is a rank one matrix of norm one,}\end{equation}
\begin{equation}\label{FMr5.17}\lim_{\e\to 0}\frac{1}{|D^cu|(B(x_0,\e))}\int_{B(x_0,\e)}f^\infty_p(x_0,u(x_0),0,A(x))d|D^cu|=f^\infty_p(x_0,u(x_0),0,A(x_0)).
\end{equation}
Fix $\delta>0$. Using the Yosida transform of $f$ introduced in Definition \ref{defYosidap} and the properties in Proposition \ref{propYosida}, we get
\begin{align*}
{\cal J}_p(u,v;B(x_0,\e))&\leq \liminf_{n \to + \infty}\int_{B(x_0,\e)}f(x,u_n,v,\nabla u_n)\,dx\le\liminf_{n \to + \infty}\left\{\int_{B(x_0,\e)}f(x_0,u(x_0),v, \nabla u_n)\,dx\right.\\
&+\left.\int_{B(x_0,\e)} \delta (1+|v|+|\nabla u_n|)+\lambda (\e+ |u_n-u(x_0)|)(1 + |v|+|\nabla u_n|)\,dx\right\}\\
&\le\liminf_{n \to +\infty}\left\{\int_{B(x_0,\e)} f\left(x_0, u(x_0),v,(Du\ast\rho_n)\right)\,dx+ (\delta+\lambda\e) (1+\|v\|_{L^\infty})|B(x_0,\e)|\right.\\
&\left.+(\lambda \varepsilon +\delta)\int_{B(x_0,\e)}|\nabla u_n|\,dx) +\lambda C\int_{B(x_0,\e)} |u_n-u(x_0)|(1+ \|v\|_{L^\infty}|\nabla u_n|)\,dx\right\}\\
\end{align*}

An argument entirely similar to \cite[Section 5, steps 1 and 2 (page 37)]{FM2}, allows us to write
\begin{equation*}
{\cal J}_p(u,v; B(x_0,\varepsilon)) \leq 
\liminf_{\varepsilon \to 0^+}\liminf_{n\to +\infty}\frac{1}{|D^c u|(B(x_0,\varepsilon))}\int_{B(x_0,\varepsilon)} f(x_0, u(x_0), v,Du \ast \varrho_n )dx+ O(\delta).
\end{equation*}

Let $h:\mathbb R^m \times \mathbb R^{d \times N} \to [0,\infty)$ given by $h(b,\xi):=\sup_{t \geq 0}\tfrac{f(x_0, u(x_0), t^{\frac{1}{p}}b, t \xi)- f(x_0, u(x_0),0,0)}{t}.$
Then, $h$ is positively homogeneous of degree $(p,1)$ and satisfies \eqref{p-lipschitzcontinuity}. The convexity of $f^\infty_p(x_0, u(x_0),\cdot,\cdot)$ when $\xi$ is at most a rank-one matrix entails $f^\infty_p(x_0,u(x_0),b,\xi)= h(b,\xi),$
for every $(b,\xi) \in \mathbb R^m \times \mathbb R^{d\times N}$, with ${\rm rank }\xi \leq 1$.
Thus, 
\begin{align}
\label{after5.18FM2}
\frac{d {\cal J}_p(u,v;B(x_0,\varepsilon))}{ d |D^c u|(B(x_0,\varepsilon))} &= \liminf_{\varepsilon \to 0^+}\liminf_{n\to +\infty}\frac{1}{|D^cu|(B(x_0,\varepsilon))} \int_{B(x_0,\varepsilon)} h(v, Du\ast \varrho_n)dx\\
&+\limsup_{\varepsilon \to 0^+}\frac{1}{|D^cu|(B(x_0,\varepsilon))} \int_{B(x_0,\varepsilon)} f(x_0, u(x_0),0,0)dx + O(\delta). \notag 
\end{align}
Observe that, \eqref{p-lipschitzcontinuity} gives
$$
\begin{array}{ll}
\displaystyle{\liminf_{\varepsilon \to 0^+}\liminf_{n\to +\infty} \int_{B(x_0,\varepsilon)} h\left(\frac{v}{(|D^c u|(B(x_0,\varepsilon)))^{\frac{1}{p}}}, \frac{Du \ast \varrho_n}{|D^c u|(B(x_0,\varepsilon))}\right) dx}\\
\\
\displaystyle{\leq \limsup_{\varepsilon \to 0^+}\limsup_{n\to +\infty} \int_{B(x_0,\varepsilon)} h\left(0, \tfrac{Du \ast \varrho_n}{|D^c u|(B(x_0,\varepsilon))}\right) dx+\lim_{\varepsilon \to 0^+} \int_{B(x_0,\varepsilon)}\tfrac{|v|^p}{|D^c u|(B(x_0,\varepsilon))}dx }\\
\displaystyle{+\lim_{\varepsilon \to 0^+} \lim_{n\to +\infty}\int_{B(x_0,\varepsilon)}\left|\tfrac{v}{(|D^c u|(B(x_0,\varepsilon)))^{\frac{1}{p}}}\right| \left|\tfrac{Du \ast \varrho_n}{|D^c u|(B(x_0,\varepsilon))}\right|^{\frac{1}{p'}. }}
\end{array}
$$
and \eqref{FMP6.17new} guarantees that the second limit  from below is $0$. The last term can be estimated via H\"{o}lder inequality, leading to
$$
\displaystyle{\lim_{\varepsilon \to 0^+}\lim_{n\to +\infty} \left(\int_{B(x_0,\varepsilon)}\frac{|v|^p}{|D^c u|(B(x_0,\varepsilon))}dx\right)^{\frac{1}{p}}  \left(\int_{B(x_0,\varepsilon)} \frac{|D u\ast \varrho_n|}{ |D^c u|(B(x_0,\varepsilon))}dx\right)^{\frac{1}{p'}}.}
$$ 
The first term of the above product is null by \eqref{FMP6.17new}, while the latter, exploiting \cite[Lemma 4.5]{AMT} (see also \cite[Lemma 2.5]{FM2}) becomes
$$
\displaystyle{\left(\lim_{\varepsilon \to 0^+}\int_{B(x_0,\varepsilon)} \frac{|Du|}{|D^c u |(B(x_0,\varepsilon))}\right)^{\frac{1}{p'}}}
$$
which is finite by \eqref{FMr5.13}.
Thus, from, \eqref{after5.18FM2} we can conclude that 
$$
\begin{array}{ll}
\displaystyle{\frac{d {\cal J}_p(u,v; B(x_0,\varepsilon))}{ d |D^c u|(B(x_0,\varepsilon))} \leq \limsup_{\varepsilon \to 0^+}\limsup_{n\to +\infty}\frac{1}{|D^cu|(B(x_0,\varepsilon))} \int_{B(x_0,\varepsilon)} h(0, Du\ast \varrho_n)dx + O(\delta).}
\end{array}
$$
Then the thesis is achieved via the same arguments in \cite{FM2}, \eqref{FMr5.17} and letting $\delta \to 0^+$.

\noindent{\bf Jump part.} We show that

\begin{equation}\label{5.19FMrinfty}
\displaystyle{{\cal J}_p(u,v; J_u\cap \Omega)\leq \int_{J_u\cap\Omega} K_p(x,0, u^-, u^+, \nu_u))d {\cal H}^{N-1}},
\end{equation}
for every $u \in BV(\Omega;\mathbb R^d)\cap L^\infty(\Omega;\mathbb R^d)\times L^\infty(\Omega;\mathbb R^m)$. 

The proof of \eqref{5.19FMrinfty} develops exploiting the arguments in \cite[Proposition 4.8]{AMT}, \cite[Lemma 4.2]{FR} and \cite[Proposition 4.1]{BF} and it is divided into three parts according to the limit function $u$.\smallskip

\noindent{\it Case 1.} $u(x)= c \chi_E+ d (1-\chi_E)$ with ${\rm Per}(E;\Omega)< \infty.$ 

\noindent{\it Case 2.} $u(x)=\sum c_i \chi_{E_i}(x)$, where $\{E_i\}_{i=1}^\infty$ forms a partitions of $\Omega$ into sets of finite perimeter.

\noindent{\it Case 3.} $u\in BV(\Omega;\mathbb R^d)\cap L^\infty(\Omega;\mathbb R^d)$.

\smallskip
\noindent {\it Proof of Case 1.} We start to consider $u:=c \chi_E + d (1-\chi_E)$, with ${\rm Per}_\Omega(E)<+\infty$, and $v \in L^\infty(\Omega;\mathbb R^m)$ and we aim to prove that
\begin{equation}
\label{5.20FM2}
\displaystyle{{\cal J}_p(u,v; A)\leq \int_{A}f(x,u,v,0)dx +\int_{J_u\cap A}K_p(x,0,c,d,\nu_u)d {\cal H}^{N-1},\, \text{for every}\, A \in \mathcal A(\Omega).}
\end{equation}
This proof is divided into several steps. 

\noindent{\bf Step 1.} First we assume that $u$ has a planar interface, i.e.
let $\nu \in S^{N-1}$, $a_0 \in \mathbb R^N$, consider $A= a_0 +\lambda Q_\nu$, an open cube centered at $a_0$, with two faces orthogonal to $\nu$, with side length $\lambda$, and let
$$
u(x):=\left\{\begin{array}{ll}
c &\hbox{ if } (x-a_0) \cdot \nu >0,\\
d &\hbox{ if } (x-a_0) \cdot \nu \leq 0.
\end{array}
\right.
$$
We start to consider the case where $f$ does not depend on $x$ and we claim that there exist a sequence $\{u_n\} \subset W^{1,1}(a_0+\lambda Q_\nu;\mathbb R^d)$ such that
$$
\displaystyle{u_n=\left\{
\begin{array}{ll}
c &\hbox{ if }x \cdot \nu=-\frac{\lambda}{2},\\
d &\hbox{ if }x \cdot \nu= \frac{\lambda}{2}, 
\end{array}
\right.}
$$
$u_n(x)= u_n(x+ k \lambda \nu_i), i=1,\dots N-1, k \in \mathbb Z$, where $\{\nu_1,\dots,\nu_{N-1},\nu \}$ is an orthonormal basis of $\mathbb R^N$,
and a sequence $\{v_n\} \subset L^p(a_0+\lambda Q_\nu;\mathbb R^m)$, such that $v_n(x)= v(x)$ if $|(x-a_0)\cdot \nu|> \frac{\lambda}{2(2n+1)}$, with 
$u_n \to u $ in $L^1(a_0+\lambda Q_\nu;\mathbb R^d)$, $v_n \rightharpoonup v$ in $L^p(a_0+\lambda Q_\nu;\mathbb R^m)$
and
\begin{equation}
\label{eqjump1}
\displaystyle{\lim_{n \to +\infty}\int_{a_0+ \lambda Q_\nu}f( u_n,v_n, \nabla u_n)dx = \int_{a_0+ \lambda Q_\nu} f( u,v,0)dx + \lambda^{N-1} K_p(0,c,d,\nu).}
\end{equation}

\noindent {\bf Step 1 a).} We first consider the case $a_0=0$ and $\lambda=1$ and without loss of generality we assume that $\nu=e_N$. We claim that for all $\xi \in {\cal A}(c,d,e_N)$ and for all $\varphi \in L^p(Q;\mathbb R^m)$,with $\int_Q\varphi dx =0$, there exists $\xi_n \in {\cal A}(c,d,e_N)$ and 
$v_n\in L^p(Q;\mathbb R^m)$ such that $v_n(x)=v(x)$ if $|x_N| >\frac{1}{2(2n+1)}$ 
\begin{equation}
\label{xinvn}
\displaystyle{\|\xi_n-u\|_{L^1(Q;\mathbb R^d)}\to 0,  \;\;\;v_n \rightharpoonup v \hbox{ in }L^p(Q;\mathbb R^m) \hbox{ as }n \to +\infty,}
\end{equation}
and
\begin{equation}
\label{eq2jumptoprove}
\displaystyle{\lim_{n\to +\infty}\int_Q f(\xi_n,  v_n,\nabla \xi_n)dx =\int_Q f(u,v, 0 )dx+ \int_Q f^\infty_p(\xi, \varphi,\nabla \xi)dx. }
\end{equation}

\noindent Let $\Sigma:=\{x\in Q: x_N=0\}$. For $k \in \mathbb N$, we label the elements of $(\mathbb Z\cap [-k,k]^N)\times \{0\}$ by $\{a_i\}_{i=1}^{{2k+1}^{N-1}}$ and we observe
$$
(2k+1)\overline{\Sigma}=\bigcup_{i=1}^{(2k+1)^{N-1}} (a_i+\overline{\Sigma}),
$$
with $(a_i+ \Sigma)\cap (a_j+\Sigma)= \emptyset \hbox{ if }i \not= j.$
Extending $\xi(\cdot, x_N) \to \mathbb R^{N-1}$ by periodicity we define
\begin{equation}
\label{xi2k+1}
\xi_{2k+1}(x):=\left\{
\begin{array}{ll}
c &\hbox{ if }x_N > \frac{1}{2(2k+1)},\\
\xi((2k+1)x )&\hbox{ if } |x_N|\leq \frac{1}{2(2k+1)},\\
d &\hbox{ if }x_N <-\frac{1}{2(2k+1)}. 
\end{array}
\right.
\end{equation}
Clearly $\xi_{2k+1}\in {\cal A}(c,d,e_N)$ and $\|\xi_{2k+1}-u\|_{L^1(Q;\mathbb R^d)}\to 0$ as $k \to +\infty$ (see proof of \cite[Lemma 4.2]{FM2}).
Extending $\varphi(\cdot, x_N)$ to $\mathbb R^{N-1}$ by periodicity define
$$
v_{2k+1}(x):=\left\{
\begin{array}{ll}
v(x) &\hbox{ if }|x_N| > \frac{1}{2(2k+1)},\\
(2 k+1)^{\frac{1}{p}}\varphi((2k+1)x )&\hbox{ if } |x_N|\leq \frac{1}{2(2k+1)}.
\end{array}
\right.
$$
\noindent We observe that $v_{2k+1} \rightharpoonup v$ in $L^p(Q;\mathbb R^m)$. Indeed, there exists $C>0$ such that for every $k \in \mathbb N$,
\begin{align*}
\int_Q |v_{2k+1}|^p dx&\leq \int_{\Sigma}\int_{|x_N|\geq \frac{1}{2 (2k+1)}}|v|^p dx +  \int_{\Sigma}\int_{|x_N|\leq \frac{1}{2(2k+1)}}(2k+1)|\varphi((2k+1)x)|^p dx\\
&\leq C + \displaystyle{\int_{\Sigma}\int_{-\frac{1}{2}}^{\frac{1}{2}}|\varphi((2k+1)x',x_N)|^p dx = C + \int_{\Sigma}\int_{-\frac{1}{2}}^{\frac{1}{2}}|\varphi(x',x_N)|^pdx \leq C}, 
\end{align*}
where the periodicity of $\varphi$ has been exploited. 
In order to achieve the weak convergence of $\{v_{2k+1}\}$ to $v$ it is enough to prove that $\lim_{k \to +\infty}\int_E v_{2k+1}dx = \int_E v dx $ for every  $E\subseteq Q$ (see \cite[Corollary 2.49]{FL}).
In fact,
\begin{equation*}
\displaystyle{\int_E(v_{2k+1}- v)dx= 
-\int_{\{x \in E:|x_N|< \frac{1}{2(2k+1)}\}}v\,dx
+ \int_{\{x \in E:|x_N|< \frac{1}{2(2k+1)}\}}(2k+1)^{\frac{1}{p}} \varphi((2k+1)x)dx.}
\end{equation*}
The first integral trivially converges to $0$ as $k \to +\infty$, while, concerning the second one, 
$$
\begin{array}{ll}
\displaystyle{\left|\int_{\{x \in E:|x_N|< \frac{1}{2(2k+1)}\}}(2k+1)^{\frac{1}{p}} \varphi((2k+1)x)
dx\right|\leq\frac{{(2k+1)}^{\frac{1}{p}}}{2k+1}\int_{\Sigma}\int_{-\frac{1}{2}}^{\frac{1}{2}} |\varphi(x',x_N)|dx,}
\end{array}
$$  
which, using periodicity of $\varphi$ and letting $k\to +\infty$ converges to 0.

\noindent Consider
$$
\begin{array}{ll}
\displaystyle{\int_Q f(\xi_{2k+1}, v_{2k+1},\nabla \xi_{2k+1})dx =\int_{\Sigma}\int_{-\frac{1}{2}}^{-\frac{1}{2(2k+1)}}f(c,v(x),0)dx+\int_{\Sigma}\int_{\frac{1}{2(2k+1)}}^\frac{1}{2}f(d, v(x),0)dx}\\
\displaystyle{ + \int_{\Sigma}\int_{|x_N|<\frac{1}{2(2k+1)}} f(\xi((2k+1)x), (2k+1)^{\frac{1}{p}}\varphi((2k+1)x), (2k+1)\nabla \xi((2k+1)x)dx.}
\end{array}
$$
The first two integrals in the right hand side, converge as $k \to +\infty$, to $\int_{Q} f(u(x),v(x),0)dx.$

The latter integral, after a change of variables becomes
$$
\begin{array}{ll}
\displaystyle{\int_{\Sigma}\int_{|x_N|<\frac{1}{2(2k+1)}} f(\xi((2k+1)x), (2k+1)^{\frac{1}{p}}\varphi((2k+1)x), (2k+1)\nabla \xi((2k+1)x))dx}\\

\displaystyle{= \frac{1}{2k+1}\int_Q f(\xi(y), (2k+1)^{\frac{1}{p}} \varphi(y), (2k+1)\nabla \xi(y))dy \to \int_{Q}f^{\infty}_p(\xi(y), \varphi(y), \nabla \xi(y))dy}
\end{array}
$$
as $k \to +\infty$. 
Putting together the last to limits we obtain \eqref{eq2jumptoprove}.

\smallskip
\noindent {\bf Step 1 b).} Let $\{(\eta_n, \varphi_n )\} \subset {\cal A}(c,d, e_N)\times L^p(Q;\mathbb R^m)$ with $\int_Q \varphi_n dy=0$ be a minimizing sequence for $K_p(0,c,d,\e_N)$.

Observe that since $K_p(0,c,d,e_N)$ is finite and $f^\infty_p$ satisfies \eqref{finftypgrowth}, then we can assume that $\{\varphi_n\}$ is bounded in $L^p(Q;\mathbb R^m)$. 
 
By \eqref{eq2jumptoprove}, for every $n \in \mathbb{N}$ we can find $k_n\in N$,  $u_{n} \in {\cal A}(c,d, e_N)$ and $v_{n} \in L^p(Q;\mathbb R^m)$ such that
$\|u_n- u\|_{L^1(Q;\mathbb R^d)}< \frac{1}{n},$ 
$\left|\int_Q (v_n- v)\psi_l dx\right|<\frac{1}{n}$, (for $l=1,\dots,n$) and $\{\psi_l\}$ a dense sequence of functions in $L^q(Q;\mathbb R^m)$, with
$$
v_{n}(x):= 
\left\{
\begin{array}{ll}
v(x)  &\hbox{ if } |x_N|> \frac{1}{2(2k_n+1)},\\
(2k_n+ 1)^{\frac{1}{p}}\varphi_n((2k_n+1)x)) &\hbox{ if } |x_N| \leq \frac{1}{2(2k_n+1)},
\end{array}
\right.
$$
and
$$
\displaystyle{\left|\int_Q f(u_n, v_n, \nabla u_n)dx- \int_Q f(u(x), v(x),0)dx - \int_Q f^\infty_p(\eta_n, \varphi_n, \nabla \eta_n)dx\right|<\frac{1}{n}}.
$$

\noindent By the lower bound inequality and the last estimate  we have 
\eqref{eqjump1}, up to a relabeling of the sequences $\{u_n\}$ and $\{v_n\}$ with the same indices $k_n$, when $\lambda=1$ and $a_0=0$.

Now we consider the case of $A:=\lambda Q$, for $\lambda >0$.
Define
\begin{equation}
\label{u0lambda}
 f_\lambda(u,v,\xi):= f\left(u,v, \tfrac{\xi}{\lambda}\right),\,
u_0:=\left\{
\begin{array}{ll}
c &\hbox{ if }x_N>0,\\
d &\hbox{ if }x_N\leq 0,
\end{array}
\right.
\hbox{ and } 
v_0(x):=v(\lambda x) \;\;\; \hbox{ for every }x\in Q.
\end{equation}
By \eqref{eqjump1} when $a_0=0$ and $\lambda=1$, there exists $(u_n,v_n) \in {\cal A}(c,d,e_N)\times L^p(Q;\mathbb R^m)$ such that $u_n \to u_0$ in $L^1(Q;\mathbb R^d)$, $v_n \rightharpoonup v_0$ in $L^p(Q;\mathbb R^m)$ and
\begin{equation}
\label{flambdatoKplambda}
\displaystyle{\int_Q f_\lambda (u_n, v_n, \nabla u_n)dx \to \int_Q f_\lambda(u_0(x), v_0(x),0)dx + (K_p)_{\lambda}(0,c,d,e_N)},
\end{equation}
where $(K_p)_\lambda$ is the function defined in \eqref{Kp}, with $f$ replaced by $f_\lambda$ above. 
Consider any $a_0 \in \mathbb R^N$ and set
\begin{equation}
\label{unlambda}
\bar{u}_n(x):= u_n\left(\tfrac{x-a_0}{\lambda}\right), \;\;\; \bar{v}_n(x):= v_n\left(\tfrac{x-a_0}{\lambda}\right), \;\;\;\; x \in a_0+\lambda Q.
\end{equation}
Clearly $\{{\bar u}_n\}$ meets the boundary conditions, and is periodic in the $e_1, \dots, e_{N-1}$ directions with period $\lambda$. 
Moreover, $\|{\bar u}_n -u\|_{L^1(a_0+\lambda Q)} \to 0$ and ${\bar v}_n \rightharpoonup v$ in $L^p(a_0+\lambda Q;\mathbb{R}^m)$
and 
\begin{equation}
\label{fnto}
\begin{array}{ll}
\displaystyle{\int_{a_0 +\lambda Q}f({\bar u}_n, {\bar v}_n, \nabla {\bar u}_n)dx =\int_{a_0+\lambda Q} f(u_n\left(\tfrac{x-a_0}{\lambda}\right), v_n\left(\tfrac{x-a_0}{\lambda}\right),\tfrac{1}{\lambda}\nabla u_n\left(\tfrac{x-a_0}{\lambda}\right))dx}\\
\displaystyle{= \lambda^N\int_Q f_\lambda(u_n(y),  v_n(y),\nabla  u_n(y))dy \to \lambda^N \int_Q f_\lambda (u_0(y), v_0(y),0)dy + \lambda^N (K_p)_{\lambda}(0,c,d,e_N),}
\end{array}
\end{equation}
as $n\to +\infty$. Moreover,
\begin{equation}
\label{Kplambda}
{\displaystyle \lambda^N\int_Q f_\lambda(u_0(y), v_0(y), 0)dy = \int_{a_0+\lambda Q}f(u(x),v(x),0)dx,
\hbox{ and }
(K_p)_{\lambda}(0,c,d, e_N)=
\frac{1}{\lambda}K_p(0,c,d,e_N).}
\end{equation}
Hence we obtain \eqref{eqjump1}. 
\smallskip
\noindent{\bf Step 1 c).} We allow $f$ to have explicit $x$-dependence.
Let $A $ be an on open subset of $\Omega$ and $A^\ast:=\alpha +\lambda Q_\nu \subset \subset A$ for some $\alpha \in \mathbb R^N, \lambda >0$.   
Without loss of generality, we may assume that $a_0=0$ and $\nu=e_N$.
We denote $Q_\nu$ by $Q$ and we let $A':=\{x \in A^\ast: x_N=0\}$ and $Q':=\{x \in Q: x_N=0\}$. Since $A^\ast$ is compactly included in $A$, fixing $\varepsilon >0$ it is possible to find $\delta>0$ such that $(H_2)_p$ and Proposition \ref{Lemma2.15FM} $(e)$ hold uniformly in $A^\ast$, i.e.
\begin{equation}
\label{4.3BF}
\displaystyle{x,y \in A^\ast, |x-y|<\delta \Rightarrow |f(x,u,b, \xi) -f(y,u,b,\xi)|\leq \varepsilon C(1+|\xi|+|b|^p)}
\end{equation}
and
\begin{equation}
\label{4.4BF}
\displaystyle{x,y \in A^\ast, |x-y|<\delta \Rightarrow |K_p(x,b,c,d, \nu)- K_p(y,b,c,d,\nu)|\leq \varepsilon C(1+|d-c|+|b|^p).}
\end{equation}
Let $h\in\mathbb N$ be such that 
\begin{equation}
\label{4.5BF}
\eta:=\frac{\lambda}{h}<\delta
\end{equation}
and partition $A'$ into $h^{N-1}$ $(N-1)$-dimensional cubes, aligned according to the coordinate axes and with mutually disjoint interiors. Namely,
\begin{equation}
\label{4.6}
A'=\bigcup_{i=1}^{h^{N-1}}(a_i+\eta{\overline Q'}).
\end{equation}

\noindent Denoting $Q'_i:=a_i+\eta Q'$ and $Q_i:= a_i+\eta Q,$ we claim that there exist $\{u_k\} \subset W^{1,1}(A^\ast;\mathbb R^d)$ and $\{v_k\} \subset L^p(A^\ast;\mathbb R^m)$ such that $u_k \to u$ in $L^{1}(A;\mathbb R^d)$, $v_k \rightharpoonup v$ in $L^p(A;\mathbb R^m)$ and 
\begin{equation}
\label{4.7BF}
\displaystyle{\lim_{k \to +\infty}\int_{A^\ast}f(x, u_k, v_k,\nabla u_k)dx = \int_{J_u \cap A^\ast}K_p(x,0,c,d,e_N)d{\cal H}^{N-1}+ \int_{A^\ast}f(x,u(x), v(x),0)dx.}
\end{equation}
By Step 1 b), there exist sequences $\{u_k^{(1)}\}\subset {\cal A}(c,d,e_N)$, related to the cube $Q_1$ and $\{v_k^{(1)}\}\subset L^p(Q_1;\mathbb R^m)$, such that 
$$
\displaystyle{\lim_{k\to +\infty}\int_{Q_1}f(a_1, u_k^{(1)}, v_k^{(1)},\nabla u_k^{(1)})dx = \eta^{N-1}K_p(a_1,0,c,d,e_N)+\int_{Q_1}f(a_1, u(x), v(x),0)dx.}
$$
By Remark \ref{3.17FMrrem} $i)$, there exist subsequences, not relabeled, $\{\xi_k^{(1)}\}\subset W^{1,1}(Q_1;\mathbb R^d)$ and $\{\overline{v}_k^{(1)}\} \subset L^p(Q_1;\mathbb R^m)$ such that $\xi_k^{(1)}\to u$ in $L^1(Q_1;\mathbb R^d) $, with $\xi_k^{(1)}(x)= U_k^{(1)}((x-a_1)/\eta)$ on $\partial Q_1$, ($U_k^{(1)}$ is a mollification of $u$), $\overline{v}_k^{(1)}\rightharpoonup v$ in $L^p(Q_1;\mathbb R^m)$ and
\begin{align*}
\limsup_{k\to +\infty}\int_{Q_1}f(a_1, \xi_k^{(1)}, \overline{v}_k^{(1)},\nabla \xi_k^{(1)})dx &\leq \liminf_{k \to +\infty}\int_{Q_1}f(a_1, u_k^{(1)}, v_k^{(1)},\nabla u_k^{(1)})dx\\
&=\eta^{N-1}K_p(a_1,0,c,d,e_N)+\int_{Q_1}f(a_1, u(x), v(x),0)dx. \notag
\end{align*}
By the lower bound inequality proved in the previous section and the above estimate we have
$$
\displaystyle{\lim_{k\to +\infty}\int_{Q_1}f(a_1, \xi_k^{(1)}, \overline{v}_k^{(1)},\nabla \xi_k^{(1)})dx =\eta^{N-1}K_p(a_1,0,c,d,e_N)+\int_{Q_1}f(a_1, u(x), v(x),0)dx.}
$$
By induction we can repeat the same argument, obtaining $h^{N-1}$ further subsequences $\{k\}$ and corresponding sequences $\{\xi^{(j)}_k\}\subset {\cal A}(c,d, e_N)$ related to the cube $Q_j$, with $\xi^{(j)}_k \to u$ in $L^1(Q_j;\mathbb R^d)$, $\xi^{(j)}_k =U^{(j)}_k$ on $\partial Q_j$ and $\{\overline{v}^{(j)}_k\}\subset L^p(Q_j;\mathbb R^m)$, with $\overline{v}^{(j)}_k  \rightharpoonup v$ in $L^p(Q_j;\mathbb R^m)$ and for every $j=1,\dots, h^{N-1},$
$$
\begin{array}{ll}
\displaystyle{\lim_{k\to +\infty}\int_{Q_j}f(a_j, \xi_k^{(j)}, \overline{v}_k^{(j)},\nabla \xi_k^{(j)})dx=\eta^{N-1}K_p(a_j,0,c,d,e_N)+\int_{Q_j}f(a_j, u(x), v(x),0)dx.}
\end{array}
$$
Next we take the $h^{N-1}$ subsequence and for all $j=1,\dots, h^{N-1}$ we consider sequences $\{\zeta_k\}$ and $\{\tilde{v}_k\}$, defined in $\cup_{j=1}^{h^{N-1}}Q_j$ with $\zeta_k=\xi^{(j)}_k$, $\tilde{v}_k=\overline{v}^{(j)}_k$ on $Q_j$, such that for every $j=1,\dots, h^{N-1}$.
\begin{align}\label{4.9BF}
\displaystyle{\lim_{k\to +\infty}\int_{Q_j}f(a_j, \zeta_k, \tilde{v}_k,\nabla \zeta_k)dx=\eta^{N-1}K_p(a_j,0,c,d,e_N)+\int_{Q_j}f(a_j, u, v,0)dx.}
\end{align}
Define the sequences $\{u_{k,\varepsilon}\}$ and $\{v_{k,\varepsilon}\}$ almost everywhere on $A^\ast$, as follows
\begin{equation}
\label{ukepsilon}
u_{k,\varepsilon}(x):=\left\{
\begin{array}{ll}
\zeta_k(x )& \hbox{ if } x\in \cup_{j=1}^{h^{N-1}}Q_j,\\
d & \hbox{ if }x_N >\eta/2,\\
c & \hbox{ if }x_N <-\eta/2,
\end{array}
\right.
\;\;\;\;\;
v_{k,\varepsilon}(x):=\left\{
\begin{array}{ll}
\tilde{v}_k(x) &\hbox{ if }x \in \cup_{j=1}^{h^{N-1}}Q_j,\\
v(x) & \hbox{ if } |x_N|> \eta/2.
\end{array}
\right.
\end{equation}

\noindent Clearly, since $\zeta_k = U^{(j)}_k$ on $\partial Q_j$ and $U^{(j)}_k(x)=d$ (respectively $c$) for $x_N=\eta/2$ (respectively for $x_N=-\eta/2$), $u_{k,\varepsilon}\in W^{1,1}(A^\ast;\mathbb R^d)$. 
Also $\tilde{v}_k \in L^p(A^\ast;\mathbb R^m)$ and moreover it coincides with $v(x)$ if $|x_N|>\eta/2$. Furthermore, we have 
\begin{equation}
\label{ukeconv}
\displaystyle{\lim_{\varepsilon\to 0^+}\lim_{k\to +\infty} \|u_{k,\varepsilon}-u\|_{L^1(A^\ast;\mathbb R^d)} =0}
\end{equation}
and $v_{k,\varepsilon }\rightharpoonup v$ in $L^p(A^\ast;\mathbb R^m)$ as $k \to +\infty$ and as $\varepsilon \to 0$.
Also,
\begin{align}
&\displaystyle{\int_{A}f(x,u_{k,\varepsilon}, v_{k,\varepsilon}, \nabla u_{k,\varepsilon}) dx=\sum_{i=1}^{h^{N-1}}\int_{Q_i}f(a_i,\zeta_k, \tilde{v}_k, \nabla \zeta_k) dx+\sum_{i=1}^{h^{N-1}}\int_{Q_i}(f(x,\zeta_k, \tilde{v}_k, \nabla \zeta_k) -f(a_i,\zeta_k, \tilde{v}_k, \nabla \zeta_k))dx}
 \notag\\
&\displaystyle{+\int_{A^\ast\cap \{x_N>\eta/2\}} f(x,d, v(x),0)dx +\int_{A^\ast\cap \{x_N<-\eta/2\}} f(x,c, v(x),0)dx =:I_1+I_2+I_3+I_4.}\label{I1234}
 \end{align}
 Then it is easily seen that by \eqref{4.9BF}, we have
 $$
 \displaystyle{\lim_{k \to +\infty}I_1 = \sum_{i=1}^{h^{N-1}}\left( \eta^{N-1} K_p(a_i, 0,c,d, e_N)+\int_{Q_i}f(a_i,u(x),v(x),0)dx\right).}
 $$
 Moreover,
 
 \begin{equation}
 \label{I34}
 \displaystyle{\lim_{k \to +\infty}(I_3+I_4)= \int_{A^\ast\cap \{x_N>\eta/2\}} f(x,d, v(x),0)dx +\int_{A^\ast\cap \{x_N<-\eta/2\}} f(x,c, v(x),0)dx.}
 \end{equation}
 Regarding $I_2$, by \eqref{4.3BF}, $(H_1)_p$ and since, by construction, the sequences $\{\tilde{v}_k\}$ and $\{\zeta_k\}$ are bounded  in $L^p(\cup_{i=1}^{h^{N-1}}Q_i;\mathbb R^m)$ and in $W^{1,1}(\cup_{i=1}^{h^{N-1}}Q_i;\mathbb R^d)$ respectively, we have
 $$
 \begin{array}{ll}
 \displaystyle{\limsup_{k \to ì\infty}I_2 \leq \limsup_{k \to +\infty} \sum_{i=1}^{h^{N-1}}\int_{Q_i}\varepsilon C(1+ |\tilde{v}_k|^p + |\nabla \zeta_k|)dx = O(\varepsilon), }
 \end{array}
 $$
By \eqref{4.4BF} we have 
 \begin{align*}
 &\left| \int_{A^\ast \cap J_u} K_p(x,0, c,d, e_N)d {\cal H}^{N-1}-\eta^{N-1}\sum_{i=1}^{h^{N-1}} K_p(a_i,0, c,d, e_N)\right|\\
 &\leq \sum_{i=1}^{h^{N-1}} \int_{Q_i}\left|K_p(x,0, c,d, e_N)-K_p(a_i,0, c,d, e_N)\right| d {\cal H}^{N-1}=O(\varepsilon).
 \end{align*}
Finally, putting together, this estimate, the limits of $I_2, I_3, I_4$ and estimating $\sum_{i=1}^{h^{N-1}}\int_{Q_i}f(a_i,u,v,0)dx$ in $I_1$ via \eqref{4.3BF}, we obtain the desired approximating sequence, just letting $\varepsilon \to 0^+$ and using a diagonalization procedure.
 Thus we have proved \eqref{5.20FM2} when $u$ has a planar interface and $A^\ast$ is a cube.
 
\smallskip 
\noindent{\bf Step 1 d).} Now let $A^\ast$ be an open subset of $\Omega$ such that
 \begin{equation}
 \label{4.10BF}
\displaystyle{\lim_{j\to +\infty}\lim_{\varepsilon \to 0^+}\frac{1}{\varepsilon} {\cal L}^N\left(\left\{x \in A^\ast: {\rm dist}(x, \partial^\ast E\cap \overline{A^\ast})< \varepsilon, {\rm dist}(x, \partial A^\ast)<\tfrac{1}{j}\right\}\right)=0.}
 \end{equation}
 
 With $u$ defined as in Step 1 c), we claim that, given any sequence $\varepsilon_n \to 0^+$, there exists a subsequence $\{\varepsilon_{n_k}\}$, a sequence $\{u_k\}\subset W^{1,1}(A^\ast;\mathbb R^d)$ and $\{v_k\}\subset L^p(A^\ast;\mathbb R^m)$ such that $u_k\to u$ in $L^1(A^\ast;\mathbb R^m)$ and $v_k \rightharpoonup v$ in $L^p(A^\ast;\mathbb R^m)$ and
 $$
 \displaystyle{\lim_{k\to +\infty} \int_{A^\ast}f(x, u_k,v_k, \nabla u_k)dx=\int_{A^\ast}f(x,u(x), v(x),0)dx + \int_{\partial^\ast E \cap A^\ast}K_p(x,0,c,d,e_N)d {\cal H}^{N-1}.}
 $$   
 
 By Whitney covering theorem we may write
 $$
 \displaystyle{A^\ast=\bigcup_{i=1}^\infty (a_i+ \delta_i \overline{Q}) =: \bigcup_{i=1}^\infty\overline{Q_i},}
 $$
 where $\delta_i, {\rm diam}Q_i, {\rm dist}(Q_i,\partial A^\ast)$ , ${\rm dist}(a_i,\partial A^\ast)$ satisfy  \cite[(4.11) and (4.12)]{BF}.
\noindent We choose $L>0$ as in \cite[page 552]{BF}
and introduce 
\begin{equation*}
\Omega_j:=\left\{x\in A^\ast: {\rm dist}(x,\partial A^\ast)\geq \tfrac{1}{j}\right\},\,
{\cal G}_j:=\left\{Q_i:{\rm dist}(a_i,\partial A^\ast )<\tfrac{1}{Lj}\right\},\,\,
{\cal F}_j:=\left\{Q_i:{\rm dist}(a_i,\partial A^\ast) \geq\tfrac{1}{Lj}\right\}.
\end{equation*}

For every $j$, ${\cal F}_j$ is a finite family of cubes, the choice of $L$, provides that if $Q_i\in {\cal G}_j$ then $Q_i\cap \Omega_j=\emptyset,$ so that $\Omega_j \subset \cup{\cal F}_j$ and thus $\Omega_j$ is covered by a finite number of cubes $Q_i$ (see \cite{BF} for details).  

By Step 1 c)  given the sequence $\varepsilon_n \to 0^+$, there exists a subsequence $\{\varepsilon_k^{(1)}\}$ and sequences $\{u_k^{(1)}\} $ in $W^{1,1}(Q_1;\mathbb R^d)$ and $\{v_k^{(1)}\}$ in $L^p(Q_1;\mathbb R^m)$ such that $u_k^{(1)}\to u$ in $L^1(Q_1:\mathbb R^d),$ $v_k^{(1)}\rightharpoonup v $ in $L^p(Q_1;\mathbb R^m)$  and
$$
\displaystyle{\lim_{k\to +\infty} \int_{Q_1}f(x, u^{(1)}_k,v^{(1)}_k,\nabla u^{(1)}_k)dx =\int_{\partial^\ast E\cap Q_1}K_p(x,0,c,d,e_N)d {\cal H}^{N-1}+\int_{Q_1}f(x,u(x),v(x),0)dx}.
$$
By i) in Remark \ref{3.17FMrrem} there exists a subsequence, still denoted by $k$, $\{w_k^{(1)}\} \subset W^{1,1}(Q_1,\mathbb R^d)$,  and $\{\tilde{v}^{(1)}_k\}\subset L^p(Q_1;\mathbb R^m)$, such that $w_k^{(1)}\to u$ in $L^1(Q_1;\mathbb R^d)$ and $w_k^{(1)}(x)=U_k^{(1)}((x-a_1)/\delta_1)$ for every $x \in \partial Q_1$, where the latter functions are mollifications of $u$, with $\tilde{v}^{(1)}_k\rightharpoonup u$ in $L^p(Q_1;\mathbb R^m)$ and
$$
\begin{array}{ll}
\displaystyle{\limsup_{k\to +\infty} \int_{Q_1}f(x, w^{(1)}_k,\tilde{v}^{(1)}_k,\nabla w^{(1)}_k)dx\leq}
\displaystyle{\liminf_{k\to +\infty} \int_{Q_1}f(x, u^{(1)}_k,v^{(1)}_k,\nabla u^{(1)}_k)dx }
\\
\displaystyle{=\int_{\partial^\ast E\cap Q_1}K_p(x,0,c,d,e_N)d {\cal H}^{N-1}+\int_{Q_1}f(x,u(x),v(x),0)dx.}
\end{array}
$$
This together with the lower bound inequality proved in the previous subsection gives
$$
\displaystyle{\limsup_{k\to +\infty} \int_{Q_1}f(x, w^{(1)}_k,\tilde{v}^{(1)}_k,\nabla w^{(1)}_k)dx=\int_{\partial^\ast E\cap Q_1}K_p(x,0,c,d,e_N)d{\cal H}^{N-1}+\int_{Q_1}f(x,u(x),v(x),0)dx.}
$$
By repeatedly taking subsequences and applying Step 1 c) and $(i)$ in Remark \ref{3.17FMrrem}, following the same arguments as above, since in ${\cal F}_j$ there are only finitely many cubes, it is possible to obtain a sequence $\{\varepsilon_k\}$ of $\{\varepsilon_n\}$ and sequences $\{\xi^{(i)}_{k}\}\subset W^{1,1}(Q_i;\mathbb R^d)$ and in ${\cal A}(c,d, e_N)$ relative to the cube $Q_i$, such that $\xi^{(i)}_k\to u$ in $L^1(Q_i;\mathbb R^d)$ and  $\xi^{(i)}_k(x)=U_{\varepsilon_{k}}((x-a_i)/\delta_i)$ for $x \in \partial Q_i$, and $\{\overline{v}^{(i)}_{k}\}\subset L^p(Q_i;\mathbb R^m)$, such that $\overline{v}^{(i)}_{k} \rightharpoonup v$ in $L^p(Q_i;\mathbb R^m)$ such that
\begin{equation}
\label{4.13BF}
\displaystyle{\lim_{k\to +\infty} \int_{Q_i}f(x, \xi^{(i)}_k,\overline{v}^{(i)}_k,\nabla \xi^{(i)}_k)dx=\int_{\partial^\ast E\cap Q_i}K_p(x,0,c,d,e_N)d{\cal H}^{N-1}+\int_{Q_i}f(x,u(x),v(x),0)dx,}
\end{equation}
for every $Q_i$ in ${\cal F}_j$. Denote by $\{\zeta_k\}$ and $\{\tilde{v}_k\}$ the sequences defined in $\cup_{\{i:Q_i\in {\cal F}_j\}} Q_i$, such that $\zeta_k:=\xi_k^{(i)}$  and $\tilde{v}_k:=\overline{v}_k^{(i)}$ in $Q_i.$ Next we define the sequences
\begin{equation*}
u_k(x):=\left\{
\begin{array}{ll}
\zeta_k(x) &\hbox{ if }x \in \cup_{\{i:Q_i\in {\cal F}_j\}} Q_i,\\
U_{\varepsilon_k}(x) &\hbox{ otherwise,}
\end{array}
\right.
\;\;\;\;
\hbox{ and }
v_k(x):=\left\{
\begin{array}{ll}
\tilde{v}_k(x) &\hbox{ if }x \in \cup_{\{i:Q_i\in {\cal F}_j\}} Q_i,\\
v(x) &\hbox{ otherwise.}
\end{array}
\right.
\end{equation*}

\noindent Clearly, $u_k \to u$ in $L^1(A^\ast;\mathbb R^d)$ and $v_k \rightharpoonup v$ in $L^p(A^\ast;\mathbb R^m)$.
Moreover, recalling that $v \in L^\infty(\Omega;\mathbb R^m)$, $\|U_{\varepsilon_k}\|_{L^\infty}\leq C$, $\|\nabla U_{\varepsilon_k}\|_{L^\infty} = O(1/\varepsilon_{n_k})$ and since $x \in A^\ast \setminus \{Q_i: Q_i \in {\cal F}_j\}$ implies that ${\rm dist}(x, \partial A^\ast)< \frac{1}{j}$, we have, by \eqref{4.13BF} and $(H_1)_p$,

\begin{align*}
&\displaystyle{\limsup_{k \to +\infty}\int_{A^\ast}f(x, u_k, v_k, \nabla u_k)dx \leq C\limsup_{k\to +\infty} \left(\tfrac{1}{\varepsilon_k}+1\right){\cal L}^N(\{x\in A^\ast: |x_N|<\varepsilon_k, {\rm dist}(x,\partial A^\ast)<\tfrac{1}{j}\}})\\
&\displaystyle{+\sum_{Q_i\in {\cal F}_j}\left(\int_{\partial E^\ast \cap Q_i}K_p(x, 0,c,d,e_N)d{\cal H}^{N-1}+ \int_{Q_i}f(x, u(x),v(x),0)dx\right)}\\
&\displaystyle{\leq C\limsup_{k\to +\infty} \left(\tfrac{1}{\varepsilon_k}+1\right){\cal L}^N\left(\left\{x\in A^\ast: |x_N|<\varepsilon_k, {\rm dist}(x,\partial A^\ast)<\tfrac{1}{j}\right\}\right)}\\
&\displaystyle{+\int_{\partial E^{\ast}\cap A^\ast}K_p(x, 0,c,d,e_N)d{\cal H}^{N-1}+\int_{A^\ast}f(x,u(x),v(x),0)dx.}
\end{align*}
Thus taking the limit as $j \to +\infty$ by \eqref{4.10BF} and the lower bound inequality we conclude that
$$
\displaystyle{\limsup_{k\to +\infty}\int_{A^\ast}f(x,u_k,v_k,\nabla u_k)dx\leq \int_{\partial E^\ast \cap A^\ast}K_p(x,0,c,d,e_N)dx +\int_{A^\ast}f(x,u,v,0)dx. }
$$
\noindent{\bf Step 2.} Now we assume that $u$ has a polygonal interface, i.e. $u:=\chi_E c+ (1-\chi_E)d$, where $E \subset \Omega$ is of the form $E=E' \cap \Omega$, $\partial^\ast E\cap \Omega= \partial^\ast E'\cap \Omega$, with $E'$ a polyhedral set. 
 
As in Step 1 d), let $A^\ast$ be an open subset of $\Omega$ such that \eqref{4.10BF} holds. 
We claim that there exists a sequence $\{u_k\}$ in $W^{1,1}(A^\ast;\mathbb R^d)$ and a sequence $\{v_k\}$ in $L^p(A^\ast;\mathbb R^m)$ such that $u_k\to u$ in $L^1(A^\ast;\mathbb R^d)$, $v_k \rightharpoonup v$ in $L^p(A^\ast;\mathbb R^m)$ and 
$$
\displaystyle{\lim_{k\to +\infty}\int_{A^\ast} f(x,u_k,v_k,\nabla u_k)dx =\int_{A^\ast}f(x,u,v,0)dx + \int_{\partial E^\ast \cap A^\ast}K_p(x,0,c,d,\nu)d{\cal H}^{N-1}.}
$$
The claim is achieved following a proof entirely similar to \cite[Section 5, Step 3]{FM2}. It relies on an induction argument and on the application of Step 1 c) and on a slicing procedure similar to Lemma \ref{3.1FMr} in order to connect recovery sequences between two domains for the $u$ and the $v$.

\smallskip
\noindent{\bf Step 3.} Finally, let $A$ be a Lipschitz subdomain of $\Omega$ and consider an arbitrary $u:=\chi_E c+ (1-\chi_E)d$ with ${\rm Per}(E;A)<+\infty$. Since $\partial E$ is Lipschitz, by Theorem \ref{Baldo} there exist polyhedral sets $E_k$ such that $\chi_{E_k}\to \chi_E$ in $L^1(A)$, ${\rm Per}_(E_k;A)\to {\rm Per}_(E;A)$, ${\cal L}^N(E_k)={\cal L}^N(E)$ and ${\cal H}^{N-1}(\partial^\ast E_k\cap \partial \Omega)=0$ so that
$$
\displaystyle{\lim_{j\to +\infty}\lim_{\varepsilon \to 0^+}\frac{1}{\varepsilon}{\cal L}^N(\{x \in A: {\rm dist}(x, \partial^\ast E_k\cap \overline{A})< \varepsilon, {\rm dist}(x, \partial A)<1/j\})=0.}
$$
By Step 2, for every $k$ there exist $\{u_n^{(k)}\}\subset L^1(A;\mathbb R^d)$, such that
$u_n^{(k)}\to \chi_{E_k}c+ (1-\chi_{E_k})d$ in $L^1(A;\mathbb R^d)$ as $n\to +\infty$, and $\{v_n^{(k)}\}\subset L^p(A;\mathbb R^m)$ such that $v_n^{(k)}\rightharpoonup v$ in $L^p(A;\mathbb R^m)$ and
$$
\displaystyle{\lim_{n\to +\infty}\int_{A}f(x, u_n^{(k)},v_n^{(k)},\nabla u_n^{(k)})dx\leq \int_{A}f(x,u,v,0)dx +\int_{\partial^\ast E_k\cap A}K_p(x,0,c,d,\nu_k)d{\cal H}^{N-1},}
$$ 
where $\nu_k$ is the measure theoretic unit normal to $\partial^\ast E_k$ at $x$.

Regarding the weak convergence of $v_n^{(k)}$ to $v$, observe that $L^{p'}(A;\mathbb R^m)$ is separable, hence we can take a dense sequence of functions $\{\psi_l\} \subset L^{p'}(\Omega;\mathbb R^m)$ such that
$\lim_{n\to +\infty}\int_{A}(v_n^{k}-v)\psi_l dx =0 \hbox{ for every }l \in \mathbb N$
Consider an increasing sequence $\{k\}$ such that
$$
\displaystyle{\|u_{n(k)}^{(k)}-\chi_{E_k}c+ (1-\chi_{E_k})d\|_{L^1(A;\mathbb R^d)}< \frac{1}{k}, \hbox{ and } \left|\int_{A}(v_{n(k)}^{(k)}-v)\psi_l dx\right|<\frac{1}{k} \hbox{ for every }l=1,\dots, k,}
$$
and
$
\displaystyle{\left|\int_{\partial^\ast E_k \cap A}K_p(x,0,c,d,\nu_k)d{\cal H}^{N-1}-\left(\int_{A}f(x,u,v,0)dx +\int_{A}f(x, u_{n(k)}^{(k)}, v_{n(k)}^{(k)},\nabla u_{n(k)}^{(k)})dx\right)\right|<\frac{1}{k}.}
$

\noindent Set $\overline{u}_k=u_{n(k)}^{(k)}$ and $\overline{v}_k=\overline{v}_{n(k)}^{(k)}$, then $
\overline{u}_k\to \chi_E c + (1-\chi_E)d \hbox{ in }L^1(A;\mathbb R^d).$ 
 Moreover, by the growth condition on $f$ and the bounds on $K_p$ it results that $\|v_k\|_{L^p(A;\mathbb R^m)}\leq \frac{1}{k}+ C$ and the density of $\{\psi_l\} $ in $L^{p'}(A;\mathbb R^m)$ ensures that $v_k \rightharpoonup v$ in $L^p(A;\mathbb R^m)$.
Furthermore, recall that for every continuous function $g:A\times \mathbb R^d\to [0,+\infty)$ we have (see \cite{AFP})
$$
\displaystyle{\int_{\partial^\ast E_k\cap A}g(x, \nu_k(x))d{\cal H}^{N-1}\to\int_{\partial^\ast E\cap A}g(x, \nu(x))d{\cal H}^{N-1} .} 
$$

\noindent Since by $(b)$ in Proposition \ref{Lemma2.15FM} $K_p(\cdot,0,c,d, \cdot)$ is upper semicontinuous, there exist continuous functions $g_m:A \times \mathbb R^N\to [0,+\infty)$ such that
$$
K_p(x,0,c,d,\xi)\leq g_m(x,\xi)\leq C|\xi|  
\hbox{ and }
K_p(x,0,c,d,\xi)=\inf_{m}g_m(x,\xi)
$$
for every $(x,\xi)\in A\times \mathbb R^N$, where $K_p(x,0,c,d,\cdot)$ has been extended as a positively one homogeneous function to $\mathbb R^N$.
Thus for all $m \in \mathbb N$, it results
$$
\displaystyle{\limsup_{k\to +\infty} \int_{A}f(x, \overline{u}_k, \overline{v}_k,\nabla \overline{u}_k)dx}
\displaystyle{\leq\int_{A}f(x,u,v,0)dx + \int_{\partial^\ast E\cap \Omega}g_m(x,\nu(x))d{\cal H}^{N-1}.}
$$
Taking the limit when $m\to +\infty$, using Lebesgue's monotone convergence theorem and the lower bound inequality we obtain 
$$
\begin{array}{ll}
\displaystyle{\lim_{k \to +\infty}\int_{\Omega}f(x,\overline{u}_k,\overline{v}_k,\nabla \overline{u}_k)dx =\int_{A}f(x,u,v,0)dx + \int_{\partial^\ast E\cap A}K_p(x,0,c,d,\nu)d{\cal H}^{N-1}.}
\end{array}
$$

\noindent{\bf Step 4.} Let $A$ be any domain in $\Omega$. For any $K$ compact subset of $A$ we can find a Lipschitz domain $A'$ such that
$K \subset A' \subset A$,  (see \cite[Remark 5.5]{BZ} and \cite[Chapter 6]{S}) and
\begin{align*}
{\cal J}_p(u,v;K)&\leq {\cal J}_p(u,v,A')\leq \int_{A'}f(x,u,v,0)dx + \int_{A'\cap \partial^\ast E}K_p(x,0,c,d,\nu)d{\cal H}^{N-1}\\
&\leq\int_{A}f(x,u,v,0)dx + \int_{A\cap \partial^\ast E}K_p(x,0,c,d,\nu)d{\cal H}^{N-1}.
\end{align*}

\noindent By the inner regularity of ${\cal J}_p(u,v;\cdot)$, it results
$$
{\cal J}_p(u,v;A)=\sup{\{{\cal J}_p(u,v, K):  K\subset A, K  \hbox{ compact}\}}\leq \int_{A}f(x,u,v,0)dx + \int_{A\cap \partial^\ast E}K_p(x,0,c,d,\nu)d{\cal H}^{N-1}.
$$
The additivity of ${\cal J}_p(u,v;\cdot)$ allows us to consider any open subset of $\Omega$, not necessarily connected.
The lower bound inequality provides 
$$
{\cal J}_p(u,v;A)=\int_A f(x,u,v,0)dx + \int_{\partial^\ast E\cap A}K_p(x,0,c,d,\nu)d{\cal H}^{N-1}.
$$
Moreover,
\begin{align*}
{\mathcal J}_p(u,v; E) &\leq \inf\{{\mathcal J}_p(u,v;A): A \subset \Omega \text { is open}, E\subset A \}\\
&\leq\inf\left\{\int_A f(x,u,v,0)dx + \int_{\partial^\ast E\cap A}K_p(x,0,c,d,\nu)d{\cal H}^{N-1}: A \subset \Omega \text { is open}, E \subset A\right\}\\
&=\int_E K_p(x,0,c,d,\nu)d{\cal H}^{N-1},
\end{align*}
where $\nu$ is the normal to $E$.

\noindent{\bf Case 2.} Consider $u: =\sum c_i \chi_{E_i}$, where $\{E_i\}_{i=1}^\infty$ forms a partitions of $\Omega$ into sets of finite perimeter. The proof follows along the lines of \cite[Proposition 4.8, Step 1]{AMT}, since the representation for the surface term is independent on the the target function $v$.  
Thus \eqref{upperboundLp} follows for every $u \in BV (\Omega;T)$ with $T$ a finite subset of $\mathbb R^d$ and $v \in L^\infty(\Omega;\mathbb R^m)$.

\noindent{\bf Case 3.} Let $u\in BV(\Omega;\mathbb R^d)\cap L^\infty(\Omega;\mathbb R^d)$. As in Case 2 the fact that the integral representation is in terms of $K_p$ evaluated at $v=0$, allows us to follow the same arguments in  \cite[Proposition 4.8, Step 2]{AMT}, exploiting property $(c)$ in Proposition \ref{Lemma2.15FM}. Hence, by \eqref{absolutely-continuousUB} and \eqref{Cantor-densityUB}, \eqref{representationLp} holds. 
 
\medskip
\noindent{\bf Part 2.} Let $u \in BV(\Omega;\mathbb R^d)$ and $v \in L^\infty(\Omega;\mathbb R^m)$. As in \cite[Section 5, Step 4]{FM2}, the lower semicontinuity of $\mathcal{J}_p$ and the result achieved in Part 1 provides
\begin{equation*}
\begin{array}{ll}
\displaystyle{{\cal J}_p(u,v;\Omega)\leq \liminf_{i\to +\infty}{\cal J}_p(\phi_i(u),v;\Omega)=\liminf_{i\to +\infty}\left\{\int_{\Omega}f(x, (\phi_i(u)),v,(\nabla \phi_i(u)))dx\right.}
\\
\displaystyle{\left. +\int_{J_{\phi_i(u)}}K_p(x,0, (\phi_i(u))^+,(\phi_i(u))^-, \nu_{\phi_i} )d {\cal H}^{N-1}+\int_\Omega f_p^\infty\left(x, \phi_i(u), 0,\tfrac{d D^c \phi_i(u)}{d |D^c \phi_i(u)|}\right)dx\right\}.}
\end{array}
\end{equation*}
where $\phi_i\in W^{1,\infty}_0(\mathbb R^d;\mathbb R^d)$ such that
$\phi_i(\xi)=\left\{
\begin{array}{ll} x, &|\xi|<a_i,\\
0, & |\xi|\geq a_i,
\end{array}
\right.
$
with the sequence $\{a_i\}\subset \mathbb R^+$ such that $a_i\to +\infty $ as $i \to +\infty$, and $\|\nabla \phi_i\|_{\infty}\leq 1$.
Then \eqref{representationLp} holds for every $u \in BV(\Omega;\mathbb R^d)$,
passing to the limit as $i\to +\infty$, exploiting $(c)$ in Proposition \ref{Lemma2.15FM}, $(H_1)_p$, \eqref{finftypgrowth}, and the lower bound inequality.
 
\smallskip
 
\noindent{\bf Part 3.} Concerning the case $BV(\Omega;\mathbb R^d) \times L^p(\Omega;\mathbb R^m)$ we follow a standard truncation argument, defining for every positive real number $L$, $\tau_L: [0,+\infty) \to [0, +\infty)$, as
$
\tau_L(t):=\left\{
\begin{array}{ll}
t &\hbox{ if } 0 \leq t < L,\\
0 &\hbox{ if } t \geq L.
\end{array}
\right.
$
For every $v \in L^p(\Omega;\mathbb R^m)$, define $v_L:= \tau_L(|v|)v$, thus $v_L \in L^\infty(\Omega;\mathbb R^m)$,  $\int_\Omega |v_L|^p dx \leq \int_\Omega |v|^p dx$ and $v_L \to v$ in $L^p(\Omega;\mathbb R^m)$, as $L \to +\infty$. By the lower semicontinuity of $\overline{J}_p$ and Part 2,   we have that
$$
\displaystyle{\overline{J}_p(u,v)\leq \liminf_{L \to +\infty} \left(\int_\Omega f(x,u, v_L, \nabla u)dx+ \int_{J_u}K_p(x,0,u^+,u^-, \nu_u) d{\cal H}^{N-1}+\int_{\Omega}f^\infty\left(x, u,0,\tfrac{d D^c u}{d |D^c u|}\right)d |D^c u|\right)}.
$$
Lebesgue's dominated convergence Theorem provides \eqref{upperboundLp} 
for every $(u,v) \in BV(\Omega;\mathbb R^d)\times L^p(\Omega;\mathbb R^m).$
\end{proof}

\begin{proof}[Proof of Theorem \ref{MainResultp}] 
The proof follows by Theorems \ref{lowerboundL_pthm} and \ref{upperboundL_pthm}.
\end{proof}

\begin{remark}
\label{other proof ub}
It is worth to observe that in the upper bound for the jump term, the proof of Case 1, Steps 1 d), 2, 3, and 4, could be replaced by arguments more similar to the ones in \cite[Proposition 4.1]{FKP1}, i.e adopt a covering $\{Q_i\}$ of the type \cite[(4.5)]{FKP1}, placing together with the sequences $\{\xi_n^{(i)}\}$ in \cite[(4.6)]{FKP1}, sequences $\{v_n^{(i)}\}$, coinciding with $v(x)$ in a layer (depending on $n$ and $(i)$) of the sets $\{x\in Q_i:(x-a_i)\cdot e_N =-\tfrac{\eta}{2}\}$ and $\{x\in Q_i:(x-a_i)\cdot e_N =\tfrac{\eta}{2}\}$, and then exploiting diagonal arguments and the reasonings in \cite{FM2}. 
\end{remark}
\section{Main Results: $BV\times L^\infty$}\label{MainLinfty}
This section is devoted to the proof of Theorem \ref{MainResultinfty} and it is divided in two subsections. The first for the lower bound and the second for the upper bound. 

\subsection{Lower semicontinuity in $BV \times L^\infty$}\label{lbbvlinfty}

\begin{theorem}\label{lowerboundL_inftythm}
Let $f:\Omega
\times\mathbb{R}^d\times\mathbb{R}^{m}\times\mathbb R ^{d\times
N}\rightarrow[ 0,+\infty) $ be a continuous function satisfying $(H_0)$, $%
( H_{1}) _{\infty}-( H_{3}) _{\infty}$. Then 
\begin{align*}
\underset{n\rightarrow+\infty}{\lim\inf}\int_{\Omega}f( x,u_{n},v_{n},\nabla u_{n}) dx
&\geq\int_{\Omega}f( x,u,v,\nabla u) dx  
+\int_{J_{u}\cap \Omega}K_\infty( x,0,u^{+},u^{-}
,\nu_{u}) d\mathcal{H}^{N-1}   \\
&+\int_{\Omega}f^{\infty}(x,u,0,\tfrac{dD^{c}u}{
d\vert D^{c}u\vert }) d\vert 
D^{c}u\vert. 
\end{align*}
in $BV(\Omega;\mathbb{R}^{d}) \times L^{\infty}(\Omega ;\mathbb{
R}^{m}) $ with respect to the ($L^1-$strong$~\times~L^{p}-$%
weak)-convergence,
where $K_\infty$ is defined in \eqref{Kinfty} and $f^\infty$ is the $(\infty,1)-$ recession function defined in \eqref{recession}.
\end{theorem}

\begin{proof} Up to a subsequence, denote by $\mu$ the weak $\ast$ limit of the measures $\mu_{n}:=f( x,u_{n},v_{n},\nabla u_{n}) \mathcal{L}^{N}$, where $\{u_n\}$ and $\{v_n\}$ can be assumed in $C^\infty_0$. Via Besicovitch derivation theorem it is enough to prove the equivalent of \eqref{lowerboundbulk}$-$\eqref{lowerboundcantor} with $f^\infty_p$ and $K_p$ replaced by $f^\infty$ and $K_\infty$, respectively. 
To achieve this, we apply the blow up method.

\smallskip
\noindent \textbf{Bulk part.} 
The proof can be found in \cite{RZ}, with the obvious adaptations to the $BV$ case.
\smallskip

\noindent \textbf{Jump part.}
We just emphasize the main differences from the $BV\times L^p$ case.
Let $x_{0}\in J_{u}$, be as in the jump part of Theorem \ref{lowerboundL_pthm}. Then the equivalent of \eqref{mujestimate1} becomes 
\begin{equation}\label{mujestimatelinfty}
\begin{array}{ll}
\mu_{j}(x_{0}) \geq  \displaystyle{\frac{1}{\vert u^{+}(x_{0}) -u^{-}( x_{0})
\vert }\lim_{k\rightarrow +\infty }\lim_{n\rightarrow +\infty
}\int_{Q}\varepsilon _{k}f(x_{0}+\varepsilon_{k}y,u_{n,k}(y),v_{n,k}(y)
 ,\frac{1}{%
\varepsilon _{k}}\nabla u_{n,k}(y)) dy,}
\end{array}
\end{equation}
where $u_{n,k}(y) :=u_{n,k}(x_{0}+\varepsilon _{k}y)$ and $v_{n,k}(y) :=v_{n}(x_{0}+\varepsilon _{k}y).$ Moreover, \eqref{ulimit} holds, with $u_0$ defined as in \eqref{u0}.
Using the separability of $L^1(Q;\mathbb R^m)$, together with a diagonalization argument, from \eqref{mujestimatelinfty}, \eqref{ulimit}  and \eqref{doubleweaklimit}, we obtain the existence of sequences $\bar{u}_k:=u_{n(k),k}$ and $\bar{v}_k:=v_{n(k),k}$ such that $\bar{u}_k \to u_0$ in $L^1(Q;\mathbb R^d)$, $\bar{v}_k \overset{\ast}{\rightharpoonup} \alpha$ in $L^\infty(Q;\mathbb R^m)$, where $\alpha$ is  a function whose average in $Q$ is $y_0$, which in turn is the limit, up to a subsequence, of $\frac{1}{\varepsilon^N}\int_{Q(x_0,\varepsilon)}v(y)dy.$

\noindent We obtain an estimation for $\mu _{j}$ as in \eqref{mujestimate} replacing $f^\infty_p$ therein by $f^\infty$.

Using Proposition \ref{proprecessioninfty} $iii)$ we get that for any $\varepsilon
>0,$ if $k$ is sufficiently large%
\begin{align*}
\int_{Q}f^{\infty }( x_{0}+\varepsilon_{k}y,\bar{u}_{k},\bar{v}_{k},\nabla \bar{u}_{k})
-f^{\infty }(x_{0},\bar{u}_{k},\bar{v}_{k},\nabla \bar{u}_{k})dy \geq -\varepsilon C\int_{Q}\vert \nabla \bar{u}_{k}\vert dy\geq O(\varepsilon) .
\end{align*}
On the other hand, by $(H_{3})_\infty$ and H\"{o}lder inequality 
we conclude
\begin{align*}
& \int_{Q}\varepsilon _{k}(f( x_{0}+\varepsilon
_{k}y,\bar{u}_{k} ,\bar{v}_{k},\frac{1}{\varepsilon _{k}}\nabla \bar{u}_{k})
-f^{\infty }(x_{0}+\varepsilon _{k}y,\bar{u}_{k}
,\bar{v}_{k},\nabla \bar{u}_{k})) dy \\
 &\leq C_M\int_{\{ y\in Q: \frac{\vert\nabla
\bar{u}_{k}\vert}{\varepsilon _{k}} \geq L\} }(\vert
\nabla \bar{u}_{k}\vert ^{1-\tau }\varepsilon _{k}^{\tau
})dy + C_M\int_{\{ y\in Q:\frac{|\nabla\bar{u}_{k}|}{\e_k} < L \} }\vert \nabla \bar{u}_k\vert dy \\
 &\leq O(\varepsilon) +C_M\varepsilon _{k}^{\tau
}\left(\int_{Q}\vert \nabla \bar{u}_{k}\vert
dy\right) ^{1-\tau }= O(\varepsilon),
\end{align*}%
where the constants $C_M$ vary from line to line but are all related to the $L^\infty$ uniform bound on $\{\overline{v}_k\}$, $M$.
Arguing as in \cite{FM2} we are led to the existence of two, not relabeled, subsequences $\{\overline{u}_k\}$ and $\{\overline{v}_k\}$ converging strongly to $u_0$ in $L^1$ and weakly $\ast$ in $L^\infty$ to $\alpha$, respectively and such that \eqref{mujestimate2} holds with $f^\infty_p$ replaced by $f^\infty$.

Next we apply Remark \ref{3.17FMrrem} $ii)$  to $f^\infty(x_0, \cdot, \cdot, \cdot)$, obtaining \eqref{mujestimate3} with the obvious adaptations, where $\xi_k \to u_0$ in $L^1(Q;\mathbb R^d)$ and $\xi_k \in {\cal A}(u^+(x_0), u^-(x_0),\nu_u(x_0))$, $\zeta_k \in L^\infty(Q;\mathbb R^m)$, converging weakly $\ast$ to $ \alpha$ in $L^\infty(Q;\mathbb R^m)$ with $\int_Q \zeta_k dy =y_0$. In particular, by \eqref{Kinfty} we have 
\begin{equation*}
\mu _{j}(x_{0})\geq K_\infty(x_0,y_0,u^{+}(x_{0}),u^{-}(x_{0}),\nu _{u}(x_{0}))\hbox{
for }\mathcal{H}^{N-1}-\hbox{a.e. }x_{0}\in J_{u}\cap \Omega,
\end{equation*} and this, together with Proposition \ref{Kindependentofv} concludes the proof of the lower bound inequality for the jump part.

\medskip
\noindent\textbf{Cantor part.} We divide the proof into several steps, just emphasizing the main differences with the $BV\times L^p$ case. Recall that $\eqref{eq1Cantorlower} - \eqref{eq5Cantorlower}$ hold.

\noindent {\bf Step 1. } It suffices to observe that \eqref{limitv=0} holds for every $1< p<\infty$, in particular there exist two sequences $\{\bar{u}_k\}$ and $\{{\bar v}_k\}$ such that  Step 1 a) in the $BV \times L^p$ case holds. 

\noindent{\bf Step 2.} With easier estimates than those of the $BV \times L^p$ case, we obtain \eqref{eqwithtildeu}, where the sequences $\{\tilde{v}_k\} $ and $\{\tilde{u}_k\}$ are obtained through a diagonalization argument from $\{\bar{v}_k\}$ and $\{\tilde{u}_k^{r,s}\}$, where this latter sequence is defined by
$$
\tilde{u}_k^{r,s}:=a_k+\varphi_{r,s}\left(|\bar{u}_k-a_k)| \right)(\bar{u}_k-a_k),
$$
with $a_k:= \frac{1}{|Q_k|}\int_{x_0+\varepsilon Q}\bar{u}_k\,dx$, and
$\varphi_{r,s}$ is as in \eqref{varphirs}.

\noindent {\bf Step 3.}  Notice that, if we consider $\{\bar{w}_k\}$ and $\{\theta_k\}$ as in \eqref{wkhat} and \eqref{thetak}, respectively, then condition (\ref{eqwithtildeu}) can be written as
\begin{equation}\label{eqwithmukinfty}
\left(1+\omega_{M,K}\left(\tfrac{1}{n}\right)\right)\,\mu^c(x_0)\ge\limsup_{k\rightarrow +\infty}\frac{1}{\theta_k}\int_{\gamma Q}f(x_0,u(x_0),\bar{v}_k(x_0+\varepsilon_k z),\theta_k\nabla\bar{w}_k(z))\,dz.\end{equation}
Then, modifying $\{\bar{v}_k(x_0+\varepsilon_k\cdot)\}$ and $\{\bar{w}_k\}$ near the boundary of $\gamma Q$ new sequences $\{\tilde{v}_k\}$ and $\{\tilde{w}_k\}$ are obtained in order to apply the convexity-quasiconvexity of $f$. We consider an inner cube $\tau Q$, $\tau\in (t,\gamma)$, and we modify the sequences in a layer $\tau Q\setminus\tau(1-\delta)Q$. Indeed we construct  $\tilde{v}_k\overset{\ast}{\rightharpoonup} \alpha$ in $L^\infty(\tau Q;\mathbb R^m)$, $\int_{\tau Q}\tilde{v}_kdx=\int_{\tau Q}\alpha dx=:y_0(\tau)$, and $\tilde{w}_k(x)=\xi_0(x)+\varphi(x)$ for some $\varphi\in W^{1,\infty}_{\rm {per}}(\tau Q;\mathbb{R}^d)$ and such that 
\begin{equation}
\label{estimateLAMBDA}
\left(1+\omega_{M,K}\left(\tfrac{1}{n}\right)\right)\mu^c(x_0)\ge \lim_{k\rightarrow +\infty}\frac{1}{\theta_k}\int_{\tau Q} f(x_0,u(x_0),\tilde{v}_k(z),\theta_k \nabla\tilde{w}_k(z))\,dz +\Lambda(1-t),
\end{equation}
for some continuous function $\Lambda:[0,1]\rightarrow \mathbb{R}$ with $\Lambda(0)=0.$
Observe that, as in Step 3 (lower bound) of the $BV\times L^p $ case, $\eqref{limitxiandhatu} - \eqref{estimategradientxi}$ hold.  More precisely, we can apply the slicing method as in the proof of Step 3 (lower bound) for the case $BV \times L^p$, observing that \eqref{wbark} holds, with obvious adaptations.
For the reader's convenience, we observe that the construction of the fields $v_k$ is different from the $L^p$ case, here it is identical to the proof of Lemma \ref{3.1FMr}. We briefly recall that, for every $j \in \mathbb N$, we can divide $\tau Q\setminus \tau(1-\delta)Q$ into $j$ layers thus we getting a sequence $\{k(j)\}$, layers $S_j:=\{z\in \tau Q\setminus \tau(1-\delta)Q:\ \alpha_j<{\rm dist}(z,\partial (\tau Q))<\beta_j\}$ and cut-off functions $\eta_j$ on $\tau Q$ such that estimates analogous to \eqref{boundvj} hold with obvious adaptations. 
Then, define $\tilde{w}_j$ as in \eqref{wtildej}, and 
$$
\tilde{v}_j(z):=(1-\eta_j(z))\tfrac{\displaystyle{\frac{1}{|\tau Q|}\int_{\tau Q}\alpha(x)-\eta_j(x)\bar{v}_{k(j)}(x_0+\varepsilon_{k(j)}x)\,dx}}{\displaystyle{\frac{1}{|\tau Q|}\int_{\tau Q}(1-\eta_j(x))\,dx}}+\eta_j(z)\bar{v}_{k(j)}(x_0+\varepsilon_{k(j)}z).
$$
Remark that $
\Vert\tilde{v}_j\Vert_{L^\infty}\le \Vert\alpha\Vert_{L^\infty}+1+\Vert\bar{v}_{k(j)}\Vert_{L^\infty}\le 2M+1,\, \, \, \frac{1}{|\tau Q|}\int_{\tau Q}\tilde{v}_j(z)\,dz=\frac{1}{|\tau Q|}\int_{\tau Q}\alpha(z)\,dz:=y_0(\tau),\ \text{for all}\,  j.$
By (\ref{eqwithmukinfty}), summing and subtracting $f(x_0,u(x_0),\tilde{v}_j(z),\theta_{k(j)}\nabla\tilde{w}_j(z))$  inside the integral, having in mind the definition of $\eta_j$ and using  $(H_1)_\infty$, we get
\begin{align*}
&(1+\omega(\tfrac{1}{n}))\,\mu^c(x_0)\ge\limsup_{j\rightarrow +\infty}\frac{1}{\theta_{k(j)}}\int_{\tau Q}f(x_0,u(x_0),\bar{v}_{k(j)}(x_0+\varepsilon_{k(j)} z),\theta_{k(j)}\nabla\bar{w}_{k(j)})\,dz\\ \notag
&\ge \limsup_{j\rightarrow +\infty}\frac{1}{\theta_{k(j)}}\left\{\int_{\tau Q}f(x_0,u(x_0),\tilde{v}_j,\theta_{k(j)}\nabla\tilde{w}_j)\,dz-\int_{x\in \tau Q:{\rm dist}(x,\partial(\tau Q))\le\beta_j}f(x_0,u(x_0),\tilde{v}_j,\theta_{k(j)}\nabla\tilde{w}_j)\,dz \right\}\\ \notag
&\ge \limsup_{j\rightarrow +\infty}\frac{1}{\theta_{k(j)}}\left\{\int_{\tau Q}f(x_0,u(x_0),\tilde{v}_j,\theta_{k(j)}\nabla\tilde{w}_j)\,dz-\int_{S_j}c(|\nabla\bar{w}_{k(j)}|+ |\nabla \eta_j|\,|\bar{w}_{k(j)}-\xi_{k(j)}|)\,dz\right.\\  \notag
&\,\, \, \,\,\left.-\int_{\tau Q\setminus\tau(1-\delta)Q}c(1+|\nabla \xi_{k(j)}|)\,dz \right\}\\  \notag
&\ge \limsup_{j\rightarrow +\infty}\frac{1}{\theta_{k(j)}}\left\{\int_{\tau Q}f(x_0,u(x_0),\tilde{v}_j,\theta_{k(j)}\nabla\tilde{w}_j)\,dz-\frac{c}{j}-\int_{\tau Q\setminus\tau(1-\delta)Q}c(1+|\nabla \xi_{k(j)}|)\,dz\right\}.
\end{align*}

\noindent By (\ref{estimategradientxi}) and (\ref{eq5Cantorlower}), $\int_{\tau Q\setminus\tau(1-\delta)Q}c(1+|\nabla \xi_{k(j)}|)\,dz\le \Lambda(1-t)$ for some continuous $\Lambda:[0,1]\rightarrow \mathbb{R}$ with $\Lambda(0)=0$. Therefore we have \eqref{estimateLAMBDA}, up to a relabelling of the sequences.

\noindent{\bf Step 4.} Analogously to Step 4 for the $BV\times L^p$ case, the functions $\tilde{v}_k$ have constant average in $\tau Q$, given by $y_0(\tau)$ and the functions $\xi_k$ satisfies the same periodicity properties. 

This, together with the fact that $\tilde{w}_j=\xi_{k(j)}$ on $\partial (\tau Q)$, yields that $\tilde{w}_j\in (\frac{\zeta_{k(j)}(\frac{\tau}{2})-\zeta_{k(j)}(-\frac{\tau}{2})}{\tau}\otimes e_N)\,x+W^{1,\infty}_{\rm per}(\tau Q;\mathbb{R}^d)$. Therefore
$$(1+\omega(\tfrac{1}{n}))\,\mu^c(x_0)\ge O(1-t)+\displaystyle{\limsup_{j\rightarrow +\infty}\tfrac{|\tau Q|}{\theta_{k(j)}}f(x_0,u(x_0),y_0(\tau),\tfrac{\zeta_k(\frac{\tau}{2})-\zeta_k(-\frac{\tau}{2})}{\tau}\otimes e_N)}.$$
If we add and subtract in the previous limit the quantity $\frac{|\tau Q|}{\theta_{k(j)}}f(x_0,u(x_0),y_0(\tau),\frac{\theta_{k(j)}}{|\tau Q|}A)$ we get two terms. One will raise the expected value of the $f^\infty$ function, namely
$$\displaystyle{\lim_{k \to +\infty}\tfrac{|\tau Q|}{\theta_{k(j)}}f(x_0,u(x_0),y_0(\tau),\tfrac{\theta_{k(j)}}{|\tau Q|}A)=
f^\infty(x_0,u(x_0),y_0(\tau),A)=f^\infty(x_0,u(x_0),0,A)},$$
the last identity following from Lemma 2.2 in \cite{FKP1} and recalling that $A$ is a rank-one matrix. The other term can be estimated using the Lipschitz continuity of $f(x_0, u(x_0),0, \cdot)$, i.e. \eqref{infty-lipschitzcontinuity} and (\ref{characthatu}).

Then, passing to the limit on $k$, and using (\ref{limitxiandhatu}), (\ref{eq5Cantorlower}) and (\ref{limitA}), we get
$$(1+\omega(\tfrac{1}{n}))\,\mu^c(x_0)\ge O(1-t)+f^\infty(x_0,u(x_0),0,A)+O(1-t).$$
Finally the desired estimate is obtained letting $\varepsilon\rightarrow 0^+$ and $t\rightarrow 1^-$.\end{proof}

\subsection{Upper bound in $BV\times L^\infty$}\label{ubbvlinfty} 
Let
$${\cal J}_\infty(u,v;A):=\inf\left\{\liminf_{n\to +\infty} J(u_n,v_n;A):\ u_n\in BV(\Omega;\mathbb{R}^d),\ v_n\in L^\infty(\Omega;\mathbb{R}^m),\ u_n\to u \hbox{ in } L^1, v_n \overset{\ast}{\rightharpoonup} v  \hbox{ in } L^\infty\right\}$$
for open sets $A\subset\Omega$ and for any $(u,v)\in BV(\Omega;\mathbb{R}^d)\times L^\infty(\Omega;\mathbb{R}^m)$,
where, $J(u,v;A)$ is as in \eqref{Jext} with $L^p$ replaced by $L^\infty.$ We observe that  $(H_1)_\infty$ implies that for every $u \in BV(\Omega;\mathbb{R}^d)$ and for every $v \in L^\infty(\Omega;\mathbb{R}^m)$,  with $\|v\|_{L^\infty}\leq M$ there exists $C_M>0$ such that $
{\cal J}_\infty(u,v;A)\le C_M\left(|A|+|Du|(A)\right).$ Moreover, ${\cal J}_\infty$ is a variational functional.

\begin{theorem}\label{upperboundL_inftythm}
Let $f:\Omega
\times\mathbb R^d\times\mathbb{R}^{m}\times\mathbb R ^{d\times
N}\rightarrow[ 0,+\infty) $ be a continuous function satisfying $(H_0)$, $( H_{1}) _\infty-( H_{3}) _\infty$, and ${\overline J}_\infty$ be defined in \eqref{relaxedp}. Then for every $(u,v)\in BV(\Omega;\mathbb R^d) \times L^{\infty}(\Omega ;\mathbb{%
R}^{m})$:
\begin{align}
\label{upperboundLinfty}
{\overline J}_\infty( u,v;\Omega)\leq\int_{\Omega}f(x,u,v,\nabla u) dx+\int_{J_{u}\cap\Omega}K_\infty( x,0,u^{+},u^{-},\nu_{u}) d\mathcal{H}^{N-1}+\int_{\Omega}f^{\infty}(x,u,0,\tfrac{dD^{c}u}{d\vert D^{c}u\vert }) d\vert
D^{c}u\vert.
\end{align}
 
\end{theorem}

\begin{proof}
The representation \eqref{upperboundLinfty} is achieved first for $u\in BV(\Omega;\mathbb R^d)\cap L^\infty(\Omega;\mathbb R^d)$, then, via an approximation argument as in \cite{AMT}, the result will be obtained in the whole space.

\smallskip 
\noindent {\bf Part 1}. Let $u\in BV(\Omega;\mathbb R^d)\cap L^\infty(\Omega;\mathbb R^d)$. As in Part 1 of Theorem \ref{upperboundL_pthm} it suffices to prove the equivalent of \eqref{absolutely-continuousUB}$-$\eqref{jump-densityUB} with $K_p$ and $f_p^\infty$ replaced by $K_\infty$ and $f^\infty$, respectively.

\noindent {\bf Bulk part.} It follows from \cite[Theorem 12]{RZ}.

\smallskip
\noindent{\bf Cantor part.}
We consider $u \in BV(\Omega;\mathbb R^d) \cap L^\infty(\Omega;\mathbb R^d)$ and $v \in L^\infty(\Omega;\mathbb R^m) $. Again we follow \cite{FM2} and \cite{FKP1}. As usual we identify $u$ with its approximate limit defined in $\Omega\setminus J_u$. Considering $u_n:=u\ast \varrho_n$, where $\{\varrho_n\}$ is a sequence of mollifiers, one has \eqref{untou}.
Recalling that $u$ is $|D^cu|-$measurable, $|Du|=|D^cu|+\eta$, where $\eta$ and $|D^cu|$ are mutually singular Radon measures, we consider  $x_0\in\Omega$ such that $\frac{d {\cal J}_\infty(u,v;\cdot)}{d |D^cu|}(x_0)$  exists and is finite, $\eqref{FMr5.13} - \eqref{FMP6.17new}$ hold, for every $1< p<\infty$, 
\begin{equation}\nonumber
\lim_{\e\to 0}\frac{1}{|D^cu|(B(x_0,\e))}\int_{B(x_0,\e)}f^\infty(x_0,u(x_0),0,A(x))d|D^cu|=f^\infty(x_0,u(x_0),0,A(x_0)),\end{equation}
and \eqref{FMr5.16} hold.
Fixing $\delta>0$ and arguing as in the Cantor part of Theorem \ref{upperboundL_pthm} we obtain
\begin{equation*}
{\cal J}_\infty(u,v; B(x_0,\varepsilon)) \leq 
\liminf_{\varepsilon \to 0^+}\liminf_{n\to +\infty}\frac{1}{|D^c u|(B(x_0,\varepsilon))}\int_{B(x_0,\varepsilon)} f(x_0, u(x_0),v,Du \ast \varrho_n )dx+ O(\delta).
\end{equation*}
Then the Cantor upper bound inequality is achieved as in the proof of \cite[(6.6)]{FKP1}.

\smallskip
\noindent{\bf Jump part.}
We claim that
\begin{equation*}
\displaystyle{{\cal J}_\infty(u,v; J_u\cap \Omega)\leq \int_{J_u\cap\Omega} K_\infty(x, 0, u^-, u^+, \nu_u)d {\cal H}^{N-1},}
\end{equation*}
for every $(u,v) \in BV(\Omega;\mathbb R^d)\cap L^\infty(\Omega;\mathbb R^d)\times L^\infty(\Omega;\mathbb R^m)$.  
\noindent The proof develops exploiting the arguments in \cite[Step 3 of Section 5]{FM2}, being divided into Cases $1$, $2$ and $3$ as in the $BV\times L^p$ case. Here we just present Case 1, since the others are entirely similar to the ones in Theorem \ref{upperboundL_pthm}. 

\noindent{\bf Case 1.} We start to consider $u:=c \chi_E + d (1-\chi_E)$, with ${\rm Per}(E;\Omega)<+\infty$, and $v \in L^\infty(\Omega;\mathbb R^m)$ and we aim to prove that
\begin{equation}
\label{5.20FM2infty}
\displaystyle{{\cal J}_\infty(u,v;A)\leq \int_{A}f(x,u,v,0)dx +\int_{J_u\cap A}K_\infty(x,0,c,d,\nu)d {\cal H}^{N-1}, \text{ for every } A \in \mathcal A(\Omega)}.
\end{equation}
This inequality is achieved in several steps, and we present just the main differences with the case $BV\times L^p$. 

\noindent{\bf Step 1.} First we assume that $u$ has a planar interface and we keep the same notations as in Step 1 of \eqref{jump-densityUB}. 
Suppose that $f$ does not depend on $x$, and we claim that there exist $\{u_n\}$ as in the proof of \eqref{jump-densityUB} 
and a sequence $\{v_n\} \subset L^\infty(a_0+\lambda Q_\nu;\mathbb R^m)$, such that $v_n(x)= v(x)$ if $|(x-a_0)\cdot \nu|> \frac{\lambda}{2(2n+1)}$, with 
$u_n \to u $ in $L^1(a_0+\lambda Q_\nu;\mathbb R^d)$, $v_n \overset{\ast}{\rightharpoonup} v$ in $L^\infty(a_0+\lambda Q_\nu;\mathbb R^m)$
and
\begin{equation}
\label{eqjump1infty}
\displaystyle{\lim_{n \to +\infty}\int_{a_0+ \lambda Q_\nu}f( u_n,v_n, \nabla u_n)dx = \int_{a_0+ \lambda Q_\nu} f( u,v,0)dx + \lambda^{N-1} K_\infty(0,c,d,\nu).}
\end{equation}

\noindent {\bf Step 1 a).} As in the proof of Theorem \ref{upperboundL_pthm} we claim that for all $\xi \in {\cal A}(c,d,e_N)$ and for all $\varphi \in L^\infty(Q;\mathbb R^m)$, with $\int_Q\varphi dx =0$, there exists $\xi_n \in {\cal A}(c,d,e_N)$ and 
$v_n\in L^\infty(Q;\mathbb R^m)$ such that $v_n(x)=v(x)$ if $|x_N| >\frac{1}{2(2n+1)}$ 
and \eqref{xinvn} hold replacing $\rightharpoonup$ in $L^p$ by $\overset{\ast}{\rightharpoonup}$ in $L^\infty$ and \eqref{eq2jumptoprove} hold with $f^\infty_p$ replaced by $f^\infty$.

Let $\Sigma$ be as in the proof of \eqref{eq2jumptoprove}.
For $k \in \mathbb N$, we label the elements of $(\mathbb Z\cap [-k,k]^N)\times \{0\}$ by $\{a_i\}_{i=1}^{{(2k+1)}^{N-1}}$ and we recall that
$
(2k+1)\overline{\Sigma}=\bigcup_{i=1}^{(2k+1)^{N-1}} (a_i+\overline{\Sigma}),
$
with
$
(a_i+ \Sigma)\cap (a_j+\Sigma)= \empty \hbox{ if }i \not= j.
$
We extend $\xi(\cdot, x_N)$ to $\mathbb R^{N-1}$ by periodicity and define $\{\xi_{2k+1}\}$ as in \eqref{xi2k+1}. Clearly $\xi_{2k+1}\in {\cal A}(c,d,e_N)$ and $\|\xi_{2k+1}-u\|_{L^1(Q;\mathbb R^d)}\to 0$ as $k \to +\infty$. Extending $\varphi(\cdot, x_N)$ to $\mathbb R^{N-1}$ by periodicity define
$$
v_{2k+1}(x):=\left\{
\begin{array}{ll}
v(x) &\hbox{ if }|x_N| > \frac{1}{2(2k+1)},\\
\varphi((2k+1)x )&\hbox{ if } |x_N|\leq \frac{1}{2(2k+1)}.
\end{array}
\right.
$$

\noindent As in Step 1 a) of Theorem \ref{upperboundL_pthm} we observe that $v_{2k+1} \overset{\ast}{\rightharpoonup} v$ in $L^\infty(Q;\mathbb R^m)$.

\noindent We can  argue as in the $BV\times L^p$ case, exploiting the periodicity of $\varphi$.

\begin{align*}
&\displaystyle{\int_Q f(\xi_{2k+1}, v_{2k+1},\nabla \xi_{2k+1})dx =\int_{\Sigma}\int_{-\frac{1}{2}}^{-\frac{1}{2(2k+1)}}f(c,v(x),0)dx+\int_{\Sigma}\int_{\frac{1}{2(2k+1)}}^\frac{1}{2}f(d, v(x),0)dx}\\
&\displaystyle{ + \int_{\Sigma}\int_{|x_N|<\frac{1}{2(2k+1)}} f(\xi((2k+1)x), \varphi((2k+1)x), (2k+1)\nabla \xi((2k+1)x)dx.}
\end{align*}

\noindent The first two integrals in the right hand side, converge as $k \to +\infty$, to $\int_{Q} f(u(x),v(x),0)dx.$
The latter integral, after a change of variables becomes
\begin{align*}
&\displaystyle{\int_{\Sigma}\int_{|x_N|<\frac{1}{2(2k+1)}} f(\xi((2k+1)x), \varphi((2k+1)x), (2k+1)\nabla \xi((2k+1)x))dx}\\
&\displaystyle{= \frac{1}{2k+1}\int_Q f(\xi(y), \varphi(y), (2k+1)\nabla \xi(y))dy \underset{k\to +\infty}{\to} \int_{Q}f^{\infty}(\xi(y), \varphi(y), \nabla \xi(y))dy.}
\end{align*}
From the last two convergence we obtain the desired limit.

\noindent {\bf Step 1 b).} To prove \eqref{eqjump1infty},
let $\{(\eta_n, \varphi_n )\} \subset {\cal A}(c,d, e_N)\times L^\infty(Q;\mathbb R^m)$ with $\int_Q \varphi_n dy=0$ be a minimizing sequence for $K_\infty(c,d, 0,\e_N)$. 

By Step 1 a), for every $n \in \mathbb N$ we can find $k_n\in N$ and choose $u_{n} \in {\cal A}(c,d, e_N)$ and $v_{n} \in L^\infty(Q;\mathbb R^m)$ such that $
\|u_n- u\|_{L^1(Q;\mathbb R^d)}\leq \frac{1}{n}$, $\left|\int_Q (v_n- v)\psi_l dx\right|<\frac{1}{n}$, (for $l=1,\dots,n$ with $\{\psi_l\}$ a dense sequence of functions in $L^1(Q;\mathbb R^m)$), with $v_n$ defined as 
$$
v_{n}(x):= 
\left\{
\begin{array}{ll}
v(x)  &\hbox{ if } |x_N|> \frac{1}{2(2k_n+1)},\\
\varphi_n((2k_n+1)x) &\hbox{ if } |x_N| \leq \frac{1}{2(2k_n+1)}
\end{array}
\right.
$$
and
$$
\displaystyle{\left|\int_Q f(u_n, v_n, \nabla u_n)dx- \int_Q f(u, v,0)dx - \int_Q f^\infty(\eta_n, \varphi_n, \nabla \eta_n)dx\right|<\frac{1}{n}}.
$$
By the lower bound inequality we have that
\begin{equation}
\label{lambda1a00infty}
\begin{array}{ll}
\displaystyle{\int_{Q} f(u,v,0)dx + K_\infty(0,c,d,e_N)\leq \liminf_{n \to +\infty}\int_Q f(u_n,v_n, \nabla u_n)dx \leq\limsup_{n \to +\infty} \int_{Q}f(u_n, v_n, \nabla u_n)dx}\\
\displaystyle{ \leq \lim_{n\to +\infty} \left\{ \int_Q f(u,v,0)dx + \int_Q f^\infty (\eta_n, \varphi_n, \nabla \eta_n)dx+\frac{1}{n}\right\}=\int_Q f(u,v,0)dx + K_\infty(0,c,d,e_N),}
\end{array}
\end{equation}
which proves \eqref{eqjump1infty}, up to relabeling $\{u_n\}$ and $\{v_n\}$ with the same indices $k_n$, when $\lambda=1$ and $a_0=0$.

Considering the case of $A:=\lambda Q$, for $\lambda >0$, we define $f_\lambda, u_0$ and $v_0$ as in \eqref{u0lambda}. By \eqref{lambda1a00infty} there exists $(u_n,v_n) \in {\cal A}(c,d,e_N)\times L^\infty(Q;\mathbb R^m)$ such that $u_n \to u_0$ in $ L^1(Q;\mathbb R^d)$, $v_n \overset{\ast}{\rightharpoonup} v_0$ in $L^\infty(Q;\mathbb R^m)$ and \eqref{flambdatoKplambda} with $(K_p)_\lambda$ replaced by $(K_\infty)_\lambda$, where $(K_\infty)_\lambda$ is the function defined in \eqref{Kinfty}, with $f$ replaced by $f_\lambda$. 
Consider any $a_0 \in \mathbb R^N$ and set $\bar{u}_n$ and $\bar{v}_n$
as in \eqref{unlambda}. 

Clearly $\{\bar u_n\}$ satisfies all the properties stated in the proof of Theorem \ref{upperboundL_pthm}, Step 1 b), ${\bar v}_n \overset{\ast}{\rightharpoonup} v$ in $L^\infty(a_0+\lambda Q)$
and \eqref{fnto} holds with $(K_p)_\lambda$ replaced by $(K_\infty)_\lambda$.
Moroever, \eqref{Kplambda} holds with the obvious adaptation $(K_\infty)_\lambda=\tfrac{1}{\lambda}K_\infty$.
Hence we conclude that \eqref{4.9BFinfty} holds.

\noindent{\bf Step 1 c).} We allow $f$ to have explicit $x$-dependence and, given $r\geq 0$, we prove \eqref{5.20FM2infty} with $K_\infty$ replaced by $K_r$ as in Proposition \ref{PropKr}.

Let $A, A^\ast, Q_\nu, A'$ and $Q'$ as in Theorem \ref{upperboundL_pthm},  Step 1 c).  Since $A^\ast$ is a compactly included in $\Omega$, fixing  $\varepsilon >0$, it is possible to find a $\delta>0$ such that $(H_2)_\infty$ and \eqref{ucKr} hold uniformly in $A^\ast$, i.e.
\begin{equation}
\label{4.3BFinfty}
\displaystyle{x,y \in A^\ast, |x-y|<\delta \Rightarrow |f(x,u,b, \xi) -f(y,u,b,\xi)|\leq \varepsilon C_M(1+|\xi|),}
\end{equation}
for any $b \in \mathbb R^m$,
and
\begin{equation}
\label{4.4BFinfty}
\displaystyle{x,y \in A^\ast, |x-y|<\delta \Rightarrow |K_r(x,b,c,d, \nu)- K_r(y,b,c,d,\nu)|\leq \varepsilon C'_{|b|+r}(1+|d-c|).}
\end{equation}
Let $h\in\mathbb N$ be as in \eqref{4.5BF}, partition $\Omega'$ according to \eqref{4.6}
and denote $Q'_i:=a_i +\eta Q'$ and $Q_i:=a_i+\eta Q$.

We claim that there exists $\{u_k\} \subset W^{1,1}(A^\ast;\mathbb R^d)$ and a sequence $\{v_k\} \subset L^\infty(A^\ast;\mathbb R^m)$ such that $u_k \to u$ in $L^{1}(A;\mathbb R^d)$, $v_k \overset{\ast}{\rightharpoonup} v$ in $L^\infty(A;\mathbb R^m)$ and 
\begin{equation*}
\displaystyle{\lim_{k \to +\infty}\int_{A^\ast}f(x, u_k, v_k,\nabla u_k)dx \leq \int_{J_u \cap A^\ast}K_r(x,0,c,d,e_N)d{\cal H}^{N-1}+ \int_{A^\ast}f(x,u, v,0)dx.}
\end{equation*}
By Step 1 b), there exist sequences $\{u_k^{(1)}\}\subset {\cal A}(c,d,e_N)$, related to the cube $Q_1$ and $\{v_k^{(1)}\}\in L^\infty(Q_1;\mathbb R^m)$, such that \eqref{4.7BF} holds with $K_p$ replaced by $K_\infty$, thus, by \eqref{KKr}
$$
\displaystyle{\lim_{k\to +\infty}\int_{Q_1}f(a_1, u_k^{(1)}, v_k^{(1)},\nabla u_k^{(1)})dx \leq \eta^{N-1}K_r(a_1,0,c,d,e_N)+\int_{Q_1}f(a_1, u, v,0)dx.}
$$
By iii) in Remark \ref{3.17FMrrem} there exists $\{\xi_k^{(1)}\}\subset W^{1,1}(Q_1;\mathbb R^d)$ and a sequence $\{\overline{v}_k^{(1)}\} \subset L^\infty(Q_1;\mathbb R^m)$ such that $\xi_k^{(1)}\to u$ in $L^1(Q_1;\mathbb R^d) $, with $\xi_k^{(1)}(x)= U_k^{(1)}((x-a_1)/\eta)$ on $\partial Q_1$, ($U_k^{(1)}$ is a mollification of $u$) and $\overline{v}_k^{(1)}\overset{\ast}{\rightharpoonup} v$ in $L^\infty(Q_1;\mathbb R^m)$, and
\begin{equation*}
\begin{array}{ll}
\displaystyle{\limsup_{k\to +\infty}\int_{Q_1}f(a_1, \xi_k^{(1)}, \overline{v}_k^{(1)},\nabla \xi_k^{(1)})dx} &\displaystyle{\leq\liminf_{k \to +\infty}\int_{Q_1}f(a_1, u_k^{(1)}, v_k^{(1)},\nabla u_k^{(1)})dx}\\
&\displaystyle{\leq\eta^{N-1}K_r(a_1,0,c,d,e_N)+\int_{Q_1}f(a_1, u, v,0)dx.}
\end{array}
\end{equation*}
We can repeat the same induction argument as in Step 1 c) in Theorem \ref{upperboundL_pthm},  
obtaining $h^{N-1}$ sequences $\{\xi^{(j)}_k\}\subset {\cal A}(c,d, e_N)$ related to the cube $Q_j$, with $\xi^{(j)}_k \to u$ in $L^1(Q_j;\mathbb R^d)$, $\xi^{(j)}_k =U^{(j)}_k$ on $\partial Q_j$ and $\{\overline{v}^{(j)}_k\}\subset L^\infty(Q_j;\mathbb R^m)$, with $\overline{v}^{(j)}_k  \overset{\ast}{\rightharpoonup} v$ in $L^\infty(Q_j;\mathbb R^m)$ and
$$
\displaystyle{\eta^{N-1}K_r(a_j,0,c,d,e_N)+\int_{Q_j}f(a_j, u, v,0)dx\geq \lim_{k\to +\infty}\int_{Q_j}f(a_j, \xi_k^{(j)}, \overline{v}_k^{(j)},\nabla \xi_k^{(j)})dx}
$$
for every $j=1,\dots, h^{N-1}.$

Next for all $j=1,\dots, h^{N-1}$ we consider the subsequences $\{\zeta_k\}$, with $\zeta_k:=\xi^{(j)}_k$ and $\{\tilde{v}_k\}$, with $\tilde{v}_k:=\overline{v}^{(j)}_k$, denoted by the same index $k$ such that
\begin{equation}\label{4.9BFinfty}
\begin{array}{ll}
\displaystyle{\lim_{k\to +\infty}\int_{Q_j}f(a_j, \zeta_k, \tilde{v}_k,\nabla \zeta_k)dx \leq\eta^{N-1}K_r(a_j,0,c,d,e_N)+\int_{Q_j}f(a_j, u, v,0)dx.}
\end{array}
\end{equation}
Define the sequences $\{u_{k,\varepsilon}\}$ and $\{v_{k,\varepsilon}\}$ almost everywhere on $A^\ast$, as in \eqref{ukepsilon}, thus \eqref{ukeconv} holds.
 
Since  $\{\tilde{v}_k \}\subset L^\infty(A^\ast;\mathbb R^m)$ and coincides with $v(x)$ if $|x_N|>\eta/2$,  $v_{k,\varepsilon }\overset{\ast}{\rightharpoonup} v$ in $L^\infty(A^\ast;\mathbb R^m)$ as $k \to +\infty$ and as $\varepsilon \to 0^+$.
 Moreover, we can write, as  in \eqref{I1234}
 $$
 \displaystyle{\int_{\Omega}f(x,u_{k,\varepsilon}, v_{k,\varepsilon}, \nabla u_{k,\varepsilon}) dx=:I_1+I_2+I_3+I_4.}
$$ 
 Then it is easily seen that by \eqref{4.9BFinfty}, we have
 $$
 \displaystyle{\limsup_{k \to +\infty}I_1 \leq \sum_{i=1}^{h^{N-1}}\left( \eta^{N-1} K_r(a_i, 0,c,d, e_N)+\int_{Q_i}f(a_i,u,v,0)dx\right).}
 $$
 On the other hand, the asymptotic behaviour of $I_3+I_4$ is given by \eqref{I34}. Regarding $I_2$ we can observe that, by Remark \ref{Kre} and \eqref{4.3BFinfty},
 $$
 \begin{array}{ll}
 \displaystyle{\limsup_{k \to ì\infty}I_2 \leq \limsup_{k \to +\infty} \sum_{i=1}^{h^{N-1}}\int_{Q_i}\varepsilon C_{|b|+r}(1+ |\nabla \zeta_k|)dx = O(\varepsilon), }
 \end{array}
 $$
 since, by construction and $(H_1)_\infty$,  $\{\tilde{v}_k^{(i)}\}$ and $\{\zeta_k^{(i)}\}$ are bounded in $L^\infty(Q_i;\mathbb R^m)$ and $W^{1,1}(Q_i;\mathbb R^d)$, respectively. By Remark \ref{Kre} and \eqref{4.4BFinfty}, 
 \begin{align*}
 \displaystyle{\left| \int_{A^\ast \cap J_u} K_r(x, 0, c, d, e_N)d {\cal H}^{N-1}-\eta^{N-1}\sum_{i=1}^{h^{N-1}} K_r(a_i,0, c,d, e_N)\right|}\\
= \displaystyle{\sum_{i=1}^{h^{N-1}} \int_{Q_i}\left|K_r(x,0, c,d, e_N)-K_r(a_i,0,c,d, e_N)\right| d {\cal H}^{N-1}=O(\varepsilon).}
 \end{align*} 
Finally, putting together this estimate, the limits of $I_2, I_3, I_4$ and estimating $\sum_{i=1}^{h^{N-1}}\int_{Q_i}f(a_i,u,v,0)dx$ in $I_1$ via \eqref{4.3BFinfty}, we obtain the desired approximating sequence, just letting $\varepsilon \to 0^+$ and using a diagonalization procedure. In fact, we can say that there exist  $\{\overline{\zeta}_k\} \subset L^1(A^\ast;\mathbb R^d)$ and $\{\overline{v}_k\} \subset L^\infty(A^\ast;\mathbb R^m)$, converging to $u$ in $L^1$ and  weakly $\ast$ in $L^\infty$ to $v$, respectively, such that
$$
\begin{array}{ll}
\displaystyle{\lim_{k\to +\infty}\int_{A^\ast}f(x,\overline{\zeta}_k,\overline{v}_k,\nabla \overline{\zeta}_k)dx \leq\int_{J_u\cap A^\ast}K_r(x,0,c,d,\nu)d{\cal H}^{N-1}+\int_{A^\ast}f(x,u(x),v(x),0)dx.}
\end{array}
$$
Hence \eqref{5.20FM2infty} follows by \eqref{KKr} sending $r \to +\infty$. From the lower bound, the equality is achieved.

\noindent {\bf Step 1 d).} In order to consider $A^\ast$ any open subset of $\Omega$ such that \eqref{4.10BF} holds, it suffices to argue exactly as in Theorem \ref{upperboundL_pthm}, Step 1 d), replacing $K_p$ by $K_\infty$, weak convergence in $L^p$ by weak $\ast$ convergence in $L^\infty$ and invoking iii) in Remark \ref{3.17FMrrem}.

\noindent{\bf Step 2. } 
In order to obtain the representation
$$
{\cal J}_\infty(u,v;E)=\int_{E}K_\infty(x,0,c,d,\nu)dx,$$
when $u:= c\chi_E+ d(1-\chi_E) $, $E$ being a set of finite perimeter with unit normal $\nu$, we can argue as in Steps 2, 3 and 4 of Theorem \ref{upperboundL_pthm}, replacing the densities $f^\infty_p$ and $K_p$ by $f^\infty$ and $K_\infty$, respectively, and hypotheses $(H_1)_p-(H_3)_p$ by $(H_1)_\infty-(H_3)_\infty$. 

\noindent{\bf Part 2.} If $(u,v)\in BV\times L^\infty$ the proof is identical to the one of Part 2 in Theorem \ref{upperboundL_pthm}. It is enough to replace $K_p$ and $f^\infty_p$ therein by $K_\infty$ and $f^\infty$, respectively, and invoking the correspondent properties.\end{proof}

\medskip

\begin{proof}[Proof of Theorem \ref{MainResultinfty}] 
The proof follows by Theorems \ref{lowerboundL_inftythm} and \ref{upperboundL_inftythm}.
\end{proof}
\section{Appendix}\label{Appendix}

The following theorem is devoted to remove assumption $(H_0)$ in Theorem \ref{MainResultp}.

\begin{theorem}
\label{MainResultpnoqcx} Let $J$ be given by (\ref{functional}), with $f$
verifying  $(H_1)_p-(H_3)_p$, $1<p<\infty$, and let ${\overline J}_p$ be given by %
\eqref{relaxedp} then 
\begin{multline}
\label{reppapp}
\overline{J}_p(u,v)= \int_\Omega {\cal CQ }f(x,u,v, \nabla u) dx +
\int_{J_u} K_p(x,0, u^+, u^-,\nu_u)d\mathcal{H}^{N-1} +
\int_\Omega {\cal CQ }f_p^\infty\left(x,u,0,\tfrac{dD^cu}{d|D^c u|}\right) d|D^c u|,
\end{multline}
for every $(u,v)\in BV(\Omega;\mathbb{R}^d)\times L^p(\Omega;\mathbb{R}^m)$,
where $\mathcal{CQ}f$ denotes the convex-quasiconvex envelope of $f$, given by
\begin{equation}\nonumber
\mathcal{CQ}f(x,u, b,\xi):=\inf\left\{\frac{1}{|D|}\int_{D}f(x,u, b+\eta,\xi+\nabla \varphi)\,dy:
\eta\in L^{\infty}( D;\mathbb{R}^{m}), \int_{D}\eta dy=0,\varphi\in W_{0}^{1,\infty}(D;\mathbb{R}^d) \right\},
\end{equation}
 and $K_p$ is given by \eqref{Kp} with $f_p^\infty$ replaced
by $\mathcal{CQ}f_p^\infty$, where
\begin{equation}  \nonumber
		\mathcal{CQ}f_p^\infty(x,u,b,\xi):=\limsup_{t \to + \infty}\frac{\mathcal{CQ}f(x,u,t^{\frac{1}{p}}b,t\xi)}{t}.
\end{equation}

\end{theorem}
\begin{proof}[Proof]
First we recall that the convex-quasiconvex envelope ${\cal CQ}f(x,u,\cdot,\cdot)$ of a function $f(x,u,\cdot,\cdot)$ is the largest convex-quasiconvex function below $f$.

In analogy with \eqref{functional} and \eqref{relaxedp} we define, for every $(u,v) \in W^{1,1}(\Omega;\mathbb{R}^d) \times L^p(\Omega ;\mathbb{R}%
^m)$, the functional 
\begin{equation}  \nonumber
J_{{\cal CQ} f}( u,v) :=\int_{\Omega}{\cal CQ }f(
x,u(x) ,v(x) ,\nabla u(x) dx,
\end{equation}
  and for every $(u,v) \in BV(\Omega;\mathbb{R}^d)
  \times L^{p}(\Omega;\mathbb{R}^{m}),$ the functional
\begin{equation}  \nonumber
\overline{J_{{\cal CQ}f_p}}(u,v) :=\inf\{ \underset{n\rightarrow +\infty}{\lim\inf}J_{{\cal CQ}f}(u_{n},v_{n}) :u_{n}\in
W^{1,1}(\Omega;\mathbb{R}^d) ,~v_{n}\in L^{p}( \Omega;%
\mathbb{R}^{m}) ,~u_{n}\rightarrow u\text{ in }L^{1},~v_{n}%
\rightharpoonup v\text{ in }L^{p}\},
\end{equation}
\noindent Clearly, it results that for every $( u,v) \in BV( \Omega;%
\mathbb R^d) \times L^p( \Omega;\mathbb R^m)$, and for every $1<p< \infty$,
\begin{equation}  \nonumber
\overline{J_{{\cal CQ}f}}_p( u,v)\leq {\overline J}_p(u,v).
\end{equation}
It remains to prove the opposite inequality.
First we observe that ${\cal CQ}f$ satisfies $(H_1)_p-(H_3)_p$. Regarding $(H_1)_p$ and $(H_2)_p$ it is enough to argue as in \cite[Proposition 2.2]{RZCh}. For what concerns $(H_3)_p$ we observe that it is equivalent to say that there exist $0<\tau< 1$ such that
\begin{equation}  \nonumber
|f^\infty_p(x,u,b,\xi)- f(x,u,b,\xi)|\leq C(1+|b|^{(1-\tau)p}+ |\xi|^{1-\tau})
\end{equation}
for every $(x,u,b,\xi) \in \Omega \times \mathbb{R}^d \times \mathbb{R}^m
\times \mathbb{R}^{d \times N}$, then, arguing as in \cite[Proposition 2.3]%
{RZCh} one can prove that $\mathcal{CQ}f_p^\infty$ inherits the same property.
Thus, applying the same arguments as in \cite[Lemma 8 and Remark 9]{RZ}, the proof is concluded.
\end{proof} 

We also observe that an argument entirely similar to \cite[Proposition 3.4]{BZZ}, guarantees that there is no ambiguity in \eqref{reppapp} when omitting the parentheses in $CQ(f^\infty_p)$, since  $(CQf)^\infty_p=CQ(f^\infty_p)$. 
When $p=\infty$ removing $(H_0)$, arguing as above, one can prove the following result.

\begin{theorem}
\label{MainResultpnoqcx} Let $J$ be given by (\ref{functional}), with $f$ and ${\cal CQ}f$
satisfying  $(H_1)_\infty-(H_3)_\infty$, and let ${\overline J}_\infty$ be given by %
\eqref{relaxedinfty} then
 \begin{equation}  \nonumber
 \begin{array}{ll}
 \displaystyle{{\overline J}_\infty(u,v)= \int_\Omega \mathcal{CQ} f(x,u,v,\nabla u) dx +
 \int_{J_u} K_{\infty}(x,0, u^+, u^-,\nu_u)d\mathcal{H}^{N-1}} \displaystyle%
 {+ \int_\Omega (\mathcal{CQ}f)^\infty\left(x,u,0,\tfrac{dD^cu}{d|D^c u|}\right)
 d|D^c u|}, & 
 \end{array}%
 \end{equation}
 where  $K_\infty$ is given by \eqref{Kinfty} with $f^\infty$ replaced
 by $(\mathcal{CQ}f)^\infty$.
\end{theorem}

\begin{remark}
\label{final}
Let $\alpha: [0,+\infty)\to [0,+\infty)$ be a convex and increasing function, with $\alpha(0)=0$ such that the following assumptions hold:

\noindent $(H_1)_\alpha $ There exists a constant $C>0$  such that 
$$\frac{1}{C}(\alpha(|b|)+ |\xi|)-C \leq f(x,u,b,\xi) \leq C(1+ \alpha(|b|)+
|\xi|)$$ for a.e. $(x,u)\in \Omega \times \mathbb{R}^d$ and for every $%
(b,\xi)\in \mathbb{R}^m \times \mathbb{R}^{d \times N}$.

\noindent $(H_2)_\alpha$ For every compact set $K \subset \Omega \times \mathbb{R}^d$ there
exists a continuous function $\omega^{\prime }_{K}:\mathbb{R }\to [0,
+\infty)$ such that $\omega^{\prime }_{K}(0)=0$ and 
\begin{equation}  \nonumber
|f(x,u,b,\xi)- f(x^{\prime },u^{\prime },b,\xi)| \leq \omega^{\prime
}_{K}(|x-x^{\prime }|+|u-u^{\prime }|)(1 + \alpha(|b|)+ |\xi|),\ 
\end{equation}
for every $(x,u),(x^{\prime },u^{\prime })\in K,\ \forall\ (b,\xi) \in \mathbb{%
R}^m \times \mathbb{R}^{d \times N}.$

\noindent For each $x_0 \in \Omega$ and $\varepsilon >0$, there exists $%
\delta>0$ such that 
\begin{equation}  \nonumber
|x-x_0|\leq \delta\ \Rightarrow\ f(x,u,b,\xi) - f(x_0,u,b,\xi) \geq
-\varepsilon(1 +\alpha(|b|) +|\xi|),
\end{equation}
for every $(u,b,\xi)\in \mathbb{R}^d \times \mathbb{R}^m \times \mathbb{R}^{d\times N}.$

\noindent $(H_3)_\alpha$ There exist $C>0$ and $0<\tau<1$ such that 
$$ |f(x,u,b,\xi)- f^\infty(x,u,b,\xi)|\leq C(1+\alpha^{1-\tau}(|b|)+|\xi|^{1-\tau}),
$$
for every $(x,u,b,\xi)\in \Omega \times \mathbb R^d \times \mathbb R^m\times \mathbb R^{d\times N}$.

We observe that if one replaces $(H_1)_\infty-(H_3)_\infty$ by $(H_1)_\alpha- (H_3)_\alpha$ then, arguing as in \cite[Propositions 2.1, 2.2 and 2.3]{RZCh} 
$$
{\cal CQ}(f^\infty)(x,u,b,\xi)=({\cal CQ}f)^\infty(x,u,b,\xi),
$$
for every $(x,u,b,\xi)\in \Omega \times \mathbb R^d \times \mathbb R^m\times \mathbb R^{d\times N}$, and ${\cal CQ}f$ satisfies $(H_1)_\alpha-(H_3)_\alpha$. 

Thus an analogous of Theorem \ref{MainResultpnoqcx} holds with obvious modifications just imposing $(H_1)_\alpha-(H_3)_\alpha$ on $f$.

\end{remark}

\noindent{\bf Acknowledgments.} 
This paper has been written during various visits of the authors at Departamento de Matem\'atica da Universidade
de \'Evora, at Dipartimento di Ingegneria Industriale dell' Universit\'a di Salerno and Carnegie Mellon University, whose kind
hospitality and support are gratefully acknowledged.
The authors are indebted to Irene Fonseca and Ana Margarida Ribeiro for the many discussions
on the subject.
The work of both authors was partially supported by Funda\c{c}\~{a}o para a Ci\^encia e a Tecnologia (Portuguese
Foundation for Science and Technology) through CIMA-UE, UTA-CMU/MAT/0005/2009.
They acknowledge the support of GNAMPA through the programs 'Professori Visitatori 2015' and through the project `Un approccio variazionale alla stabilit\`a di sistemi di reazione-diffusione non lineari' 2015.
The second author is member of the Gruppo Nazionale per l'Analisi Matematica, la Probabilit\'a e le loro Applicazioni (GNAMPA) of the Istituto Nazionale di Alta Matematica (INdAM).

\end{document}